\documentclass[12pt]{amsart}
\date{\today}
\usepackage{latexsym,amsmath,amsfonts,amscd,amssymb,lmodern,mathrsfs}
\usepackage{stmaryrd}
\usepackage{lscape}
\usepackage{mathdots}
\usepackage{array}
\usepackage{pdfsync}
\usepackage[all]{xy}
\usepackage{graphicx}
\usepackage[english]{babel}
\usepackage[mathscr]{eucal}
\usepackage{amsthm} 
\usepackage{xfrac}  
\usepackage{multirow} 
\usepackage{setspace}

\newcommand{\curly}{\mathscr}

\theoremstyle{plain}  
\newtheorem{theorem}{Theorem}[section]

\newtheorem*{theorem*}{Theorem}

\newtheorem{lemma}[theorem]{Lemma}
\newtheorem{proposition}[theorem]{Proposition}

\newtheorem{definition}[theorem]{Definition}
\theoremstyle{remark}
\newtheorem{example}[theorem]{Example}
\newtheorem{remark}[theorem]{Remark}

\newtheorem*{claim*}{Claim}
\numberwithin{equation}{section}

\renewcommand{\leq}{\leqslant}
\renewcommand{\leq}{\leqslant}
\renewcommand{\geq}{\geqslant}
\renewcommand{\geq}{\geqslant}
\renewcommand{\setminus}{\smallsetminus}

\newcommand{\setC}{\mathbb{C}}
\newcommand{\setR}{\mathbb{R}}
\newcommand{\setZ}{\mathbb{Z}}

\newcommand{\eA}{{\mathscr{A}}}

\newcommand{\R}{\mathbb{R}}

\newcommand{\Z}{\mathbb{Z}}
\newcommand{\C}{\mathbb{C}}

\newcommand{\RR}{\mathbb{R}}

\newcommand{\ZZ}{\mathbb{Z}}
\newcommand{\CC}{\mathbb{C}}

\newcommand{\be}{\mathbf{e}}
\newcommand{\bE}{\mathbf{E}}

\newcommand{\AAA}{\curly{A}}
\newcommand{\CCC}{\curly{C}}

\newcommand{\HHH}{\curly{H}}
\newcommand{\LLL}{\curly{L}}

\newcommand{\NNN}{\curly{N}}

\newcommand{\PP}{\mathbb{P}}

\newcommand{\cH}{\mathcal{H}}
\newcommand{\cM}{\mathcal{M}}
\newcommand{\cO}{\mathcal{O}}
\newcommand{\cR}{\mathcal{R}}
\newcommand{\cS}{\mathcal{S}}
\newcommand{\dbar}{\bar{\partial}}
\newcommand{\zbar}{\bar{z}}
\newcommand{\lie}{\mathfrak}
\newcommand{\ra}{\to}
\newcommand{\lra}{\longrightarrow}

\newcommand{\PSL}{\mathrm{PSL}}

\newcommand{\SU}{\mathrm{SU}}
\newcommand{\U}{\mathrm{U}}

\newcommand{\GL}{\mathrm{GL}}
\newcommand{\Aut}{\mathrm{Aut}}
\newcommand{\Int}{\mathrm{Int}}
\newcommand{\SL}{\mathrm{SL}}

\DeclareMathOperator{\pardeg}{pardeg}

\DeclareMathOperator{\ad}{ad}
\DeclareMathOperator{\Ad}{Ad}

\DeclareMathOperator{\rk}{rk}
\DeclareMathOperator{\rank}{rank}

\DeclareMathOperator{\Hom}{Hom}
\DeclareMathOperator{\End}{End}

\DeclareMathOperator{\Id}{Id}

\DeclareMathOperator{\Res}{Res}

\DeclareMathOperator{\Stab}{Stab}
\DeclareMathOperator{\aff}{aff}
\DeclareMathOperator{\cochar}{cochar}
\DeclareMathOperator{\coroot}{coroot}
\DeclareMathOperator{\Conj}{Conj}

\newcommand{\Gr}{\operatorname{Gr}}

\newcommand{\liea}{\mathfrak{a}}

\newcommand{\lieg}{\mathfrak{g}}
\newcommand{\lieh}{\mathfrak{h}}
\newcommand{\liel}{\mathfrak{l}}
\newcommand{\liem}{\mathfrak{m}}
\newcommand{\lien}{\mathfrak{n}}
\newcommand{\liep}{\mathfrak{p}}
\newcommand{\lieq}{\mathfrak{q}}
\newcommand{\lier}{\mathfrak{r}}
\newcommand{\lies}{\mathfrak{s}}
\newcommand{\liesl}{\mathfrak{sl}}
\newcommand{\liesu}{\mathfrak{su}}
\newcommand{\liet}{\mathfrak{t}}
\newcommand{\lieu}{\mathfrak{u}}
\newcommand{\liez}{\mathfrak{z}}

\newcommand{\liemc}{\mathfrak{m}^{\mathbb{C}}}

\newcommand{\liehc}{\mathfrak{h}^{\mathbb{C}}}

\newcommand{\liegc}{\mathfrak{g}^{\mathbb{C}}}

\newcommand{\alie}{\mathfrak{a}}

\newcommand{\glie}{\mathfrak{g}}
\newcommand{\hlie}{\mathfrak{h}}
\newcommand{\llie}{\mathfrak{l}}
\newcommand{\mlie}{\mathfrak{m}}

\newcommand{\plie}{\mathfrak{p}}

\newcommand{\tlie}{\mathfrak{t}}

\newcommand{\zlie}{\mathfrak{z}}

\newcommand{\HC}{H^\setC}

\newcommand{\la}{\langle}
\renewcommand{\ra}{\rangle}
\newcommand{\imag}{\sqrt{-1}}

\newcommand{\ov}{\overline}

\begin{document}

\title[Parabolic Higgs bundles] {Parabolic Higgs bundles and
  representations of the fundamental group of a punctured surface into a
  real group}

\author[O.  Biquard]{Olivier Biquard}
\address{Sorbonne Universit\'e  and \'Ecole Normale Sup\'erieure, UMR 8553 du CNRS}

\author[O. Garc\'{\i}a-Prada]{Oscar Garc\'{\i}a-Prada}
\address{Instituto de Ciencias Matem\'aticas, CSIC-UAM-UC3M-UCM, Madrid}

\author[I. Mundet i Riera]{Ignasi Mundet i Riera}
\address{Universitat de Barcelona}

\begin{abstract}
  We study parabolic $G$-Higgs bundles over a compact Riemann surface
  with fixed punctures, when $G$ is a real reductive Lie group, and
  establish a correspondence between these objects and representations
  of the fundamental group of the punctured surface in $G$ with
  arbitrary holonomy around the punctures. This generalizes Simpson's
results for $\GL(n,\C)$ to arbitrary complex  and real reductive Lie groups.
 Three interesting features
  are the relation between the parabolic degree and the Tits geometry of
  the  boundary at infinity of the symmetric space, the treatment of the case
when the logarithm of the monodromy is on the boundary of a Weyl alcove, and the
  correspondence of the orbits encoding the singularity via the
  Kostant--Sekiguchi correspondence. We also describe some special features of
  the moduli spaces when $G$ is a split real form or a group of Hermitian type.
\end{abstract}

\maketitle
























\begin{flushright}
{\it Dedicated with admiration and gratitude to Narasimhan and Seshadri \\
in the fiftieth anniversary of their theorem} \\ \vspace*{0.25cm}
\end{flushright}

\section{Introduction}

The relation between representations of the fundamental group of a compact
Riemann surface $X$ into a compact Lie group and holomorphic bundles on $X$
goes back to the celebrated theorem of Narasimhan and Seshadri
\cite{NS}, which implies that the moduli space of irreducible representations
of $\pi_1(X)$ in the unitary
group $\U_n$ and the moduli space of rank $n$ and zero degree stable
holomorphic  vector bundles on $X$ are homeomorphic. Of course, this generalises the classical
case of representations in $\U_1=S^1$ and their relation with the Jacobian of $X$.
The Narasimhan--Seshadri  theorem has been  a paradigm and an inspiration for
more than 50 years now  for many similar problems.  The theorem  was
generalised by  Ramanathan \cite{ramanathan}  to representations into any
compact Lie group  \cite{ramanathan}.  The gauge-theoretic point of view of
Atiyah and Bott \cite{atiyah-bott}, and the new proof of the
Narasimhan--Seshadri theorem given by Donaldson following this approach
\cite{donaldson-ns}, brought new insight and new analytic tools into the
problem.

The case of representations into a non-compact reductive Lie group $G$
required the introduction of new holomorphic objects on the Riemann surface
$X$  called
$G$-Higgs bundles. These were introduced by Hitchin
\cite{hitchin,hitchin-duke}, who established a homeomorphism between the
moduli space of reductive representation in $\SL_2\C$ and polystable
$\SL_2\C$-Higgs bundles. This correspondence was generalised by Simpson
to any complex reductive Lie group (and in fact, to higher dimensional
K\"ahler manifolds) \cite{Sim88, Sim92}. The correspondence in the case of
non-compact $G$
needed an extra ingredient --- not present in the compact case --- having to
do with the existence of twisted harmonic maps into the symmetric space
defined by $G$. This theorem was provided by Donaldson for $G=\SL_2\C$
\cite{donaldson-twisted} and by Corlette \cite{corlette} for arbitrary $G$.
In fact Corlette's theorem, which  holds for any reductive real Lie group,
can be  combined with an existence theorem for solutions to the Hitchin
equations  for a $G$-Higgs bundle, given by the second and third authors in
collaboration with Bradlow and Gothen
\cite{bradlow-garcia-prada-mundet:2003,garcia-prada-gothen-mundet:2009a},
to prove the correspondence
for any real reductive Lie group $G$. In \cite{Sim92} Simpson gives an indirect
proof of this by embedding $G$ in its complexification when such embedding exists.

There is another direction in which the Narasimhan--Seshadri
theorem has been generalised. This is by allowing punctures in
the Riemann surface. Here one is interested in studying
representations of the fundamental group of the punctured
surface with fixed holonomy around the punctures. These
representations now relate to the parabolic vector bundles
introduced by Seshadri \cite{seshadri}. The correspondence in
this case for $G=\U_n$ was carried out by Mehta and Seshadri
\cite{mehta-seshadri}. A differential geometric proof modelled
on that of Donaldson for the parabolic case was given by the
first author in \cite{biquard-thesis}.  The case of a general
compact Lie group is studied in
\cite{bhosle-ramanathan,teleman-woodward,balaji-seshadri,balaji-biswas-pandey} under
suitable conditions on the holonomy around the punctures. One
of the main issues for general $G$ is about the appropriate
generalisation of parabolic principal bundles.

The non-compactness in the group and in the surface can be combined to
study representations of the fundamental group of a punctured surface into
a non-compact reductive Lie group $G$. Simpson considered this situation when
$G=\GL_n\C$ in \cite{Sim90}. A new ingredient in his work is the study of
filtered local systems. The aim of this paper is to extend this
correspondence to the case of an arbitrary real  reductive Lie group $G$
(including the case in which $G$ is complex). We establish a one-to-one
correspondence between reductive representations of the fundamental
group of a punctured surface $X$ with fixed arbitrary holonomy around the
punctures and polystable parabolic $G$-Higgs bundles on $X$.

One of the main technical issues to prove our correspondence,
already present in
\cite{bhosle-ramanathan,teleman-woodward,balaji-seshadri}, lies
in the definition of parabolic principal bundles. If $G$ is a
non-compact reductive Lie group and $H\subset G$ is a maximal
compact subgroup, we need to define parabolic $H^\C$-bundles.
This involves a choice for each puncture of an element in a
Weyl alcove of $H$ --- the weights.  If the element is in the
interior of the alcove everything goes smoothly, but if the
element is in a wall of the alcove, its adjoint may have
eigenvalues with modulus equal to $1$ (as opposed to the
elements in the interior, whose eigenvalues has modulus
strictly smaller than $1$), and this introduces complications
in the definition of the objects, as well as in the analysis to
prove our existence theorems. However, we give a suitable
definition of parabolic $G$-Higgs bundle including the case in
which the elements are in a `bad' wall of the alcove, which is
appropriate to carry on the analysis and to prove the
correspondence with representations. Of course the need of
including elements in the walls of the Weyl alcove is
determined by the fact that we want to have totally arbitrary
fixed holonomy (conjugacy classes) around the punctures.  Our
approach for the bad weights is rather pedestrian, using
holomorphic bundles and gauge transformations between them
which can have meromorphic singularities, so that to a
representation corresponds not a single holomorphic bundle, but
rather a class of holomorphic bundles equivalent under such
meromorphic transformations. A more formal algebraic point of
view is that of  parahoric torsors
\cite{balaji-seshadri,boalch,heinloth}, but we preferred to
stick to a more concrete definition which is sufficient to
state completely the correspondence (see Section
\ref{s:parahoric} for a comparison between the two points of
view). Our definition is also more natural from the
differential geometry viewpoint which sees holomorphic bundles
only outside the punctures, and defines the sheaves of sections
at the singular points by growth conditions with respect to
model metrics.

Our approach involves some features that we would like to point out.
As in Simpson $\GL_n\C$-case \cite{Sim90}, we need to consider a slight
extension of representations to the more general notion of filtered
local systems, what we call parabolic $G$-local system. The definition
of parabolic degree for both,  parabolic $G$-local systems and parabolic
$G$-Higgs bundles, involve the Tits geometry of the boundary at infinity
of the symmetric spaces $G/H$ and $H^\C/H$, respectively.
Another new feature is given by the fact that relating the data at the punctures for
the representation and  the parabolic Higgs bundle implies a relation between
$G$-orbits in $\lieg$ (the Lie algebra of $G$) and $H^\C$-orbits in $\liem^\C$,
where $\liem^\C$ is the complexification of $\liem$ given by the Cartan
decomposition $\lieg=\lieh \oplus \liem$, which in the case of nilpotent
orbits is known as the Kostant--Sekiguchi correspondence
\cite{sekiguchi,Kro90b,Kro90a,Ver95}.
For general orbits, this correspondence is proved in \cite{bielawski,biquard-note}.

We  give now a brief description of the different sections of
the paper. In Section \ref{parabolic-bundles} we define
parabolic principal bundles and some important notions related to them,
like their sheaves of (sometimes meromorphic) local automorphisms,
or their parabolic degree. In Section \ref{s:parahoric} we explain
the relation between parabolic and parahoric principal bundles,
and we discuss Hecke transformations; this gives a transparent interpretation
of the meromorphic local automorphisms of parabolic bundles. In Section
\ref{parabolic-higgs-bundles} we introduce parabolic
$G$-Higgs bundles and define the stability criteria. Section
\ref{hitchin-kobayashi} is one of the technical sections where
we carry out the analysis to prove the correspondence between
polystable parabolic $G$-Higgs bundles and solutions of the
Hermite--Einstein or Hitchin equation. As pointed out above,
one of the main difficulties here is in dealing with parabolic
structures lying in a `bad' wall of the Weyl alcove. In Section
\ref{parabolic-local-systems} we introduce the notion of
parabolic $G$-local system, prove an existence theorem for
harmonic reductions and establish the relation with parabolic
$G$-Higgs bundles completing the correspondence. In Section
\ref{moduli} we consider the moduli spaces of parabolic
$G$-Higgs bundles, parabolic $G$-local systems and
representations of the  fundamental group of the punctured
surface, and their various correspondences among them. We also
show how they depend on the parabolic weights, residues,
monodromy and conjugacy classes. For example, we show that if
$G$ is complex, for weights and residues of the Higgs field
satisfying a certain genericity condition, all the moduli
spaces of parabolic $G$-Higgs bundles are diffeomorphic. This
generalises the case of $\GL_2\C$-Higgs bundles studied by
Nakajima \cite{nakajima}.
 In Section \ref{sec:examples} we extend  results of Hitchin
\cite{hitchin-teichmueller} for split real forms and the theory of maximal
Higgs bundles \cite{BGG06,GGM,BGR} to the punctured set up.
We finish with two appendices containing  Lie-theoretic background and some
considerations on the Tits geometry of the boundary of our symmetric spaces,
to  define the relative degree of two parabolic subgroups, needed for the
definition of  parabolic  degree of Higgs bundles and local systems.

We believe that the results in this paper are a starting point for applying
Higgs bundle methods to the systematic study of the topology of the moduli
spaces of
representations of the fundamental group of a punctured surface in non-compact
reductive Lie groups. And, in particular, to exhibit the existence of higher
Teichm\"uller components in the punctured case, as we briefly show in
Section \ref{sec:examples}.
This has been extremely successful in the case of a
compact surface (see the survey paper \cite{garcia-prada}).
For punctured surfaces, Betti numbers have been computed  for
parabolic $\GL_2\C$-Higgs bundles
with compact holonomy by Boden
and Yokogawa \cite{boden-yokogawa}, and Nasatyr and Steer \cite{nasatyr-steer} in the case of
rational weights, and in \cite{GGMu} for
parabolic $\GL_3\C$-Higgs bundles
with compact
holonomy.  Some partial results on representations in $\U_{p,q}$ were
obtained in \cite{GLM} making some genericity assumptions,
although the relation between representations and parabolic
Higgs bundles which constitutes the main result in this paper was only
sketched there. The results in the present
paper  may also be used to translate the results on
the topology of parabolic $\U_{2,1}$-Higgs bundles obtained by Logares in
\cite{Log06}  to the context
of the moduli space of representations of the fundamental group.

One direction that can also be developed from this paper is the inclusion of
higher order poles in the Higgs field, relating to wild non-abelian
Hodge theory. This is  a problem on which we plan to come back
in the future.

A substantial amount of the content of this paper appeared in the notes
of a course given by the first author at the CRM (Barcelona) in 2010
\cite{crm-notes}. A survey on the subject has been given by the third author
in \cite{ignasi-hitchin70}. We apologize for having taken so long to produce this
paper.


\section{Parabolic principal bundles}
\label{parabolic-bundles}

\subsection{Definition of parabolic principal bundle}
\label{ss:definition-parabolic-bundle}

Let $X$ be  a compact connected Riemann surface and let $\{x_1,\cdots,x_r\}$
be  a finite set of different points of $X$. Let $D=x_1+\cdots +x_r$ be the
corresponding effective divisor.

Let $H^\setC$ be a reductive complex Lie group. We fix a maximal
compact subgroup $H\subset H^\setC$ and a maximal torus
$T\subset H$ with Lie algebra $\liet$.

Let $E$ be a holomorphic principal $H^{\setC}$-bundle over $X$.

If $M$ is any set on which $\HC$ acts on the left, we denote by
$E(M)$ the twisted product $E\times_{\HC}M$. If $M$ is a vector space
(resp. complex variety) and the action of $\HC$ on $M$ is linear
(resp. holomorphic) then $E(M)\to X$ is a vector bundle (resp.
holomorphic fibration). We denote by $E(H^{\setC})$, the
$H^{\setC}$-fibration associated to $E$ via  the adjoint
representation of $H^{\setC}$ on itself. Recall that for any $x\in X$
the fibre $E(H^{\setC})_x$ can be identified with the set of
antiequivariant maps from $E_x$ to $H^{\setC}$:
\begin{equation}
\label{eq:fibra-E-H-C} E(H^{\setC})_x=\{\phi:E_x\to H^{\setC}\mid
\phi(eh)=h^{-1}\phi(e)h \quad\forall\,e\in E_x,\,h\in H^{\setC}\,\}.
\end{equation}

We fix an alcove $\eA\subset\liet$ of $H$ such that
 $0\in \bar \eA$ (see Appendix \ref{alcoves}).  Let $\alpha_i\in \imag
\bar \eA$ and let $P_{\alpha_i}\subset H^{\setC}$ be the
parabolic subgroup defined by $\alpha_i$ as in Section
\ref{sec:parabolic-subgroups-1}. By Proposition
\ref{prop-alcove}, for $\alpha_i\in \imag\bar \eA$ the
eigenvalues of $\ad (\alpha_i)$ have an absolute value smaller
or equal than $1$, and $\liep_{\alpha_i}$ is the sum of the
eigenspaces of $\ad (\alpha_i)$ for nonpositive eigenvalues of
$\ad \alpha_i$. We shall distinguish the subalgebra
$\liep_{\alpha_i}^1\subset\liep_{\alpha_i}$ defined by
$$ \liep_{\alpha_i}^1 = \ker (\ad(\alpha_i)+1) , $$
and the associated unipotent group $P_{\alpha_i}^1\subset P_{\alpha_i}$.
Note that $P_{\alpha_i}^1$ is normal in $P_{\alpha_i}$.

We distinguish the subset $\eA'\subset \bar\eA$ of elements $\alpha$ such that the
eigenvalues of $\ad \alpha$ have an absolute value smaller than $1$.
If $\alpha_i\in \imag \eA'$ then $P_{\alpha_i}^1$ is trivial.

\begin{definition}
\label{def:par-structure}
We  define a \textbf{parabolic structure of weight $\alpha_i$ on $E$ over
a point $x_i$} as the choice of
a subgroup $Q_i\subset E(H^\setC)_{x_i}$ with the property
that
there exists some trivialization $e\in E_{x_i}$ for
which $P_{\alpha_i}=\{\phi(e)\mid\phi\in Q_i\}$ 
(here we use (\ref{eq:fibra-E-H-C}) to regard the elements of $Q_i$ as maps
$E_x\to H^{\setC}$).
\end{definition}

Note that the choice of $Q_i$ in the previous definition is equivalent to choosing
an orbit of the action of $P_{\alpha_i}$ on $E_x$, because the normalizer of
$P_{\alpha_i}$ inside $H^{\CC}$ is $P_{\alpha_i}$ itself.

\begin{definition}
\label{def:compatible-trivialization} A local holomorphic
trivialization of $E$ on a neighbourhood of $x_i$ is said to be
compatible with a parabolic structure $Q_i$ if, seen as a local
section of $E$ in
 the usual way, its value at $x_i$ is an element $e\in E_{x_i}$ satisfying $P_{\alpha_i}=\{\phi(e)\mid\phi\in Q_i\}$.
\end{definition}

If the parabolic structure is clear from the context, we will often simply say
that a given local trivialization is compatible.

Suppose that $Q_i\subset E(H^\setC)_{x_i}$ is a parabolic structure of
weight $\alpha_i$. Take any $e\in E_{x_i}$ such that
$P_{\alpha_i}=\{\phi(e)\mid\phi\in Q_i\}$. Define the following subgroup
of $Q_i$:
$$Q_i^1=\{\phi\in E(H^{\CC})_{x_i}\mid \phi(e)\in P_{\alpha_i}^1\}.$$
The subgroup $Q_i^1$ is intrinsic, i.e. it does not depend on
the choice of $e$. Indeed, if $e'\in E_{x_i}$  is another
element satisfying $P_{\alpha_i}=\{\phi(e')\mid\phi\in Q_i\}$
then $e'=eh$ for some $h\in P_{\alpha_i}$ because
$P_{\alpha_i}$ is its own normalizer inside $H^{\CC}$ is
$P_{\alpha_i}$. Since $P_{\alpha_i}^1$ is normal in
$P_{\alpha_i}$, we have
$$\{\phi\in E(H^{\CC})_{x_i}\mid \phi(e)\in P_{\alpha_i}^1\}=
\{\phi\in E(H^{\CC})_{x_i}\mid \phi(eh)\in P_{\alpha_i}^1\}.$$

We denote by $\lieq_i^1$ the Lie algebra of $Q_i^1$.


Let $\alpha=(\alpha_1,\cdots,\alpha_r)$ be a collection of elements
in $\imag \bar \eA$. A \textbf{parabolic principal bundle over}
$(X,D)$ of weight $\alpha$ is a (holomorphic) principal bundle
with a choice, for any $i$, of a parabolic structure of weight
$\alpha_i$ on $x_i$. We will usually not specify in the notation
the parabolic structure, so we will refer to them by the same
symbol denoting the underlying principal bundle. Similarly we will
often avoid referring to the weight $\alpha$.

Naturally associated to the parabolic principal bundle $E$ is the sheaf
$PE(H^\setC)$ of {\bf parabolic gauge transformations}, whose sections on
$X\setminus D$ are the holomorphic sections $g$ of $E(H^\setC)$, and near a marked
point $x_i$, the sections  are of the form $g(z) \exp(n/z)$ in some
  trivialization near $x_i$, such that $n\in \lieq_i^1$ and $g$ is
  holomorphic near $x_i$ with $g(0)\in Q_i$ (this is a group because
  $Q_i^1$ is a normal abelian subgroup of $Q_i$).

If $\alpha_i\in \imag\eA'$ then $PE(H^\setC)$ is the sheaf of
holomorphic sections of $E(H^\setC)$ such that $g(x_i)\in Q_i$, because
$\lieq_i^1=0$.

Of course we have the Lie algebra version $PE(\liehc)$, whose sections
are the sections $u$ of $E(\liehc)$ such that $u$ is meromorphic with
simple pole at $x_i$, $\Res_{x_i}u\in \lieq_i^1$ and the constant term
$u(x_i)\in \lieq_i$.

We next give a useful geometric interpretation of parabolic gauge transformations.
With respect to any compatible holomorphic trivialization of $E$ on a
neighbourhood of $x_i$ containing a disk $\Delta$, and any holomorphic
coordinate $z:\Delta\to\CC$ satisfying $z(x_i)=0$,
the holomorphic sections of $E(H^\setC)$ (resp. $E(\liehc)$) are
identified with maps $\Delta\to H^ \setC$ (resp. $\liehc$). Denoting
$\Delta^*=\Delta\setminus\{x\}$,
one can check immediately:
\begin{equation}\label{eq:33}\begin{split}
  \Gamma(\Delta,PE(\liehc))&= \{ u:\Delta^*\to\liehc\mid \Ad( |z|^{-\alpha_i})u(z) \text{ is uniformly bounded on }\Delta^*\},\\
  \Gamma(\Delta,PE(H^ \setC))&= \{ g:\Delta^*\to H^ \setC\mid |z|^{-\alpha_i}g(z)|z|^{\alpha_i} \text{ is uniformly bounded on }\Delta^*\}.
\end{split}\end{equation}
This means that parabolic transformations are holomorphic transformations
on the punctured disc $\Delta^*$, which remain bounded with respect to the metric $|z|^{-2\alpha_i}$. Recall that a metric for an $H^\C$-bundle
is a reduction of structure group to $H$ (see discussion before Definition
\ref{def:adapted_metric}).

\begin{remark}
\label{rmk:meromorphic-transformations}
  If $\alpha_i\in \imag(\bar\eA\setminus \eA')$ then there are local sections of
  $PE(H^ \setC)$ which are strictly meromorphic.
These can be understood as
  transforming the bundle $E$ into another principal bundle
of the same topological type but with a possibly different holomorphic structure.
In this case the
  holomorphic bundle underlying a parabolic bundle is only defined up
  to such transformations, although
the sheaves $PE(H^ \setC)$ and $PE(\liehc)$ are unique (this admits a transparent interpretation
  in terms of parahoric bundles, see Section \ref{s:parahoric} below).
\end{remark}

\begin{remark}
  The definition of parabolic bundle can be extended to arbitrary
  $\alpha$'s, i.e., not necessarily in $\imag\bar\eA$. However, this leads
  to more complicated objects that can be interpreted in terms of
  parahoric subgroups (see for example \cite{boalch}) for which
  our analysis to prove the Hitchin--Kobayashi correspondence does not
  apply directly. But we do not need to go to these objects since by
  taking $\alpha\in \imag\bar\eA$ we are able to parametrize all conjugacy
  classes of $H$ and also of a non-compact real reductive group $G$
  with maximal compact $H$ (see Appendix \ref{alcoves}) and hence
  obtain all possible monodromies at the punctures.
\end{remark}

\begin{example}\label{ex:defin-parab-princ}
  The basic example is $H=\U_n$  and $\HC=\GL_n\setC$, so an
  $\HC$-bundle is in one-to-one correspondence with the
  the  holomorphic vector bundle $V=E(\C^n)$ associated to the
  fundamental
representation of $\GL_n\C$. The parabolic structure at
  the point $x_i$ is given by a flag
  $V_{x_i}=V_i^1\supset V_i^2\supset\cdots\supset V_i^{\ell_i}$, with corresponding weights
  $1>\alpha_i^1>\alpha_i^2>\cdots >\alpha_i^{\ell_i}\geq 0$. The matrix $\alpha_i$ is diagonal,
  with eigenvalues the $\alpha_i^j$ with multiplicity
  $\dim(V_i^j/V_i^{j+1})$, and the eigenvalues $\alpha_i^j-\alpha_i^k$ of
  $\ad(\alpha_i)$ belong to $(-1,1)$. The parabolic subalgebra consists of
  the endomorphisms of $V_{x_i}$, preserving the flag: it is the sum
  of the non positive eigenspaces of $\ad \alpha_i$. Of course, in this case
  $P_{\alpha_i}^1=\{1\}$.
\end{example}

\begin{remark}
  It is important here to note that our convention for the parabolic
  weights is different from the usual convention in the literature, since
  we take a decreasing sequence of weights. This is somehow a more natural
  choice in the group theoretic context: in K\"ahler quotients, it is natural
  to have an action of the group on the symmetric space on the right; then,
  if $\alpha_i$ lies in a positive Weyl chamber, it defines a point on the
  boundary at infinity of the symmetric space $H\backslash H^\setC$, whose stabilizer is
  the parabolic subgroup whose Lie algebra is exactly the one defined
  above. See appendix \ref{sec:parabolic-subgroups} for details.

  A consequence of our convention is that the monodromy corresponding to the
  parabolic structure is $\exp(2\pi\imag \alpha_i)$ instead of $\exp(-2\pi\imag
  \alpha_i)$. Also note later the change of sign in the definition of the
  parabolic degree.
\end{remark}

\begin{remark}\label{rem:induced-par-str}
  If we have a morphism between two reductive groups $f:L^\CC \rightarrow H^\CC$, then given a parabolic principal $L^\CC$-bundle $E$, there is an induced principal $H^\CC$-bundle $E_f$. Suppose the parabolic structure of $E$ has weight $\alpha$ at a point $x$, then the description (\ref{eq:33}) shows that one can define parabolic transformations on $E_f$; but it is not completely obvious whether this comes from a parabolic structure on $E_f$ at $x$ with weight $f_*\alpha$.

  The simplest case (and the only one that we shall use) is that when $f_*\alpha\in \sqrt{-1}\bar\eA$: in that case we have standard parabolic groups $P_\alpha\subset L^\CC$ and $P_{f_*\alpha}\subset H^\CC$; there is a trivialization $e\in E_x$ in which the parabolic structure is given by $P_\alpha$, and we use the same $e$ seen as a trivialization of $(E_f)_x$ to define a parabolic structure given by the group $P_{f_*\alpha}$ on $E_f$ at $x$. It is not difficult to check that this does not depend on the choice of $e$ (as one may guess already from (\ref{eq:33})).

  We will not describe the general case, which requires to apply a Hecke transformation to $E_f$ in order to get back $f_*\alpha$ in the Weyl alcove, see also Section \ref{s:parahoric}.
\end{remark}


\subsection{Parabolic degree of parabolic reductions}\label{p-degree}

Let $E$ be a parabolic principal bundle over $(X,D)$ of weight $\alpha$ and
let $Q_i\subset E(H^{\setC})_{x_i}$ denote the parabolic subgroups specified by
the parabolic structure. For any standard parabolic subgroup $P\subset
H^{\setC}$, any antidominant character $\chi$ of $\liep$ (see Appendix
\ref{sec:parabolic-subgroups}),
and any holomorphic reduction $\sigma$ of
the structure group of $E$ from $H^{\setC}$ to $P$ we are going to define
a number $\pardeg(E)(\sigma,\chi)\in\setR$, which we call the parabolic degree. This
number will be the sum of two terms, one global and independent of the
parabolic structure, and the other local and depending on the
parabolic structure.

Before defining the parabolic degree, let us recall that the set of
holomorphic reductions of the structure group of $E$ from $\HC$ to $P$
is in one-to-one correspondence with the set of holomorphic sections
$\sigma$ of $E(\HC/P)$ (the latter is the bundle associated to the action
of $\HC$ on the left on $\HC/P$). Indeed, there is a canonical
identification $E(\HC/P)\simeq E/P$ and the quotient $E\to E/P$ has the
structure of a $P$-principal bundle.  So given a section $\sigma$ of
$E(\HC/P)$ the pullback $E_{\sigma}:=\sigma^*E$ is a $P$-principal bundle over
$X$, and we can identify canonically $E\simeq E_{\sigma}\times_{P}\HC$ as principal
$\HC$-bundles. Equivalently, we can look at $E_{\sigma}$ as a holomorphic
subvariety $E_{\sigma}\subset E$ invariant under the action of $P\subset\HC$ and
inheriting a structure of principal bundle.

%

Now fix a  parabolic subgroup $P=P_s\subset H^{\setC}$, for $s\in \imag\lieh$
(see Appendix \ref{sec:parabolic-subgroups-1}), an
antidominant character $\chi$ of $P$, and a holomorphic reduction
$\sigma$ of the structure group of $E$ from $H^{\setC}$ to $P$.

The global term in $\pardeg(E)(\sigma,\chi)$ is the degree $\deg(E)(\sigma,\chi)$
defined in \cite[(A.57)]{GGM}. We introduce it here using Chern--Weil
theory instead of the algebraic constructions of [op. cit.].

Let $E_{\sigma}$ be the $P$-principal bundle corresponding to the reduction
$\sigma$.  Given an antidominant character $\chi:\plie\to\CC$, where $\plie$ is the Lie algebra
of $P$, the degree is defined by the
Chern--Weil formula
\begin{equation}
\label{eq:def-deg-sigma}
\deg(E)(\sigma,\chi):=\frac{\imag}{2\pi}\int_X \chi_*(F_A)
\end{equation}
for any $P$-connection $A$ on $E_\sigma$. Here, $\chi_*(F_A)$ is the 2-form resulting from applying the
character $\chi$ to the $\plie$-valued 2-form $F_A$. Since $P\cap H$ is a maximal compact
subgroup of $P$ and the inclusion $P\cap H\hookrightarrow P$ is a homotopy equivalence
(see Appendix \ref{sec:parabolic-subgroups}), one can evaluate
(\ref{eq:def-deg-sigma}) using a $P\cap H$-connection, and it follows
that $\deg(E)(\sigma,\chi)$ is a real number. Recall that, by definition, an antidominant
character is real (see Appendix \ref{sec:parabolic-subgroups}).


At each marked point $x_i$ we have two parabolic subgroups of
$E(H^\setC)_{x_i}$ equipped with an antidominant character:
\begin{itemize}
\item one coming from the parabolic structure, $(Q_i,\chi_{\alpha_i})$, where
$\chi_{\alpha_i}$ is the  antidominant character of $\lieq_i$, defined in
Appendix \ref{sec:parabolic-subgroups-1};
\item one coming from the reduction, $(E_\sigma(P)_{x_i},\chi)$.
\end{itemize}
In appendix \ref{sec:parabolic-subgroups}, we define a relative degree
$\deg((Q_i,\alpha_i),(E_\sigma(P)_{x_i},\chi))$ of such a pair. Then we define the
parabolic degree as follows:
\begin{equation}
\label{eq:def-par-deg-sigma}
\pardeg_{\alpha}(E)(\sigma,\chi):=\deg(E)(\sigma,\chi)-\sum_i \deg\big((Q_i,\alpha_i),(E_\sigma(P)_{x_i},\chi)\big).
\end{equation}
When it is clear from the context we will omit the subscript
$\alpha$ in the notation of the parabolic degree.

The definition of parabolic degree of a reduction also makes sense
if one takes as parabolic subgroup the whole $\HC$. In this case
the set of antidominant characters is simply
$\Hom_{\setR}(\zlie,\imag\setR)$, where $\zlie$ is the centre of
$\hlie$. Trivially, in this case there is a unique reduction of
the structure group, which we denote by $\sigma_0$. We then define
for any $\chi\in \Hom_{\setR}(\zlie,\imag\setR)$
$$\pardeg_\chi E:=\pardeg_\alpha(E)(\sigma_0,\chi).$$

Of course, a priori the parabolic degree of a reduction does not seem to be well defined, because we can change the bundle by a meromorphic gauge transformation. Actually we will see that:
\begin{itemize}
\item after a meromorphic gauge transformation, the parabolic reductions are the same (proposition \ref{reduction});
\item there is an analytic formula for the degree (next section), which by its very definition is invariant my meromorphic gauge transformation.
\end{itemize}
So the parabolic reductions and their degree make perfect sense.

\subsection{Analytic formula for the parabolic degree of a
parabolic reduction}
\label{ss:analytic-global-degree}

Let $s\in\imag\lieh$ be any element, let $P=P_s\subset\HC$ be the
corresponding parabolic subgroup and let $\chi:\plie\to\CC$ be the
antidominant character defined as $\chi(\alpha)=\langle\alpha,s\rangle,$
where $\langle\cdot,\cdot\rangle:\liehc\times
\liehc\to\CC$ is the extension of an invariant scalar product on $\lieh$ to a
Hermitian pairing. Note that the intersection of $P$ and $H$ can be
identified with the centralizer of $s$ in $H$:
\begin{equation}
\label{eq:compact-parabolic} P\cap H=Z_H(s)=\{h\in H\mid
\Ad(h)(s)=s\}.
\end{equation}
Let $\sigma$ be a holomorphic reduction of the structure group
of $E$ from $\HC$ to $P$, and $E_\sigma$ the corresponding
$P$-principal bundle. Let $N\subset P$ be the unipotent part of
$P$, and choose the Levi subgroup $L=Z_{H^\setC}(s)\subset P$.
Then there is a well-defined $L$-action on $E_\sigma/N$ which
turns it into a principal $L$-bundle, which we denote
$E_{\sigma,L}$. (In the vector bundle case, the $P$-reduction
is a flag of sub-vector bundles, and the $L$-bundle is the
associated graded bundle.) Note that since $L=Z_{H^\setC}(s)$,
the element $s\in \imag\lieh$ defines a canonical section
\begin{equation}
  \label{eq:9}
  s_\sigma\in \Gamma(E_{\sigma,L}(\liel\cap \imag\lieh)).
\end{equation}

Let $h$ be a smooth metric on $E$, defined on the whole curve $X$, and
let $\bE$ be the $H$-principal bundle obtained by reducing the
structure group of $E$ from $\HC$ to $H$ using $h$. Combining $\sigma$ and
$h$ we obtain a reduction of the structure group of $E$ from $\HC$ to
$P\cap H$. Denote by $\bE_{\sigma}$ the resulting bundle. But $P\cap H$ is a
compact form of $L$, and the complexified bundle clearly identifies to
$E_{\sigma,L}$, so we can also think of $\bE_\sigma$ as the bundle $E_{\sigma,L}$
equipped with the metric $h_{\sigma,L}$ induced by $h$.
From this point of view, the section
$s_\sigma$ above can be seen as a section
\begin{equation}
  \label{eq:10}
  s_{\sigma,h}\in \Gamma(\bE_{\sigma}(\imag\lieh))\simeq\Gamma(\bE(\imag\lieh)).
\end{equation}
Note the difference between (\ref{eq:9}) and (\ref{eq:10}): the
section $s_\sigma$ is canonical, while different choices of $h$ lead to
different sections $s_{\sigma,h}$ of $E(\lieh^\setC)$.

We now introduce some notation. Let $V$ be a Hermitian vector space,
let $\rho:\lieh\to\lieu(V)$ be a morphism of Lie algebras, and denote also
by $\rho:\liehc\to\End V$ its complex extension.  Choose elements
$a\in\imag\lieh$ and $v\in V$. Then $\rho(a)$ diagonalizes and has real
eigenvalues, so that we may write $v=\sum v_j$ in such a way that
$\rho(a)(v)=\sum l_jv_j$. Now, for any function $f:\setR\to\setR$ we define
$f(a)(v):=\sum f(l_j)v_j$.

\begin{lemma}\label{lemma:analytic-degree-global}
  Define the function $\varpi:\setR\to\setR$ as $\varpi(0)=0$ and as $\varpi(x)=x^{-1}$ if $x\neq
  0$. Applying the previous definition to the adjoint representation,
  and extending it to sections of $E(\liehc)\otimes K$, we have:
$$\deg(E)(\sigma,\chi)=\frac{\imag}{2\pi}
\int_X\langle F_h,s_{\sigma,h}\rangle-\langle\varpi(s_{\sigma,h})(\overline{\partial}s_{\sigma,h}),\overline{\partial}s_{\sigma,h}\rangle.$$
\end{lemma}

Here $F_h$ is the curvature of the unique connection compatible with $h$ and the holomorphic structure of $E$, the Chern connection \cite{singer}.
The proof follows Chern--Weil theory and
from identifying the RHS of the formula with the
curvature $F_{h,L}$ of the Chern connection of the metric
$h_{\sigma,L}$ on $E_{\sigma,L}$ defined by $h$. More precisely,

\begin{lemma}\label{lemma:curvatura-Levi}
$\langle F_{h,L},s_{\sigma,h}\rangle= \langle F_h,s_{\sigma,h}\rangle-\langle\varpi(s_{\sigma,h})(\overline{\partial}s_{\sigma,h}),\overline{\partial}s_{\sigma,h}\rangle.$
\end{lemma}
\begin{proof}
  We can think of $E$ and $E_{\sigma,L}$ as giving two holomorphic structures on the same principal bundle obtained by complexifying $\bE$. Therefore the difference between the corresponding $\dbar$ operators is a $\bE(\liehc)$-valued $(0,1)$-form: $\dbar^E-\dbar^{E_{\sigma,L}}=a$ with $a\in \Omega^{0,1}(\lien)$ since $E_{\sigma,L}$ is a reduction of $E$. Noting $\dbar^{E_{\sigma,L}}+\partial_h^{E_{\sigma,L}}$ the Chern connection of $E_{\sigma,L}$, we have $F_h=F_{h,L}+\partial_h^{E_{\sigma,L}}a+\dbar^{E_{\sigma,L}}\tau_h(a)+[a,\tau_h(a)]$, and therefore \[\langle F_{h,L},s_{\sigma,h}\rangle= \langle F_h,s_{\sigma,h}\rangle+\langle[a,\tau(a)],s_{\sigma,L}\rangle= \langle F_h,s_{\sigma,h}\rangle+\langle a,[a,s_{\sigma,h}]\rangle.\] The formula follows since $\dbar^Es_{\sigma,L}=[a,s_{\sigma,h}]$ hence $\varpi(s_{\sigma,h})(\overline{\partial}s_{\sigma,h})$ is the projection of $-a$ on nonzero eigenspaces of $\ad s_{\sigma,h}$.
\end{proof}

Our aim in the remainder of this section is to state and prove an
analogue of Lemma \ref{lemma:analytic-degree-global} giving the
parabolic degree. For that it will be necessary to replace the
metric $h$ (which was chosen to be smooth on the whole $X$) by a
metric which blows up at the divisor $D$ at a speed specified by
the parabolic weights $\alpha_i$.

\subsection*{$\alpha$-adapted metrics and parabolic degree}

It may be useful here to remind in a few words how we write local formulas
for the metrics of principal bundles.  There is a right action $h\mapsto h\cdot g$ of
gauge transformations $g\in \Gamma(E(\HC))$ on metrics $h\in\Gamma(E/H)$, which
identifies in each fibre to the standard action of $\HC$ on the symmetric
space $H\backslash \HC$. In concrete terms, a choice of $h\in\Gamma(E/H)$
is equivalent to a map $\chi:E\to H\backslash H^{\CC}$ satisfying $\chi(e\gamma)=\chi(e)\gamma$
for $e\in E$ and $\gamma\in H^{\CC}$ (any such $\chi$ corresponds to the section
$h\in\Gamma(E/H)$ such that $h(x)=\{[e]\mid e\in E_x,\,\chi(e)\in H\in H\backslash H^{\CC}\}$),
and a gauge transformation $g\in \Gamma(E(\HC))$ is equivalent to a map
$\zeta:E\to H^{\CC}$ satisfying $\zeta(e\gamma)=\gamma^{-1}\zeta(e)\gamma$.
Then $h\cdot g$ is the section of $E/H$ corresponding to the map $\chi\zeta:E\to H\backslash H^{\CC}$.

Let $\tau$ be the involution of $\HC$ fixing $H$.
A local trivialization $e$ of $E$ defines a metric $h_0$;
any another metric is given by $h=h_0\cdot g$ for some $g$ with values in $\HC$; of
course $h$ depends only on $\tau(g)^{-1}g$ and we identify $h=\tau(g)^{-1}g$ (in
the linear case, this is just writing $h=g^*g$). Of course it is always
possible to move $g$ by an element of $H$ so that $\tau(g)^{-1}=g$ and
then $h=g^2$.

Let us now consider a parabolic bundle $E$, and a metric
$h\in\Gamma(X\setminus D;E/H)$ defined away from the divisor
$D$.
\begin{definition}\label{def:adapted_metric}
  We say that $h$ is an {\bf $\alpha$-adapted metric} if for any parabolic point
  $x_i$ the following holds. Choose a local holomorphic coordinate $z$ and a local holomorphic trivialization $e_i$
  of $E$ near $x_i$ compatible
  with the parabolic structure (see Definition \ref{def:compatible-trivialization}).
  Then there is some meromorphic gauge transformation $g\in PE(H^\CC)$ near $x_i$, such that in the trivialization $g(e_i)$ one has
  \begin{equation}
    \label{eq:h-model}
    h = h_0 \cdot |z|^{-\alpha_i}e^c,
  \end{equation}
  where  $h_0$ is the standard constant metric,
  $\Ad(|z|^{-\alpha_i})c=o(\log |z|)$, $\Ad(|z|^{-\alpha_i})dc\in L^2$ and
  $\Ad(|z|^{-\alpha_i})F_h\in L^1$.
\end{definition}
The definition of $h$ and the conditions on $c$ are clarified by observing that in the trivialization $e_i\cdot|z|^{\alpha_i}$ (which is orthonormal for the metric $|z|^{-2\alpha_i}=h_0\cdot |z|^{-\alpha_i}$), the metric $h$ can be written as $h_0\cdot\exp(\Ad(|z|^{-\alpha_i})c)$. If one chooses $c$ so that $\tau(\Ad(|z|^{-\alpha_i})c)^{-1}=\Ad(|z|^{-\alpha_i})c$, then $h$ can simply be written as $$h=|z|^{-2\alpha_i}e^{2c}.$$

Note that the $L^2$ and $L^1$ conditions on $dc$ and $F_h$ are conformally
invariant.

Some discussion is in order on the role of the meromorphic gauge transformation in the definition. Suppose we are given a holomorphic $g=e^c\in PE(H^ \setC)$, so $c\in PE(\lieh^ \setC)$. Then $\Ad( |z|^{-\alpha_i})c$ is bounded so the first condition is satisfied (this says that the gauge transformed metric remains at bounded distance in $H\backslash H^\setC$). Now analyse $\Ad(|z|^{-\alpha_i})dc$ by decomposing $c=\sum c_\lambda$ along the eigenvalues $\lambda$ of $\ad \alpha_i$:

\begin{itemize}
\item if $-1<\lambda\leq 1$ we get a term $|z|^\lambda dz\in L^2$;
\item if $\lambda=\pm 1$ we get a term $(\frac z{|z|})^{\pm 1}\frac{dz}z$ which is not $L^2$, so the conditions are not invariant by meromorphic equivalence.
\end{itemize}
More generally, it is easy to check that if the weight $\alpha_i\in \imag\AAA'$, then the conditions on $c$ actually do not depend on the choice of the holomorphic trivialization, so one can test the condition in any given trivialization. This is different when $\ad \alpha_i$ has $\pm 1$ eigenvalues, and this reflects a larger choice of `good' metrics. We have not tried to give a more precise description of  this class of metrics here: the definition is sufficient for the calculation of the degree in the next lemma, and from the proof follows that the result does not depend on the meromorphic gauge. The actual space of metrics used to solve the equations is defined in (\ref{eq:27}) and does not depend on any gauge.



We are now ready to state and prove the analogue of Lemma
\ref{lemma:analytic-degree-global} for the parabolic degree.

\begin{lemma}
\label{lemma:analytic-degree-parabolic} Let
$\varpi:\setR\to\setR$ be defined as in Lemma
\ref{lemma:analytic-degree-global}, and let $h$ be an
$\alpha$-adapted metric. Then:
$$\pardeg_{\alpha}(E)(\sigma,\chi)=
\frac{\imag}{2\pi} \int_{X\setminus D}\langle F_h,s_{\sigma,h}\rangle -
  \langle\varpi(s_{\sigma,h})(\overline{\partial}s_{\sigma,h}),\overline{\partial}s_{\sigma,h}\rangle .$$
\end{lemma}

\begin{proof}
  For any $v>0$ let $X_{v}=\{x\in X\mid d(x,D)\geq e^{-v}\}$ and $B_{v}=X\setminus X_{v}$.
  Let $h_{v}$ be a smooth metric on $E$ (defined on the whole $X$) which
  coincides with $h$ in a neighbourhood of $X_{v}\subset X$. The metrics $h$ and
  $h_v$ induce metrics $h_L$ and $h_{L,v}$ on $E_{\sigma,L}$, and we denote by
  $\bE_{\sigma,v}$ the resulting Hermitian holomorphic bundle.  Denote the
  curvatures of $h_L$ and $h_{L,v}$ by $F_{h,L}$ and $F_{h_{v},L}$. It
  follows from the definition of $\deg(E)(\sigma,\chi)$ that
  \begin{align*}
    \deg(E)(\sigma,\chi)&=\frac{\imag}{2\pi}\int_{X}\langle
    F_{h_{v},L},s_\sigma\rangle \\
    &=\frac{\imag}{2\pi}\left( \int_{X_{v}}\langle F_{h_{v},L},s_\sigma\rangle + \int_{B_{v}}\langle
      F_{h_{v},L},s_\sigma\rangle\right).
  \end{align*}
  By Lemma \ref{lemma:curvatura-Levi} we have
  \begin{align*}
    \lim_{v\to\infty}\int_{X_{v}}\langle F_{h_{v},L},s_\sigma\rangle
    & =\int_{X}\langle F_{h,L},s_\sigma\rangle \\
    & =\int_{X\setminus D}\langle F_h,s_{\sigma,h}\rangle-\langle\varpi(s_{\sigma,h})(\overline{\partial}s_{\sigma,h}),\overline{\partial}s_{\sigma,h}\rangle.
  \end{align*}
  Hence we need to prove that the remaining integral converges to the local
  terms in the definition of the parabolic degree, i.e.:
  \begin{equation}
    \lim_{v\to\infty}\frac{\imag}{2\pi}\int_{B_{v}}\langle F_{h_{v},L},s_\sigma\rangle
    =\sum_i\deg\big((P,\chi),(Q_i,\alpha_i)\big).\label{eq:8}
  \end{equation}
  Observe that on a small ball, we can use a holomorphic trivialization of
  the bundle $E_\sigma$, and the connection form becomes $A_v=h_{L,v}^{-1}\partial
  h_{L,v}$ with curvature $F=dA_v$. It follows that the quantity we have to
  study is
  $$ \lim_{v\to\infty} \frac{\imag}{2\pi} \int_{\partial B_v} \langle A_v,s\rangle
  =\lim_{v\to\infty} \frac{\imag}{2\pi} \int_{\partial B_v} \langle h_L^{-1}\partial h_L,s\rangle, $$ since
  $h_{L,v}$ coincides with $h_L$ in a neighbourhood of $X_v$.

  We can choose a holomorphic trivialization of $E$ in which the reduction
  $\sigma$ has constant coefficients. This induces also a holomorphic
  trivialization of the $L$-bundle $E_{\sigma,L}$.

  We begin by the case where the metric $h$ on   $E$ in this trivialization can be written as
  \begin{equation}
  h=\tau(g)^{-1}g, \text{ where }g=|z|^{-\alpha_i}e^c, \label{eq:35}
  \end{equation}
  and using $\HC=HP=HLN$ we decompose
  $$ g = k(z) l(z)n(z) , \quad \text{with } k(z)\in H, l(z)\in L, n(z)\in N. $$
  The pair $(k(z),l(z))$ is defined only up to the action of $P\cap H$, so we
  can as well suppose that $\tau(l)=l^{-1}$. It then follows
  that $h=\tau(n(z))^{-1}l(z)^2n(z)$, so the induced metric on
  $E_{\sigma,L}$ is
  $$ h_L = l(z)^2 . $$
  Then, because $s$ is $L$-invariant, the limit to calculate becomes
  $$ \lim_{v\to\infty} \frac{\imag}{\pi} \int_{\partial B_v} \langle \partial l l^{-1},s\rangle. $$

  Now one has
  $$ \partial l l^{-1} = \Ad(k)^{-1}(\partial g g^{-1})-k^{-1}\partial k - \Ad(l) (\partial n n^{-1}) .
  $$
  Observe that since $s\in \imag\lieh$ and $k\in H$, the term
  $\frac{\imag}\pi\int_{\partial B_v}\langle k^{-1}\partial k,s\rangle$ reduces to $\frac1{2\pi}\int_{\partial
    B_v}\langle\imag k^{-1}\partial_\theta k,s\rangle$.
  Therefore the limit reduces to
  \begin{equation}
    \label{eq:28}
    \lim_{v\to\infty} \frac{\imag}{\pi} \int_{\partial B_v} \langle\Ad(k)^{-1}(\partial g g^{-1})-\tfrac12 k^{-1}\partial_\theta k,s\rangle .
  \end{equation}
  On the other hand, decompose similarly $e^{t\alpha}=\tilde k(t)\tilde p(t)$
  with $\tilde k(t)\in H$ and $\tilde p(t)\in P$ then $\langle s\cdot e^{-t\alpha},\alpha) =
  \langle\Ad(\tilde k(t))s,\alpha\rangle $, so we obtain:
  \begin{equation}
    \label{eq:29}
    \deg\big( (P,\chi),(Q_i,\alpha_i)\big)
    = \lim_{t\to\infty}\langle s\cdot e^{-t\alpha},\alpha\rangle
    = \lim_{t\to\infty}\langle s,\Ad(\tilde k(t))^{-1}\alpha\rangle.
  \end{equation}
  If we have exactly the model behaviour $g=|z|^{-\alpha_i}$, then one has
  $k(z)=\tilde k(-\ln|z|)$ and $\partial g g^{-1}=-\frac{\alpha_i}2\frac{dz}z$, so $\partial_\theta k=0$ and
  $$
  \frac{\imag}{\pi} \int_{\partial B_v} \langle\Ad(k(z))^{-1}(\partial g g^{-1}),s\rangle = \langle\Ad(\tilde
  k(v)^{-1})\alpha_i,s\rangle ,
  $$
  so the two limits (\ref{eq:28}) and (\ref{eq:29}) are the same. It is
  then easy to check that this remains true if $g=|z|^{-\alpha_i}e^c$, where $c$
  is a perturbation satisfying the conditions below (\ref{eq:h-model}).

Finally we must see what is happening if the metric has the form (\ref{eq:35}) in another meromorphic gauge. It is sufficient to check this when the metric comes from the
standard metric $|z|^{-\alpha_i}$ via a meromorphic gauge
transformation. So write $g=|z|^{-\alpha_i}q$ with $q\in
PE(\lieh^ \setC)$ holomorphic outside $x_i$. Then in
(\ref{eq:28}) one has
$$ \partial g g^{-1} = - \frac{\alpha_i}2 \frac{dz}z + \Ad( |z|^{-\alpha_i} ) (\partial q q^{-1}). $$
As observed before the lemma, the only problems come from the eigenspaces of $\ad \alpha_i$ for the eigenvalues $\pm 1$: we have $\partial q q^{-1}=a \frac{dz}{z^2} + b dz + \cdots $, where $[\alpha_i,a]=-a$ and we distinguish the component $b_1$ of $b$ on the eigenspace of $\ad \alpha_i$ for the eigenvalue $1$. Then
$$  \Ad( |z|^{-\alpha_i} ) (\partial q q^{-1}) = a \frac{|z|dz}{z^2} + b_1 \frac{dz}{|z|} + O(\frac1{|z|^{1-\epsilon}}). $$
The integration on the circle can therefore bring no contribution into (\ref{eq:28}) and the result does not depend on the meromorphic gauge.
\end{proof}

\section{Parahoric bundles}
\label{s:parahoric}

\newcommand{\hol}{\operatorname{hol}}
\newcommand{\mer}{\operatorname{mer}}

\newcommand{\aA}{{\EuScript{A}}}
\newcommand{\bB}{{\EuScript{B}}}
\newcommand{\cC}{{\EuScript{C}}}
\newcommand{\dD}{{\EuScript{D}}}
\newcommand{\eE}{{\EuScript{E}}}
\newcommand{\fF}{{\EuScript{F}}}
\newcommand{\gG}{{\EuScript{G}}}
\newcommand{\hH}{{\EuScript{H}}}
\newcommand{\iI}{{\EuScript{I}}}
\newcommand{\jJ}{{\EuScript{J}}}
\newcommand{\kK}{{\EuScript{K}}}
\newcommand{\lL}{{\EuScript{L}}}
\newcommand{\mM}{{\EuScript{M}}}
\newcommand{\oO}{{\EuScript{O}}}
\newcommand{\qQ}{{\EuScript{Q}}}
\newcommand{\pP}{{\EuScript{P}}}
\newcommand{\sS}{{\EuScript{S}}}
\newcommand{\tT}{{\EuScript{T}}}
\newcommand{\vV}{{\EuScript{V}}}
\newcommand{\wW}{{\EuScript{W}}}
\newcommand{\xX}{{\EuScript{X}}}
\newcommand{\zZ}{{\EuScript{Z}}}

\newcommand{\un}{\underline}

\newcommand{\EE}{{\Bbb E}}

\newcommand{\std}{\operatorname{std}}

Our definition of parabolic bundles is suitable for prescribing
the asymptotic behaviour near a divisor $D\subset X$ of
reductions, defined on $X\setminus D$, of the structure group
of an $H^\CC$-principal bundle to the maximal compact subgroup
$H\subset H^{\CC}$. But it has the obvious inconvenient that in
the situations when $P_{\alpha_i}^1$ is nontrivial the sheaf
$PE(H^{\CC})$ contains meromorphic sections whose effect on $E$
is somewhat unclear (see Remark
\ref{rmk:meromorphic-transformations} above), in that they
transform $E$ into a different bundle. This issue becomes
transparent from the point of view of parahoric bundles
introduced by Pappas and Rapoport \cite{pappas-rapoport}, and
whose moduli problem was studied by Balaji and Seshadri in
\cite{balaji-seshadri} in relation with local systems with
compact structure group on punctured Riemann surfaces. In this
section we recall the definition of parahoric bundles and we
relate it to the objects we just defined.
We would like to emphasize that this section is not strictly
necessary for our arguments, except for the definition of
meromorphic equivalence given in Definition
\ref{def:meromorphic-equivalence} and the results in Section
\ref{ss:meromorphic-maps}. The main purpose of most of this
section is to clarify the nature of the objects we are going to
work with.

We also show how the notion of parahoric bundles is the natural
context for Hecke transformations.
This is of course well known in the algebraic world (see e.g.
\cite[Section 8.2.1]{balaji-seshadri}), but our approach, much
more analytic than that of
\cite{balaji-seshadri,pappas-rapoport}, seems to be new (see
Subsection \ref{ss:Hecke-transformations} below).

The existence of Hecke transformations implies that even restricting
to the case $\alpha_i\in\imag\bar\eA$ we cover, up to
isomorphism, all possible choices of weights $\alpha$. This is
one reason for avoiding parahoric bundles in our approach. A
second reason, not unrelated to the first, is that taking
weights in $\imag\bar\eA$ suffices to realize all possible
monodromies around the punctures in the correspondence between
parabolic Higgs bundles and local systems, as we prove in the
paper. On the other hand one can use Hecke transformations to
identify, for any choice of weights $\alpha$, the category of
parahoric bundles of weight $\alpha$ with the category of
parabolic orbibundles on $X$ for suitable choices of: weights
in $\imag\eA'$, orbifold structure on the divisor $D$, and
topological type of orbibundles. This can be done thanks to
Proposition \ref{prop-alcove}.


The Hecke transformations are compatible with the reductions of the
structure group to standard parabolic subgroups, since they
{\it use} only the action of the complexified maximal torus
$T^{\CC}$, which is contained in all standard parabolic
subgroups. This, combined with an easy computation of degrees,
implies that Hecke correspondence is compatible with the
stability notions {\it \`a la} Ramanathan, see Subsection
\ref{stability}.

Finally, Hecke transformations can be used to extend the definition of parabolic
Higgs bundles for a choice of parabolic structure of weights in $\imag\bar\eA$
(see Section \ref{sec:defin-parab-g}) to a notion of parahoric Higgs bundles
for an arbitrary choice of weights $\alpha$.



\subsection{Definition of parahoric bundles}
\label{ss:parahoric}

Let $\alpha=(\alpha_1,\dots,\alpha_r)$ be a collection of arbitrary elements of $\imag\tlie$.
Let $X^*=X\setminus D=X\setminus\{x_1,\dots,x_r\}$ and for any open subset $\Omega\subset X$ denote
$\Omega^*:=\Omega\cap X^*$. Choose for each $j$ a local holomorphic coordinate $z_j$ on
a neighbourhood of $x_j$ satisfying $z_j(x_j)=0$.

Let $\gG_{\alpha}$ be the sheaf of groups on $X$ defined as
follows: for any open subset $\Omega\subset X$,
$\gG_{\alpha}(\Omega)$ is equal to the set of all holomorphic
maps $\phi:\Omega^*\to \HC$ with the property that for every
$j$ such that $x_j\in\Omega$
$$|z_j(p)|^{-\alpha_j}\phi(p)|z_j(p)|^{\alpha_j}=\exp(-\ln|z_j(p)|\alpha_j)\phi(p) \exp(\ln|z_j(p)|\alpha_j)$$
stays in a compact subset of $\HC$ as $p\to x_j$. It is easy to
check that this definition is independent of the choice of
local holomorphic coordinates. The sheaf $\gG_{\alpha}$ is the
holomorphic analogue of the Bruhat--Tits group scheme
$\gG_{\Theta,X}$ in \cite{balaji-seshadri}.

A {\bf parahoric bundle} over $(X,D)$ of weight $\alpha$ is a sheaf of torsors over the
sheaf of groups $\gG_{\alpha}$.

Let $W$ denote the Weyl group of $H$ and let
$W\eA'=\bigcup_{w\in W}w\eA'$. We next prove that if
$\alpha_j\in \imag W\eA'$ for every $j$ then a parahoric bundle
of weight $\alpha$ is equivalent to a parabolic principal
bundle in the usual sense. Define $\gG_{\alpha}^{\std}$ to be
the sheaf of groups on $X$ such that, for any open
$\Omega\subset X$, $\gG_{\alpha}^{\std}(\Omega)$ is equal to the set of all
holomorphic maps $\phi:\Omega\to H^{\CC}$ satisfying
$\phi(x_j)\in P_{\alpha_j}$ for each $j$ such that $x_j\in
\Omega$. We now have
\begin{equation}
\label{eq:parabolic-standard}
\gG_{\alpha}=\gG_{\alpha}^{\std}\qquad\Longleftrightarrow\qquad
\alpha_j\in \imag W\eA'.
\end{equation}
This is an immediate consequence of the characterisation of
$\imag W\eA'$ as the subset of $\imag\tlie$ consisting of those
elements $\beta$ such that all eigenvalues of $\ad(\beta)$ have
absolute value smaller than $1$ (see (3) in Proposition \ref{prop-alcove}).
To conclude our argument, let
us see that a sheaf of torsors over $\gG_{\alpha}^{\std}$ is
the same thing as a holomorphic $H^{\CC}$-principal bundle $E$
and a choice of parabolic structures $\{Q_j\subset
E(H^{\CC})_{x_j}\}_j$.

Suppose given $(E,\{Q_j\})$.
For each $j$ define $R_j\subset E_{x_j}$
to be the set of all $e\in E_{x_j}$ such that $P_{\alpha_j}=\{\phi(e)\mid \phi\in Q_j\}$.
Let $\eE$ be the sheaf whose sections on an open subset
$U\subset X$ are the holomorphic sections of $E|_U$ whose value at each $x_j$
belongs to $R_j$. Then $\eE$ is a sheaf of $\gG_{\alpha}^{\std}$-torsors.
For the converse, let $\gG$ denote the sheaf of local holomorphic maps to $\HC$.
Clearly $\gG_{\alpha}^{\std}$ is a subsheaf of $\gG$, so if $\eE$ is a sheaf
of $\gG_{\alpha}^{\std}$-torsors then $\eE'=\eE\times_{\gG_{\alpha}^{\std}}\gG$ is
a sheaf of $\gG$-torsors, which can be identified with the sheaf of local holomorphic
sections of a holomorphic principal $\HC$-bundle $E$ over $X$. Now, $\eE$ is naturally
a subsheaf of $\eE'$, and defining for every $j$ the subset $R_j\subset E_{x_j}$ as the
set of images at $x_j$ of local sections of $E$ contained in $\eE$ we obtain an orbit
of the action of $P_{\alpha_j}$ at $E_{x_j}$. Then setting $Q_j=\{\phi\in E(\HC)_{x_j}\mid \phi(R_j)=P_{\alpha_j}\}$ we obtain a parabolic structure at $x_j$. This construction
is clearly the inverse of the previous one.

\subsection{Parahoric bundles vs. parabolic bundles with weights in $\imag\bar\eA$}
Our next aim is to understand the parabolic bundles defined in
Subsection \ref{ss:definition-parabolic-bundle} (which may have weights in $\imag\bar\eA$)
from the viewpoint of parahoric bundles.

\newcommand{\wall}{\operatorname{wall}}

For any $\alpha=(\alpha_1,\dots,\alpha_r)$,
define a sheaf of groups
$\gG_{\alpha}^{\wall}$
over $X$ by the prescription that, for any open subset $\Omega\subset X$,
$\gG_{\alpha}^{\wall}(\Omega)$
is equal to the group of holomorphic maps $\Omega^*\to H^{\CC}$
such that for every $x_j\in\Omega$ and any local holomorphic coordinate
$z$ defined on a neighbourhood of $x_j$ and satisfying $z(x_j)=0$ we have
$\phi=g\exp(n/z),$
where $g$ is holomorphic map from a neighbourhood of $x_j$ to $\HC$ satisfying
$g(x_j)\in P_{\alpha_j}$ and
$n$ is an element of $\plie_{\alpha_j}^1$. Similarly to (\ref{eq:parabolic-standard}), we
have
\begin{equation}
\label{eq:parabolic-wall}
\gG_{\alpha}=\gG_{\alpha}^{\wall}\qquad\Longleftrightarrow\qquad
\alpha_j\in \imag W\bar{\eA}=\imag\bigcup_{w\in W}w\bar\eA.
\end{equation}
%

Let $\alpha=(\alpha_1,\dots,\alpha_r)$ be a collection of elements of $\imag \bar \eA$.
%
%
Let $(E,\{Q_i\})$ be a parabolic bundle of weight $\alpha$, and let
$\eE^{\std}$ be the corresponding sheaf of $\gG_{\alpha}^{\std}$-torsors,
defined in the subsection above. Now,
$\gG_{\alpha}^{\std}$ is a subsheaf of $\gG_{\alpha}$
(this is not true for arbitrary values of $\alpha$: in fact, it is equivalent to the condition that
$\alpha_j$ belongs to $\bar\eA$ for each $j$).
Hence, one can associate a parahoric
bundle to $(E,\{Q_i\})$ by extending the structure group of $\eE^{\std}$. Namely,
$$\eE=\eE^{\std}\times_{\gG_{\alpha}^{\std}}\gG_{\alpha}.$$
Then equality (\ref{eq:parabolic-wall})
implies that the sheaf of groups $PE(H^{\CC})$ is canonically isomorphic to
the sheaf of automorphisms of $\eE$.

It is interesting to understand how to go the other way round, since this
explains why we need to identify different principal bundles.
Passing from $\eE$ to $\eE^{\std}$ is the same as reducing the structure group
from $\gG_{\alpha}$ to $\gG_{\alpha}^{\std}$, equivalently, choosing a section of $\eE/\gG_{\alpha}^{\std}$.
The latter is supported on $D$, and the stalk over $x_j$ can be identified, using
the residue map (hence, non canonically),
with $(\gG_{\alpha})_{x_j}/(\gG_{\alpha}^{\std})_{x_j}\simeq\plie_{\alpha_j}^1$.
Hence the collection of
all reductions from $\eE$ to $\eE^{\std}$ is parametrized by the vector space
$\prod_j\plie_{\alpha_j}^1$. The set of all reductions defines a holomorphic principal bundle
over $X\times \prod_j\plie_{\alpha_j}^1$,
so all possible reductions have topologically equivalent underlying principal $H^{\CC}$-bundles.
However, they will usually be different as holomorphic bundles.

\begin{example}
Suppose that $H=\SU_2$, so that $H^{\CC}=\SL_2\CC$ and assume that $V=L\oplus L^{-1}$ for some
line bundle $L\to X$. Such $V$ defines in the usual way a principal $\SL_2\CC$-bundle $E$.
Take $D=x$, and choose $\alpha\in \sqrt{-1}\liesu_2$ to be the diagonal matrix
with entries $1/2$ and $-1/2$.
Let $Q$ be the group of automorphisms of $L_x\oplus L_x^{-1}$ preserving the
summand $L_x$. Then $(E,Q)$ is a parabolic principal bundle with weight $\alpha$.
Let $\eE^{\std}$ be the sheaf of $\gG_{\alpha}^{\std}$-torsors associated to $(E,Q)$ and let $\eE=\eE^{\std}\times_{\gG_{\alpha}^{\std}}\gG_{\alpha}$. We have
$$\plie^1=\left\{\left(\begin{array}{cc} 0 & \lambda \\ 0 & 0\end{array}\right):\lambda\in\CC\right\}.$$
A choice of $\lambda$ defines a reduction of $\eE$ whose underlying principal $\SL_2\CC$-bundle
corresponds to a holomorphic vector bundle $V_{\lambda}$ sitting in a short exact sequence
$$0\to L\to V_{\lambda}\to L^{-1}\to 0.$$
We claim that in general this sequence does not split. The relation
between $V$ and $V_{\lambda}$ can be described explicitly in
terms of local trivialization; namely, passing from $V$ to
$V_{\lambda}$ consists on multiplying the patching map for
suitable trivializations of $V$ on $X\setminus\{x\}$ and
a disk $\Delta$ centred at $x$ by the map $$\Delta\setminus\{0\}\to\SL_2\CC,\qquad
z\mapsto \left(\begin{array}{cc} 1 & \lambda z^{-1} \\ 0 &
1\end{array}\right).$$
Choosing $L$ appropriately, the vector
bundle $V_{\lambda}$ is going to be (holomorphically) different
from $V$ when $\lambda\neq 0$.
\end{example}


\begin{definition}
\label{def:meromorphic-equivalence}
Let $E_0$, $E_1$ be two holomorphic
$H^{\CC}$-principal bundles and let
$\{Q_{i,j}\subset E_i(H^{\CC})_{x_j}\}$,
be parabolic structures on $E_i$, $i=0,1$, defined along $D$.
We say that $(E_0,\{Q_{0,i}\})$ and $(E_1,\{Q_{1,i}\})$ are {\bf meromorphically
equivalent} if, denoting by $\eE_0^{\std}$ and $\eE_1^{\std}$ the corresponding sheaves
of $\gG_{\alpha}^{\std}$-torsors, there is an isomorphism
$$\Psi:\eE_0\times_{\gG_{\alpha}^{\std}}\gG_{\alpha}
\to \eE_1\times_{\gG_{\alpha}^{\std}}\gG_{\alpha}$$ of
sheaves of $\gG_{\alpha}$-torsors. We call $\Psi$ a
{\bf meromorphic equivalence} between $(E_0,\{Q_{0,i}\})$ and
$(E_1,\{Q_{1,i}\})$.
\end{definition}

A choice of meromorphic equivalence between
$(E_0,\{Q_{0,i}\})$ and $(E_1,\{Q_{1,i}\})$
is the same thing as an
isomorphism of principal bundles
$$\psi:E_0|_{X^*}\to E_1|_{X^*}$$
satisfying the following property, for every $j$. Let $R_{i,j}\subset E_{i,x_j}$
be defined as at the end of the previous subsection. Let $U\subset X$ be a small
disk centred at $x_j$ and disjoint from all other points
of the support of $D$. Let $\sigma_i\in\Gamma(U,E_i)$, $i=0,1$,
be holomorphic sections satisfying $\sigma_i(x_j)\in R_{i,j}$.
Consider the holomorphic map $f:U\setminus\{x_j\}\to H^{\CC}$ defined by the condition
$$\psi(\sigma_0(y))=\sigma_1(y)f(y)\qquad\text{for every $y\in U\setminus\{x_j\}$}.$$
Then
$f\in\gG_{\alpha}=\gG_{\alpha}^{\wall}$.

\subsection{Hecke transformations}
\label{ss:Hecke-transformations}

Consider for each $j$
an element $\lambda_j\in \imag\tlie$ such that
$2\pi\imag\lambda\in\Lambda_{\cochar}$ (see Appendix
\ref{alcoves})
and let $\lambda=(\lambda_1,\dots,\lambda_r)$.
In this subsection we define
the Hecke transformation $\tau_{\lambda}$, which is
a 
natural 1---1 correspondence between
sheaves of $\gG_{\alpha}$-torsors
and sheaves of $\gG_{\alpha+\lambda}$-torsors.
%

Choose an isomorphism $T\simeq (S^1)^k$ and let
$\theta_i:S^1\to H$ be the composition of the inclusion of the
$i$-th factor $S^1\hookrightarrow T$ with the inclusion
$T\hookrightarrow H$. Let $\chi_i\in\imag\Lambda_{\cochar}$ be
defined by the condition that $\theta_i(e^{2\pi\imag
u})=\exp(u\chi_i)$ and write
$$\lambda_j=\sum_i \frac{\imag}{2\pi} a_{ji}\chi_i,$$
where $a_{ji}\in\ZZ$. Consider the invertible sheaf
$\lL_i=\oO(\sum_j a_{ji}x_j)$, and let $\sigma_i\in
H^0(X,\lL_i)$ satisfy $\sigma_i^{-1}(0)=\sum_j a_{ji}x_j$. Let
$\lL_i^*\subset\lL_i$ be the subsheaf of local nowhere
vanishing sections and let $\lL^*=\prod_i\lL_i^*$. Let $\oO^*$
be the sheaf of local nowhere vanishing holomorphic functions
on $X$. Note that $\lL^*$ is in a natural way a sheaf of
$(\oO^*)^k$-torsors.

Denote the complexification of $\theta_i$ by $\theta_i^{\CC}:\CC^*\to H^{\CC}$.
The morphisms $(\theta_1^{\CC},\dots,\theta_k^{\CC})$ give a monomorphism
of sheaves of groups
$$\Theta:(\oO^*)^k\to\gG$$
(recall that $\gG$ is the sheaf on $X$ of local holomorphic maps to $\HC$).

Fix some $\alpha=(\alpha_1,\dots,\alpha_r)$ with arbitrary components $\alpha_j\in\imag\tlie$.
Define the sheaf $\lL^*\otimes_{\Ad}\gG_{\alpha}$ to be the quotient of
$\lL^*\times\gG_{\alpha}$ by the relation that identifies
$$((l_1\psi_1,\dots,l_k\psi_k),\phi) \sim ((l_1,\dots,l_k),\Theta(\psi_1,\dots,\psi_k)\phi
\Theta(\psi_1,\dots,\psi_k)^{-1})$$ for every $x\in X$ and
elements $l_i\in(\lL_i^*)_x$, $\psi_i\in\oO^*_x$ and
$\phi\in(\gG_{\alpha})_x$ (the subindex $x$ denotes here as
usual the stalk over $x$). We next define a structure of sheaf
of groups on $\lL^*\otimes_{\Ad}\gG_{\alpha}$. It suffices to
describe the structure at the level of stalks. Given $x\in X$
and two elements
$\zeta,\zeta'\in(\lL^*\otimes_{\Ad}\gG_{\alpha})_x$ one can
take representatives of $\zeta$ and $\zeta'$ of the form
$\zeta=[((l_1,\dots,l_k),\phi)]$ and
$\zeta'=[((l_1,\dots,l_k),\phi')]$, and then we set
$\zeta\zeta'=[((l_1,\dots,l_k),\phi\phi')].$ This operation is
well defined and endows $\lL^*\otimes_{\Ad}\gG_{\alpha}$ with
the structure of sheaf of groups.

\begin{lemma}
\label{lemma:Hecke-sheaves-groups} The sheaf of groups
$\lL^*\otimes_{\Ad}\gG_{\alpha}$ is isomorphic to
$\gG_{\alpha+\lambda}$.
\end{lemma}
\begin{proof}
We construct a morphism of sheaves of groups
$\Xi:\lL^*\otimes_{\Ad}\gG_{\alpha}\to \gG_{\alpha+\lambda}$ at
the level of stalks. Given $\zeta=[((l_1,\dots,l_k),\phi)]\in
(\lL^*\otimes_{\Ad}\gG_{\alpha})_x$ consider the expression
$$\mu=\Theta(l_1\sigma_1^{-1},\dots,l_k\sigma_k^{-1})
\cdot\phi\cdot
\Theta(l_1\sigma_1^{-1},\dots,l_k\sigma_k^{-1})^{-1}.$$ Taking
into account that each $l_i\sigma_i^{-1}$ is a germ of
meromorphic function at $x$, we can view $\mu$ as a germ of
meromorphic map from a neighbourhood $U$ of $x$ to $\HC$. We
claim that $\mu\in(\gG_{\alpha+\lambda})_x$. If
$x\notin\{x_1,\dots,x_r\}$ then this is obvious, since $\mu$ is
actually holomorphic. Suppose now that $x=x_j$. Let $z=z_j$ be
the local holomorphic coordinate near $x$ chosen in Subsection
\ref{ss:parahoric}. One checks
that 
$\Theta(z^{-a_{j1}},\dots,z^{-a_{jk}})=\exp((\ln
z)\lambda_j)=z^{\lambda_j}.$ Now:
\begin{align*}
\mu\in(\gG_{\alpha+\lambda})_x & \quad \Longleftrightarrow \quad
\Theta(z^{-a_{j1}},\dots,z^{-a_{jk}})\cdot\phi\cdot \Theta(z^{-a_{j1}},\dots,z^{-a_{jk}})^{-1} \in(\gG_{\alpha+\lambda})_x\\
&\quad \Longleftrightarrow \quad z^{\lambda_j}\phi z^{-\lambda_j} \in(\gG_{\alpha+\lambda})_x\\
&\quad \Longleftrightarrow \quad
|z_j|^{-(\alpha_j+\lambda_j)}z^{\lambda_j}\phi z^{-\lambda_j}|z_j|^{\alpha_j+\lambda_j}
\quad\text{uniformly bounded near $x$}\\
&\quad \Longleftrightarrow \quad
|z_j|^{-\alpha_j}\phi |z_j|^{\alpha_j}
\quad\text{uniformly bounded near $x$}\\
&\quad \Longleftrightarrow \quad
\phi\in(\gG_{\alpha})_x.
\end{align*}
This proves that setting $\Xi(\zeta)=\mu$ defines a morphism of
sheaves from $\lL^*\otimes_{\Ad}\gG_{\alpha}$ to
$\gG_{\alpha+\lambda}$. One checks easily that $\Xi$ is an
isomorphism of sheaves of groups.
\end{proof}

Now let $\eE$ be a parahoric bundle of weight $\alpha$. We
define its $\lambda$-Hecke transformation to be
$$\tau_{\lambda}(\eE)=\eE\otimes_{(\oO^*)^k}\lL^*,$$
where $\eE\otimes_{(\oO^*)^k}\lL^*$ denotes the quotient of
$\eE\times\lL^*$ by the relation that identifies, at any point
$x$, $$(\epsilon,(l_1\psi_1,\dots,l_k\psi_k))\sim
(\epsilon\cdot\Theta(\psi_1,\dots,\psi_k)^{-1},(l_1,\dots,l_k))$$
for every $\epsilon\in\eE_x$, $l_i\in(\lL_i^*)_x$ and
$\psi_i\in\oO^*_x$. This makes sense because
$\Theta(\psi_1,\dots,\psi_k)\in(\gG_{\alpha})_x$ (in fact this
holds for every $\alpha$), and $\eE$ is a sheaf of torsors over
$\gG_{\alpha}$. Now we define on $\eE\otimes_{(\oO^*)^k}\lL^*$
a structure of sheaf of
$\lL^*\otimes_{\Ad}\gG_{\alpha}$-torsors, by the condition that
$$(\epsilon,(l_1,\dots,l_k))\cdot((l_1,\dots,l_k),\phi)=(\epsilon\cdot\phi,(l_1,\dots,l_k)).$$
It is easy to prove that $\tau_{-\lambda}(\tau_{\lambda}(\eE))$
is naturally isomorphic to $\eE$. Taking into account Lemma
\ref{lemma:Hecke-sheaves-groups}, we obtain the following.

\begin{theorem}
\label{thm:Hecke} The $\lambda$-Hecke transformation
$\tau_{\lambda}$ establishes a natural 1--1 correspondence
between parahoric bundles of weight $\alpha$ and parahoric
bundles of weight $\alpha+\lambda$.
\end{theorem}

\subsection{Meromorphic maps to $H^{\CC}$}
\label{ss:meromorphic-maps}

Let $\Delta$ be the unit disc in $\CC$ centred at $0$, and let
$\Delta^*=\Delta\setminus\{0\}$. A holomorphic map
$g:\Delta^*\to H^{\CC}$ is said to be {\bf meromorphic} if for
any holomorphic morphism
$\rho:H^{\CC}\to\SL_N\CC \subset \End\CC^N$ there exists some
integer $k$ such that
$$\Delta^*\ni z\mapsto z^k\rho(g(z))\in \End\CC^N$$
extends to a holomorphic map $\Delta\to\End\CC^N$. By Riemann's
extension theorem the sections of the sheaves $\gG_{\alpha}$
are meromorphic in the previous sense.

%

\begin{lemma}\label{reduction-0}
  Let $P\subset H^\CC$ be a parabolic
  subgroup.  Let $\sigma: \Delta \lra H^\CC/P$ be a holomorphic map and let
  $g:\Delta^*\lra H^\CC$ be a meromorphic map at $0$. Then the map
  $g\cdot\sigma:\Delta^*\lra H^\CC/P$ defined by $g(z)\cdot\sigma(z)$ can be extended to a
  holomorphic map $\Delta\lra H^\CC/P$.
\end{lemma}
\begin{proof}
Choose an embedding $\xi:H\to\SU_N$ for some big $N$, and
denote by the same symbol both the holomorphic extension to the
complexifications $\xi:H^{\CC}\to\SL_N\CC$ and the induced
morphism $\xi:\hlie^{\CC}\to\liesl_N\CC$ of Lie algebras.
Suppose that $P=P_{\beta}$ for some $\beta\in\imag\hlie$ and
let $P'=P_{\xi(\beta)}\subset\SL_N \CC$ be the corresponding
parabolic subgroup. Identifying $H^\CC$ with $\xi(H^{\CC})$ we
may write $P=P'\cap H^{\CC}$, which implies that the inclusion
$\xi:H^{\CC}\hookrightarrow \SL_N\CC$ induces a (holomorphic)
inclusion $H^{\CC}/P\hookrightarrow \SL_N\CC/P'$ with closed
image. Hence it suffices to consider the case
$H^{\CC}=\SL_N\CC$. Then $\SL_N\CC/P'$ is a partial flag
variety. Via the Pl\"ucker embedding, we may further reduce our
statement to the following one: if $g:\Delta^*\lra \SL_N\CC$
is a map whose composition with the inclusion
$\SL_N\CC\hookrightarrow\End\CC^N$ is meromorphic, if
$\sigma:\Delta\to\PP^{N'}$ is a holomorphic map, and if there
is a linear action of $\SL_N\CC$ on $\PP^{N'}$, then
$g\cdot\sigma:\Delta^*\to \PP^{N'}$ extends to a holomorphic
map $\Delta\to\PP^{N'}$. This last statement follows from an
immediate computation, so the proof of the lemma is complete.
\end{proof}
%

This has the following immediate consequence.

\begin{lemma}
\label{reduction} Suppose that two parabolic principal
$H^{\CC}$-bundles $(E,\{Q_i\})$ and $(E',\{Q_i'\})$ are
meromorphically equivalent in the sense of Definition
\ref{def:meromorphic-equivalence}. Then for any parabolic
subgroup $P\subset H^{\CC}$ the holomorphic reductions of the
structure groups $\Gamma(E/P)$ and $\Gamma(E'/P)$ are in
bijection. The bijection is given by extending across $D$ the
natural holomorphic isomorphism $E|_{X\setminus D}\to
E'|_{X\setminus D}$ which exists in each of the two cases.
\end{lemma}

This statement can be summarized by saying that $P$-reductions are not affected by meromorphic equivalence of bundles. It it interesting to note that if $L\subset P$ is a Levi subgroup, the same is not true for $L$-reductions. More precisely, if $E$ and $E'$ are as in the lemma, then a $L$-reduction for $E$ may not give a $L$-reduction to $E'$ (it is easy to construct an explicit example on a $\SL_2\CC$-bundle). This is because $H^\CC/P$ is compact but $H^\CC/L$ is not so the reduction may `escape' at infinity.

\section{Parabolic $G$-Higgs bundles}
\label{parabolic-higgs-bundles}

\subsection{Definition of parabolic $G$-Higgs bundle}
\label{sec:defin-parab-g}

Following the definition and notation in Appendix \ref{reductive}, let
$G=(G,H,\theta,B)$ be a real reductive Lie group.  Let $X$ be a compact
connected Riemann surface and let $\{x_1,\cdots,x_r\}$ be a finite set of
different points of $X$. Let $D=x_1+\cdots +x_r$ be the corresponding
effective divisor.  Let $E$ be a parabolic principal $H^\setC$-bundle over
$(X,D)$.  Let $E(\liemc)$ be the bundle associated to $E$ via the
isotropy representation (see Appendix \ref{reductive}).

We now define the \textbf{sheaf $PE(\liemc)$ of parabolic sections of}
$E(\liemc)$ and the \textbf{sheaf $NE(\liemc)$ of strictly parabolic
  sections of $E(\liemc)$}. These consist of meromorphic sections of
$E(\liemc)$, holomorphic on $X\setminus D$, with singularities
of a certain type on $D$.  More precisely, choose a holomorphic
trivialization $e_i$ of $E$ near $x_i$ compatible with the
parabolic structure. Let $\alpha_i\in \imag\bar \eA$ be the
parabolic weight at $x_i$. In the trivialization $e_i$, we can
decompose the bundle $E(\liemc)$ under the eigenvalues of
$\ad(\alpha_i)$ (acting on $\liem^\CC$),
$$ E(\liemc)=\oplus_\mu\liemc_\mu . $$
Decompose accordingly a section $\varphi$ of $E(\liemc)$ as $\varphi=\sum \varphi_\mu$, then
we say that $\varphi$ is a section of the sheaf $PE(\liemc)$ (resp.
$NE(\liemc)$) if $\varphi$ is meromorphic at $x_i$, and $\varphi_\mu$ has order
\begin{equation}
  \label{eq:12}
  v(\varphi_\mu) \geq -\lfloor-\mu\rfloor  \quad (\text{resp. }v(\varphi_\mu)>-\lfloor-\mu\rfloor).
\end{equation}
This means that if $a-1< \mu\leq a$ (resp. $a-1\leq \mu<a$) for some integer
$a$, then $\varphi_\mu=O(z^a)$.

An equivalent way to define it is to say that a section of
$PE(\liemc)$ is a holomorphic section of the bundle
$$\oplus_\mu\liemc_\mu(\lfloor-\mu\rfloor x_i).$$ Of course, in general, if we take a
holomorphic bundle with some decomposition at a point, this
construction does not make sense, because the result depends on the
extension of the decomposition near the point. However, the following
lemma proves that, in our case, the definition does not depend on the
choice of the trivialization.

\begin{lemma}
  The action of  a section $g$ of the sheaf $PE(H^\setC)$ preserves the
  set of sections of the sheaves  $PE(\liemc)$ and $NE(\liemc)$.
\end{lemma}
\begin{proof}
  We write the proof only for $PE(\liemc)$.  The first case is that of
  a section $g=\exp(n/z)$ with $n\in \lieq_i^1$.  Then, for $\varphi\in E(\liemc)$,
  \begin{align*}
    \Ad(g)\varphi&=e^{\ad \frac nz}\varphi\\
    &=\varphi+[\frac nz,\varphi]+\frac{1}{2}[\frac nz,[\frac nz,\varphi]]+\cdots
  \end{align*}
  Since $[\alpha_i,n]=-n$, one has $[n,\liemc_\mu]\subset\liemc_{\mu-1}$, so
  $\Ad(g)\varphi$ satisfies (\ref{eq:12}) if $\varphi$ does.

  The second case is that of a constant $g\in Q_i$: it is clear that
  nothing is changed if $g$ belongs to the Levi subgroup $L_{\alpha_i}$
  (see Section \ref{sec:parabolic-subgroups-1}), so we can suppose that $g$
  belongs to the unipotent part of $Q_i$: let us write $g=\exp(n)$
  with
  \begin{equation}
    n\in \oplus_{\lambda<0}\liehc_\lambda,\label{eq:1}
  \end{equation}
  where $\liehc=\oplus_\lambda\liehc_\lambda$ is the eigenspace decomposition of
  $\liehc$ under the action of $\ad(\alpha_i)$. Then, as above,
  \begin{equation*}
    \Ad(g)\varphi =\varphi+[n,\varphi]+\frac{1}{2}[n,[n,\varphi]]+\cdots
  \end{equation*}
  Because of (\ref{eq:1}), if $\varphi\in \liemc_\mu$ then $[n,\varphi]\in
  \oplus_{\lambda<0}\liemc_{\lambda+\mu}$, and more generally $\Ad(g)\varphi-\varphi\in
  \oplus_{\mu'<\mu}\liemc_{\mu'}$, so that $\Ad(g)\varphi$ again satisfies
  (\ref{eq:12}).

  The third and last case consists in applying a holomorphic change of
  trivialization by a $g\in H^\setC$ such that $g(x_i)=1$. Let us write
  $g=\exp(zu)$ with $u\in\liehc$ holomorphic, and decompose $u=\oplus_\lambda u_\lambda$.
  Then, again,
  $$ \Ad(g)\varphi = \varphi + z[u,\varphi] + \frac{z^2}{2}[u,[u,\varphi]] + \cdots $$
  Here the important point is that all the eigenvalues $\lambda$ satisfy
  $|\lambda|\leq 1$, so that if $\varphi\in \liemc_\mu$, then $\ad(u)^k\varphi\in \oplus_{\mu'\leq
    \mu+k}\liemc_{\mu'}$ and $z^k\ad(u)^k\varphi$ again satisfies
  (\ref{eq:12}). This concludes the proof that the sheaf $PE(\liemc)$
  is well defined.
\end{proof}

\begin{remark}
\label{rmk:meromorphic-equivalence} In terms of meromorphic
equivalences (see Section
\ref{ss:parahoric}), the previous lemma can be
stated as follows. If $(E,\{Q_i\})$ and $(E',\{Q_i'\})$ are two
parabolic principal $H^{\CC}$-bundles, then any meromorphic
equivalence $\psi:E|_{X\setminus D}\to E'|_{X\setminus D}$
induces an isomorphisms of sheaves $PE(\mlie^{\CC})\to
PE'(\mlie^{\CC})$ and $NE(\mlie^{\CC})\to NE'(\mlie^{\CC})$.
\end{remark}

The sheaves $PE(\liemc)$ and $NE(\liemc)$ have a much simpler
description when $\alpha_i\in \imag\eA'_\lieg$, where

\begin{equation}\label{very-good-guys}
\eA'_\lieg=\{\alpha \in \bar\eA \;\;  \mbox{such that the eigenvalues $\lambda$
  of}\;\ad(\alpha) \; \mbox{on} \;  \lieg \; \mbox{satisfy}\;|\lambda|<1\}.
\end{equation}

So the eigenvalues of $\ad(\alpha)$ have
modulus smaller than 1, not only on $\lieh$, but on the whole $\lieg$,
and in particular on $\liem$ (one can often choose $\eA$ so that this
happens). To show this, consider for $\alpha\in \imag \lieh$ the subspaces
of $\liem^\C$ defined by
\begin{align*}
&\liem_{\alpha}=\{v\in \liem^\C\ :\ \Ad(e^{t\alpha})v
\text{ is bounded as}\;\; t\to\infty\}\\
&\liem^0_{\alpha}=\{v\in \liem^\C\ :\ \Ad(e^{t\alpha})v=v\;\;
\mbox{for every} \;\; t\}.
\end{align*}
We have that $\liem_\alpha^0\subset \liem_\alpha$ and we can choose a complement
$\lien_\alpha$ so that $\liem_\alpha =\liem_\alpha^0\oplus \lien_\alpha$

Recall that when $\alpha_i\in\imag \eA'$, the parabolic structure at $x_i$ is
given by a parabolic subgroup $Q_i\subset E(H^{\setC})_{x_i}$ isomorphic to
$P_{\alpha_i}$.  This determines an isomorphism of $E(\liem^\CC)_{x_i}$
with $\liem^\C$.  We can then define the subspaces $\liem_i$,
$\liem^0_i$ and $\lien_i$ of $E(\liem^\CC)_{x_i}$ corresponding to
$\liem_{\alpha_i}$, $\liem_{\alpha_i}^0$ and $\lien_{\alpha_i}$, respectively.
Then, when $\alpha_i\in\imag \eA'_\lieg$ the \textbf{sheaf $PE(\liemc)$ of
  parabolic sections of} $E(\liemc)$ is the sheaf of local holomorphic
sections $\psi$ of $E(\liemc)$ such that $\psi(x_i)\in \liem_i$. Similarly,
the \textbf{sheaf $NE(\liemc)$ of nilpotent sections of} $E(\liemc)$
is the sheaf of local holomorphic sections $\psi$ of $E(\liemc)$ such
that $\psi(x_i)\in \lien_i$.  We then have short exact sequences of sheaves
\begin{displaymath}
  0 \to PE(\liemc) \to E(\liemc) \to   \bigoplus_i
  E(\liemc)_{x_i}/\liem_i \to  0,
\end{displaymath}
and
\begin{displaymath}
  0 \to NE(\liemc) \to E(\liemc) \to   \bigoplus_i
  E(\liemc)_{x_i}/\lien_i \to  0.
\end{displaymath}

After these preliminaries, we can define a \textbf{parabolic $G$-Higgs
  bundle} to be a pair of the form $(E,\varphi)$, where $E$ is a parabolic
$\HC$-principal bundle over $(X,D)$ and $\varphi$ --- the
\textbf{Higgs field} --- is a holomorphic section
of $PE(\liemc)\otimes K(D)$. We shall say that $(E,\varphi)$ is
\textbf{strictly parabolic} if in addition $\varphi$ is a section of
$NE(\liemc)\otimes K(D)$.

\begin{remark}\label{general-pairs} It is worth point out that these definitions can be extended to more general
Higgs pairs, where we replace $\liem^\C$ by an arbitrary representation
of $H^\C$ (see \cite{garcia-prada-gothen-mundet:2009a}).
\end{remark}

We now define the residue of $\varphi$ at the points $x_i$. This is again
much simpler if the weights $\alpha_i\in \imag\eA'_\lieg$.  The Higgs field $\varphi$ is
then a meromorphic section of $E(\liemc)\otimes K$ with a simple pole at
$x_i\in D$ and the \textbf{residue} of $\varphi$ at $x_i$ is hence an element
$$
\Res_{x_i}\varphi\in \liem_i.
$$
We denote the projection of $\Res_{x_i}\varphi$ in $\liem^0_i$
by $\Gr\Res_{x_i}\varphi$.  The space $\liem^0_i$ is invariant
under the action of $L_i\subset Q_i$, the subgroup
corresponding to the Levi subgroup $L_{\alpha_i}\subset
P_{\alpha_i}$. As we will see, the orbit of this projection
under $L_i$  is what is relevant in relation to local systems
and the construction of the appropriate moduli space.

More generally, if $\alpha_i\in \imag\bar\eA$, $\varphi$ is a
section of $PE(\liemc)\otimes K(D)$, with
$PE(\liemc)\simeq\oplus_\mu\liemc_\mu(\lfloor-\mu \rfloor x_i)$
in a neighbourhood of $x_i$. Choosing a holomorphic
trivialization of $\cO(D)$ at the point $x_i$, we can identify
the fibre of $PE(\liemc)$ at $x_i$ with $\liemc$, and then
project the residue of $\varphi$ at $x_i$ to the space
$\widetilde{\liem^0_i}\subset E(\liem^\C)_{x_i}$ corresponding
to
\begin{equation}\label{centralizer}
\widetilde{\liem^0_{\alpha_i}}:=\ker_{\liem^\C}(\Ad(\exp{2\pi\imag \alpha_i})-1).
\end{equation}

 We denote again  this projection by
$\Gr\Res_{x_i}\varphi$. (In the case where $\alpha_i\in \imag\eA'_{\lieg}$,
$\widetilde{\liem^0_i}= \liem^0_i$, and this is just
the projection on the Levi part $\liem_i^0$ as mentioned above.)
Now what will be relevant is the orbit of $\Gr\Res_{x_i}\varphi$ under the
group $\widetilde{L_i}$ corresponding
to
\begin{equation}\label{stabilizer}
\widetilde{L_{\alpha_i}}:= \Stab_{H^\C}(\exp{2\pi\imag \alpha_i})
\end{equation}
under the isomorphism of $Q_i$ with $P_{\alpha_i}$ given by the parabolic
structure of $E$.

In concrete terms, if we have a local coordinate $z$ near $x_i$ and a
holomorphic trivialization of $E$ near $x_i$, we can write
$$ \varphi = \sum_{\mu} \frac{\varphi_\mu}{z^{\lfloor-\mu\rfloor}} \frac{dz}{z} , $$
with $\varphi_\mu$ holomorphic, and then
$$\Gr\Res_{x_i}\varphi=\sum_{\mu\in \setZ}\varphi_\mu(0).$$
If we change the coordinate, $z'=fz$, then
$$ \varphi = \sum_{\mu} f^{\lfloor-\mu\rfloor}
\frac{\varphi_\mu}{(z')^{\lfloor-\mu\rfloor}} \left(\frac{dz'}{z'}-\frac{df}{f}\right) , $$
and $\sum_{\mu\in \setZ}\varphi_\mu(0)$ is changed into
$$ \sum_{\mu\in \setZ}f(0)^{-\mu} \varphi_\mu(0) = \Ad(e^{-\alpha_i\ln f(0)})
\sum_{\mu\in \setZ}\varphi_\mu(0) , $$ for any choice of logarithm of $f(0)$. So we
deduce that $\Gr\Res_{x_i}\varphi$ is well defined up to the action of the
1-parameter group generated by $\alpha_i$. Note that this ambiguity exists
only in the case where $\ad(\alpha_i)$ has non zero integer eigenvalues on
$\liemc$. In particular, the orbit of $\Gr\Res_{x_i}\varphi$ is well-defined
under the action of $\widetilde{L_i}$.

We must also verify that the definition of $\Gr\Res_{x_i}\varphi$ does not
depend on the choice of a gauge $g\in PE(H^\setC)$. The only significant
case is that of a $g(z)=\exp \frac nz$ with $n\in \lieq_i^1$. Then
$$ \Ad(g(z))\varphi = \varphi+\frac 1z[n,\varphi]+\frac 1{2z^2}[n,[n,\varphi]]+\cdots $$
Since $[n,\liemc_\mu]\subset\liemc_{\mu-1}$, it follows that $\Gr\Res_{x_i}\varphi$ is
transformed into $$\Ad(e^n)\Gr\Res_{x_i}\varphi.$$ The other cases are left to
the reader.

Let us now see what the definition means in three simple cases. The
first case is the one of Example \ref{ex:defin-parab-princ}, that is
$G=\GL_n\setC$. Here a $G$-Higgs bundle is just an ordinary Higgs bundle.
In that case, the ambiguity on $\Gr\Res_{x_i}\varphi$ does not show up. The
Higgs field $\varphi$ is a meromorphic 1-form with simple poles at the $x_i$
such that $\Res_{x_i}\varphi$ preserves the flag, and $\Gr\Res_{x_i}\varphi$ is
the endomorphism induced by $\Res_{x_i}\varphi$ on the associated graded
space (hence our notation $\Gr\Res_{x_i}\varphi$).

The second example, in which the ambiguity shows up, is $G=\SL_2\setR$ and
$H=U_1$. A $G$-Higgs bundle in this case is given by an $H^\setC$-bundle
$E$, which is equivalent to the data of the  line bundle $L=E(\setC)$, and a Higgs field $\varphi$
which is a
1-form with values in $E(\liemc)=L^2\oplus L^{-2}$. The parabolic structure
at a point $x$ is given by a weight $\alpha$ which we can take in the
interval $(-\frac12,\frac12]$. The eigenvalues of $\ad(\alpha)$ on $\liemc$
are $\pm 2\alpha$, and integer eigenvalues $\pm 1$ appear only for
$\alpha=\frac12$. Let us examine this case more carefully. If we represent
$\alpha\in \imag\lieu_1$ as the matrix
\begin{equation}
 \alpha=\begin{pmatrix} \frac12 & 0 \\ 0 & -\frac12 \end{pmatrix} ,\label{eq:30}
 \end{equation}
then we obtain, for holomorphic $\varphi_\pm$,
\begin{equation}
  \label{eq:2}
  \varphi = \begin{pmatrix} 0 & z\varphi_+ \\ \frac1z\varphi_- & 0 \end{pmatrix}\frac{dz}z,
\quad \Gr\Res_x\varphi=\begin{pmatrix} 0 & \varphi_+(0) \\ \varphi_-(0) & 0 \end{pmatrix},
\end{equation}
and $\Gr\Res_{x_i}\varphi$ is defined up to
$$\begin{pmatrix} 0 & \varphi_+(0) \\ \varphi_-(0) & 0 \end{pmatrix}
\longrightarrow
\begin{pmatrix} 0 & c\varphi_+(0) \\ c^{-1}\varphi_-(0) & 0 \end{pmatrix}.
$$
Of course there is another way to consider the same object: we can
think of $(E,\varphi)$ as a $\GL_2\setC$-Higgs bundle, the underlying holomorphic
bundle of which is $L\oplus L^{-1}$, with parabolic weights
$(\frac12,-\frac12)$.  Actually, to get weights in a correct interval,
it is better to consider $L\oplus L^{-1}(D)$, with weights
$(\frac12,\frac12)$.  If $(e_1,e_2)$ is a trivialization of $L\oplus
L^{-1}$, then $(e_1,z^{-1}e_2)$ is a trivialization of $L\oplus L^{-1}(D)$
and the Higgs field (\ref{eq:2}) becomes
$$\varphi = \begin{pmatrix} 0 & \varphi_+ \\ \varphi_- & 0 \end{pmatrix}\frac{dz}z,$$
so that $\Gr\Res_{x_i}\varphi$ coincides with the one obtained in
(\ref{eq:2}), with the same ambiguity coming from the choice of
a trivialization of $\cO(D)$.

The third and last example is just our second example considered for
the group $G=\SL_2\setC$, so $H=\SU_2$. Then $H^\setC=\SL_2\setC$ and
$\liemc=\lies\liel_2\setC$. The difference is now that the weight
(\ref{eq:30}) has nonzero integer eigenvalues on $\liehc$ itself, so
one must consider gauge transformations which are meromorphic:
$$ g =
\begin{pmatrix}
  O(1) & O(z) \\ O(\frac 1z) & O(1)
\end{pmatrix}.
$$
This gauge transformation becomes holomorphic in the $\GL_2\setC$-gauge
$(e_1,z^{-1}e_2)$ considered above, and all the data of the $G$-Higgs
bundle is equivalent to that of the $\GL_2\setC$-Higgs bundle obtained
after this Hecke transformation. Nevertheless, it is useful to have a
general definition which does not require to change the group.

\subsection{Stability of parabolic $G$-Higgs bundles}\label{stability}

The notion of stability, semistability and polystability of a
parabolic $G$-Higgs bundle depends on an element of $\imag\liez$,
where $\liez$ is the centre of $\hlie$.  We will develop the theory
here for any element in $\imag\zlie$.  However, in order to relate
parabolic $G$-Higgs bundles to $G$-local systems, one requires this
element to lie also in the centre of $\lieg$, and actually to be
$0$. This is always the case in particular if $G$ is semisimple.

Let $s\in \imag \lieh$. We consider the parabolic subgroup $P_s$ of
$H^\C$, and the corresponding Levi subgroup $L_s$ as defined in
Section \ref{sec:parabolic-subgroups-1}.  Let $\chi_s$ be the corresponding
antidominant character of $\liep_s$, where $\liep_s$ is the Lie
algebra of $P_s$.  We consider

\begin{align*}
&\liem_{s}=\{v\in \liem^\C\ :\ \Ad(e^{ts})v
\text{ is bounded as}\;\; t\to\infty\}\\
&\liem^0_{s}=\{v\in \liem^\C\ :\ \Ad(e^{ts})v=v\;\;
\mbox{for every} \;\; t\}.
\end{align*}

One has that $\liem_s$ is invariant under the action of $P_s$ and
$\liem^0_s$ is invariant under the action of $L_s$. If $G$ is complex,
$\liem^\C=\mathfrak{g}$ and hence the isotropy representation $\iota$
coincides with the adjoint representation, then
$\liem_s=\mathfrak{p}_s$ and $\liem^0_s=\mathfrak{l}_s$.

We need to consider the subalgebra of $\lieh$ defined by
\begin{equation}
\lieh_0=\lieh\cap(\cap_{\chi \text{ character of }\lieg} \ker \chi),\label{eq:44}
\end{equation}
or, equivalently, by $\lieh_0=(\lieh\cap \liez(\lieg))^\perp$.

Let $(E,\varphi)$ be a parabolic $G$-Higgs bundle over $(X,D)$.  Let $\liez$
be the Lie algebra of $Z(H)$, and let $c\in\imag\liez$. We say that
$(E,\varphi)$ is $c$-\textbf{semistable} if for every $s\in\imag \lieh$ and
any holomorphic reduction of the structure group of $E$ to $P_s$, $\sigma$,
such that $\varphi|_{X\setminus D}\in H^0(X\setminus D, E_{\sigma}(\liem_s)\otimes K)$, where $E_\sigma$ is
the principal $P_s$-bundle obtained from the reduction $\sigma$, we have
\begin{equation}\label{stability-condition}
\pardeg E(\sigma,\chi_s)-\la c,s\ra\geq 0.
\end{equation}
If we always have strict inequality for $s\in\imag \lieh_0$ we say that
$(E,\varphi)$ is $c$-\textbf{stable}.  The Higgs bundle $(E,\varphi)$
is $c$-\textbf{polystable} if it is semistable, and if equality occurs
for some $s\in\imag \lieh_0$ and $\sigma$, than there is a meromorphically equivalent Higgs bundle $(E',\varphi)$ and a further holomorphic reduction
$\sigma_{L_s}$ of the structure group to $L_s$, so that
\begin{enumerate}
\item $\varphi|_{X\setminus D}\in H^0(X\setminus D,E_{\sigma_{L_\sigma}}(\liem^0_s)\otimes K)$;
\item the bundle $E_{L_s}$ has a parabolic structure at the points of $D$ which is compatible with that of $E'$ in the sense that the parabolic bundle $E'$ is induced from $E_{L_s}$ through the injection $L_s\hookrightarrow H^\CC$ as in Remark \ref{rem:induced-par-str}; this implies in particular that the parabolic weights $\alpha_i$ of $E$ at the punctures actually lie in $\liel_s$.
\end{enumerate}
In this definition, in contrast to parabolic reductions, it is important to allow the reduction to the Levi to exist on a meromorphically equivalent bundle, see Lemma \ref{reduction} and the remarks following it. It then follows that our various stability conditions are invariant by meromorphic equivalence.

We will refer to $0$-stability simply as stability.

\begin{remark}
  If the weights are in $\imag \eA'_\lieg$ we can define the sheaf
  $PE_{\sigma}(\liem_s)$ of parabolic sections of $E_{\sigma}(\liem_s)$ as the
  sheaf of holomorphic sections $\psi$ of $E_\sigma(\liem_s)$ such that
  $\psi(x_i)\in \liem_i\cap E_{\sigma}(\liem_s)_{x_i}$, and require that $\varphi \in
  H^0(X,PE_{\sigma}(\liem_s)\otimes K(D))$.
\end{remark}

\begin{remark}
\label{integral-characters} Semi(stability) can be formulated
as above in terms of any parabolic subgroup $P\subset H^\C$
conjugated to a parabolic subgroup of the form $P_s$, and any
antidominant character $\chi$ of $\liep$, the Lie algebra of
$P$. Under certain conditions on the reductive structure of
$(G,H,\theta,B)$, which must be satisfied in order for the
moduli space to admit a GIT  construction (see \cite{schmitt}),
it is enough to check (\ref{stability-condition}) for
antidominant characters $\chi$ that lift to a character
$\widetilde{\chi}:P\to \C^*$. In this situation
$\deg(E)(\sigma,\widetilde{\chi})$ is the degree of the line
bundle associated to $E$ via $\widetilde{\chi}:P \to \C^*$, and
hence an integer. This is satisfied, in particular if $G$ is
complex. The set of characters of $P$ is an abelian free group,
and the subset of antidominant characters is a cone $A$ inside
of this group. The characters for which the Higgs field
$\varphi$ satisfies $\varphi|_{X\setminus D}\in H^0(X\setminus
D, E_{\sigma}(\liem_s)\otimes K)$ is a subcone $B\subset A$. It
is hence enough to check the numerical condition for the
elements of the $1$-dimensional faces of $B$, and hence for a
finite number of antidominant characters of $P$. If $G$ is
complex, these antidominant characters define characters of
maximal parabolic subgroups (not merely subalgebras).
\end{remark}

\section{Hitchin--Kobayashi correspondence }
\label{hitchin-kobayashi}
Let $X$ be a compact Riemann surface and let $D$ a divisor of
$X$ as above. Choose a smooth 2-form $\omega$ on $X\setminus
D$. Suppose either that $\omega$ extends smoothly across $D$ or
that it blows up near $D$ less rapidly than the Poincar\'e
metric in $X\setminus D$ (given by \ref{poincare}). In any case
we assume that $\int \omega=2\pi$. Let $G=(G,H,\theta,B)$ be a
real reductive Lie group. Let $(E,\varphi)$ be a parabolic
$G$-Higgs bundle on $(X,D)$.  Let $c\in\imag\liez$, as in
Section \ref{stability}. We are looking for a metric $h$ on $E$
outside the divisor $D$, i.e. $h \in \Gamma(X\setminus
D,E(H\backslash H^\setC))$, satisfying the
$c$-Hermite--Einstein equation:
\begin{equation}
 R(h)-[\varphi,\tau_h(\varphi)] +\imag c\omega =0,\label{eq:11}
\end{equation}
where $R(h)$ is the curvature of the unique connection $A(h)$
compatible with the holomorphic structure of $E$
and the metric $h$, and $\tau_h$ is the conjugation on
$\Omega^{1,0}(E(\liem^\C))$
defined by combining the metric $h$  and the standard conjugation on $X$ from
$(1,0)$-forms to $(0,1)$-forms. We shall denote
$$ F(h)=R(h)-[\varphi,\tau_h(\varphi)]+\imag c\omega. $$

\subsection{Initial metric}
\label{sec:initial-metric}

We begin by constructing a singular metric $h_0$ which gives an
approximate solution to the equations. This metric is
$\alpha$-adapted, in the sense of Section \ref{ss:analytic-global-degree}.

We first construct a model metric near each singular point $x_i$.
We decompose $\Gr\Res_{x_i}\varphi$ into its semisimple part and its
nilpotent part,
\begin{equation}
  \label{eq:3}
  \Gr\Res_{x_i}\varphi=s_i+Y_i .
\end{equation}

Let $e_i\in E_{x_i}$ be an element belonging to the $P_{\alpha_i}$
orbit specified by the parabolic structure. Choose a local
holomorphic coordinate $z$, and extend the trivialization $e_i$
into a holomorphic trivialization of $E$ around $x_i$, so we can
identify locally the metric with a map into $H\backslash H^\setC$. If $Y_i=0$,
then the model metric is
\begin{equation}
  \label{eq:4}
  h_0 = |z|^{-2\alpha_i} (=e^{-2\alpha_i\ln |z|}) .
\end{equation}
If we change the trivialization by a gauge transformation $g\in
\Gamma(X,PE(H^\setC))$, the resulting metric will remain at
bounded distance of $h_0$ in $H\backslash H^\setC$, so $h_0$ defines a
\emph{quasi-isometry} class of metrics on $E$ near $x_i$.

If $Y_i\neq 0$, then consider the reductive subalgebra
\begin{equation}
\lier=\ker(\Ad(e^{2\pi\imag \alpha_i})-1)\cap\ker(\ad
s_i)\label{eq:20}
\end{equation}
of $\lieg^\setC$. Recall that $\Gr\Res_{x_i}\varphi$ is the
projection of the residue of $\varphi$ to
$\widetilde{\liem^0_i}$ the centralizer of $e^{2\pi\imag \alpha_i}$,
defined in (\ref{centralizer}), and hence $Y_i$ belongs to $\lier$. We can
complete $Y_i$ into a Kostant--Rallis $\liesl_2$-triple $(H_i,X_i,Y_i)$
(see Appendix \ref{triples}) such that
\begin{equation}
  \label{eq:5}
  H_i\in \liehc\cap \lier, \quad X_i,Y_i\in \liem^\setC\cap \lier .
\end{equation}
Moreover, maybe after conjugation (which means that we change the
trivialization), we can assume that
\begin{equation}
  \label{eq:6}
  H_i \in \imag\lieh, \quad X_i=-\tau(Y_i),
\end{equation}
where $\tau$ is the conjugation of $\lieg^\C$ with respect to a compact
form so that $\theta$ and $\tau$ commute.
Now the Higgs field has the form
\begin{equation}
  \label{eq:31}
  \varphi = \big(s_i + \Ad(z^{\alpha_i}) (Y_i)+ \psi \big) \frac{dz}z, \quad\psi\in NE(\liem^\setC).
\end{equation}
This is well defined because $Y_i\in \ker(\Ad(e^{2\pi\imag\alpha_i})-1)$
so that $Y_i$ decomposes on eigenspaces of $\ad \alpha_i$ corresponding to
the integral eigenvalues. The model for the metric is
\begin{equation}
  \label{eq:7}
  h_0 = |z|^{-\alpha_i} (-\ln|z|^2)^{\Ad(e^{\imag \theta \alpha_i})H_i} |z|^{-\alpha_i} .
\end{equation}
Again this is well defined because $H_i\in \ker(\Ad(e^{2\pi\imag\alpha_i})-1)$. In the case where $\alpha_i\in\eA'$ (that is all eigenvalues of $\ad \alpha_i$ in $\lieh^\setC$ have modulus less than 1) then $[\alpha_i,H_i]=0$ and the formula simplifies to $h_0= |z|^{-2\alpha_i} (-\ln|z|^2)^{H_i}$. The general formula (\ref{eq:7}) is obtained from the gauge $z^{\alpha_i}$, which exists only on the universal cover of the punctured disc: in that gauge the Higgs field writes $(s_i+Y_i)\frac{dz}z$ and the metric $h_0$ is simply $(-\ln|z|^2)^{H_i}$.

Again, a different choice of trivialization leads to a metric which
remains quasi-isometric to $h_0$. Also note that $h_0$ is
$\alpha_i$-adapted in the sense of definition \ref{def:par-structure}.

Now extend the local model to some global metric $h_0$ on $E$. The
quasi-isometry class of $h_0$ is well-defined.
We shall now prove that the metric $h_0$ gives an approximate
solution of the Hermite--Einstein equation near the marked point
$x_i$. Instead of working in the holomorphic gauge $e$ of $E$, we
choose the unitary gauge
\begin{equation}
\label{eq:unitary-holomorphic-gauge}
\be=eg_0, \quad g_0=|z|^{\alpha_i}(-\ln|z|^2)^{-\Ad(e^{\imag \theta \alpha_i})H_i/2} ,
\end{equation}
so that the $\dbar$-operator of $E$  can be written locally as
\begin{equation}
\label{eq:d-bar-E}
 \dbar^E = \dbar + g_0^{-1}\dbar g_0 =
\dbar + \big(\alpha_i-\frac{\Ad(e^{\imag \theta \alpha_i})H_i}{\ln|z|^2}\big)
\frac{d\zbar}{2\zbar}
\end{equation}
and the associated $H$-connection is
\begin{equation}
 A_0 = d - \imag \big(\alpha_i-\frac{\Ad(e^{\imag \theta \alpha_i})H_i}{\ln|z|^2}\big) d\theta .
\label{eq:14}
\end{equation}

Here by a $\dbar$-operator on $E$ we mean a holomorphic structure on the  underlying  smooth $H^\C$-bundle. The space of holomorphic structures on this  smooth bundle is an affine space modelled on  the space of $(0,1)$-forms with values in the adjoint bundle $E(\liehc)$. The $\dbar$ appearing in (\ref{eq:d-bar-E}) represents an origin in this
affine space. After choosing a faithful representation of $H^\C$, one obtains a true $\dbar$-operator on the associated vector bundle.

On the other hand, still in the unitary gauge $\be$, we obtain the
expression for the Higgs field by the action of $\Ad(g_0^{-1})$,
which gives
\begin{equation}
 \varphi = \big(s_i-\frac{\Ad(e^{\imag\theta\alpha_i})Y_i}{\ln|z|^2}\big)\frac{dz}{z} + O(|z|^\epsilon) \frac{dz}{z}.\label{eq:15}
\end{equation}
As in formula (\ref{eq:31}) this is well defined.
Therefore we obtain, using (\ref{eq:6}),
\begin{equation}
\begin{split}
   R(h_0) &= \Ad(e^{\imag \theta \alpha_i})H_i \frac{dz\land d\zbar}{|z|^2(\ln|z|^2)^2} , \\
[\varphi,\tau_{h_0}(\varphi)] &= \Ad(e^{\imag \theta \alpha_i})H_i \frac{dz\land
d\zbar}{|z|^2(\ln|z|^2)^2}
           + O(|z|^\epsilon)\frac{dz\land d\zbar}{|z|^2} ,
\end{split}\label{eq:13}
\end{equation}
and we clearly have an approximate solution of the
Hermite--Einstein equation (\ref{eq:11}).

We can regard the unitary gauge $\be$ as
defining  a different extension of $E$ near the punctures, which is
a unitary extension. The resulting $H$-principal bundle on $X$
will be denoted $\bE$.

\subsection{The correspondence}
\label{sec:correspondence}

We are now in a position to state the main theorem in this section. First
remark that if $\chi$ is a character of $G$, then $\chi([\varphi,\theta(\varphi)])=0$ and
therefore the Hermite--Einstein equation (\ref{eq:11}) implies $\pardeg_\chi
E=\chi(c)$.  In particular, if $c=0$, $\pardeg_\chi E=0$.  It is important to
note that this is no longer true in general for a character of $H$ alone,
so we cannot conclude that the total parabolic degree of $E$ must
vanish. This justifies the topological condition in the following theorem.
\begin{theorem}\label{th:correspondence-1}
  Let $(E,\varphi)$ be a parabolic $G$-Higgs bundle, equipped with an adapted
  initial metric $h_0$. Let $c\in \imag\,\liez(\lieh)$ such that $\pardeg_\chi E=\chi(c)$ for all characters
  $\chi$ of $\glie$.  Then $(E,\varphi)$ admits a $c$-Hermite--Einstein metric $h$
  quasi-isometric to $h_0$ and $\alpha_i$-adapted at each puncture $x_i$, if and only if $(E,\varphi)$ is $c$-polystable.
  Moreover, any two such Hermite-Einstein metrics are related by an automorphism of $(E,\varphi)$, that is, an element of $PE(H^\C)$ fixing $\varphi$.
\end{theorem}
It seems difficult to reduce Theorem \ref{th:correspondence-1} to the
theorem of Simpson \cite{Sim90} for the case $G=\GL_n\setC$ by taking a faithful
representation, since in particular it is not clear how the stability
conditions would relate. Instead, we prefer to give a direct proof by
checking that the proof in \cite{Biq97} still applies here.

We prove the Theorem in the next sections. For clarity, we restrict to
the case $c=0$ (it is well-known how to modify the proof to handle
nonzero $c$, see for example \cite{bradlow-garcia-prada-mundet:2003}),
and we will refer to $0$-polystability as polystability.

\subsection{The polystability condition is necessary}
\label{sec:polyst-cond}

Suppose that $s\in \sqrt{-1}\lieh$ and we have a holomorphic reduction $\sigma$ of the structure group of $E$ to the parabolic group $P_s$, such that
\begin{equation}
\varphi|_{X\setminus D}\in H^0(X\setminus D,E_\sigma(\liem_s)\otimes K).\label{eq:36}
\end{equation}
Using the operator $D''$ defined by $D''s=\overline{\partial}s+[\varphi,s]$, we can rewrite the formula in Lemma \ref{lemma:analytic-degree-parabolic} as
$$\pardeg_{\alpha}(E)(\sigma,\chi)=
\frac{\imag}{2\pi} \int_{X\setminus D}\langle R_h-[\varphi,\tau_h(\varphi)],s_{\sigma,h}\rangle -
  \langle\varpi(s_{\sigma,h})(D''s_{\sigma,h}),D''s_{\sigma,h}\rangle .$$
(The additional terms involving the Higgs field cancel out).
If $h$ is a Hermite-Einstein metric, then  the first term in the sum vanishes and there remains
$$\pardeg_{\alpha}(E)(\sigma,\chi)=
\frac{\imag}{2\pi} \int_{X\setminus D}  - \langle\varpi(s_{\sigma,h})(D''s_{\sigma,h}),D''s_{\sigma,h}\rangle \geq 0.$$
The inequality comes from the fact that for a holomorphic reduction satisfying (\ref{eq:36}), then $D''s_{\sigma,h}$ lives in the negative eigenspaces of $\ad s_{\sigma,h}$. Equality therefore occurs if and only if $D''s_{\sigma,h}=0$. Since $s\in \sqrt{-1}\lieh$, this implies that $s_{\sigma,h}$ is parallel: $$\nabla_hs_{\sigma,h}=[\varphi,s_{\sigma,h}]=[\tau_h(\varphi),s_{\sigma,h}]=0.$$
It is now clear that $s_{\sigma,h}$ induces a reduction of the Higgs bundle $(E,\varphi)$ to the Levi subgroup $L_s$ over $X\setminus D$, and there remains to understand the behavior at the punctures $x_i$.

Near $x_i$ we choose to write the equations in the bundle $E'$ meromorphically equivalent to $E$, such that the $\alpha_i$-adapted metric $h$ can be written $h=h_0\cdot |z|^{-\alpha_i}e^c$, where $c$ satisfies the conditions of Definition \ref{def:adapted_metric}. Since $s_{\sigma,h}$ is parallel, it has constant norm, which implies that in the trivialization the coefficients of $s_{\sigma,h}$ satisfy $\Ad(|z|^{-\alpha_i})s_{\sigma,h}=O(|z|^{-\epsilon})$ for every $\epsilon>0$. Since $s_{\sigma,h}=-\tau_h(s_{\sigma,h})$ we have as well $\Ad(|z|^{\alpha_i})s_{\sigma,h}=O(|z|^{-\epsilon}$, and since $s_{\sigma,h}$ is holomorphic, it follows that it extends at the puncture $x_i$, with $[\alpha_i,s_{\sigma,h}(x_i)]=0$ (in particular $\alpha_i\in \liel_s$). It follows that the $L_s$-reduction of $E|_{X\setminus D}$ extends to a $L_s$-reduction of $E'$, which is a parabolic bundle over $X$ with parabolic weights $\alpha_i$ at $x_i$. This finishes the proof of polystability.

Observe the role of meromorphic equivalence here: the metric $h$ selects a particular bundle $E'$ in the class of bundles meromorphically equivalent to $E$, and it is this bundle $E'$ which carries the reduction to the Levi subgroup.

\subsection{Preliminaries: functional spaces}
We now pass to the existence of the Hermite-Einstein metric. The basic idea to prove that polystability is sufficient is
common to a whole collection of results extending the original
Hitchin--Kobayashi correspondence on existence of
Hermite--Einstein metrics on holomorphic vector bundles. Recall
that when looking for Hermite--Einstein metrics one considers
the Donaldson functional $M(h,h')$, defined for two metrics $h$
and $h'$ on $E$ 
such that
\begin{align}
  M(h,h'')&=M(h,h')+M(h',h'') \label{eq:cocycle} \\
  \frac{d}{dt}M(h,he^{ts})\Big|_{t=0} &= \int_X h(\imag F(h),s)
  \label{eq:critical-points-solve-equation}
\end{align}
(see Section \ref{s:Donaldson-functional} below for some
details and references). In particular the critical points of
$M$ are the Hermite--Einstein metrics.
Roughly speaking, polystability enables to prove
$C^0$-convergence of a sequence of metrics $\{h_i\}$ such that
$M(h,h_i)$ converges to $\inf_{h'}M(h,h')$, and then
convergence in stronger functional spaces follows. Finally the
limit is proved to be a solution of the Hermite--Einstein
equation.

We shall not give full details of the argument in our situation
since the proof follows the one in \cite{Biq97}, and we refer
to this reference for details. We give only the general setup.

The equation to solve does not depend on the metric on the Riemann
surface. Because of the calculation (\ref{eq:13}), it is natural to work
with a cusp metric near the punctures, that is equal to
\begin{equation}\label{poincare}
ds^2=\frac{|dz|^2}{|z|^2\ln^2|z|^2}
\end{equation}
in some fixed local coordinate $z$ near each puncture. Writing
$z=|z|e^{\imag\theta}$ and $t=\ln(-\ln|z|^2)$, this can be written
$$ ds^2=dt^2+e^{-2t}d\theta^2 . $$
Extend $t$ by a smooth function in the interior of $X$. We define weighted
$C^0$ and $L^p$ spaces by
$$ C^0_\delta = e^{-\delta t}C^0 , \quad L^p_\delta = e^{-(\delta+\frac1p)t}L^p . $$
The curious choice for $L^p$ is motivated by the compatibility with $C^0$:
indeed, with this choice, we have $C^0_\delta\subset L^p_{\delta'}$ as soon as $\delta>\delta'$,
since $vol=e^{-t}dtd\theta$ near the punctures. The exponent $p$ is thought as
being very large---a replacement of $\infty$ because $C^k$ spaces are not
suitable for elliptic analysis.

Consider the $H$-connection $\nabla^+$ induced by $h_0$ on $E$, and define
$$
\nabla=\nabla^+ \oplus \ad(\varphi) : \bE(\liegc) \to
\Omega^1\otimes(\bE(\liegc)\oplus\bE(\liegc)) .
$$
Define now the weighted spaces $C^k_\delta$ (resp. the weighted Sobolev spaces
$L^{k,p}_\delta$) of sections $f$ of $\bE(\liegc)$ such that $\nabla^jf\in C^0_\delta$
(resp. $L^p_\delta$) for $j\leq k$. We will also use the refinement
$\hat{C}^k_\delta(E(\liegc))$ (resp.  $\hat{L}^{k,p}_\delta(\bE(\liegc))$) of
sections $f$ of $\bE(\liehc)$ such that $\nabla f\in C^{k-1}_\delta$ (resp.
$L^{k-1,p}_\delta$), but we ask nothing on $f$ itself.

Recall  from (\ref{eq:20}) the subalgebra
\begin{equation}
\lier_i=\ker(\Ad(e^{2\pi\imag   \alpha_i})-1)\cap \ker(\ad
s_i),\label{eq:22}
\end{equation}
and define furthermore $\lier_i'$ the commutator of the
$\mathfrak{sl}_2$-triple $(H_i,X_i,Y_i)$,
\begin{equation}
  \label{eq:21}
  \lier_i'=\ker(\ad H_i)\cap \ker(\ad X_i)\cap \ker(\ad Y_i) .
\end{equation}
For $\delta>0$, it is easy to see that, near a puncture, an element $f$ of
$\hat{C}^k_\delta$ (resp.  $\hat{L}^{k,p}_\delta(\lieg^\setC)$) can be decomposed as
$$ f = \Ad(e^{\imag\theta\alpha_i})f(0) + f_1 , \quad f(0)\in \lier_i\cap \lier_i', \quad f_1\in C^k_\delta
\text{ (resp. }L^{k,p}_\delta\text{).}$$
As before, this is well defined since
$f(0)\in\ker(\Ad(e^{2\pi\imag\alpha_i})-1)$. Therefore the elements of
$\hat{C}^2_\delta$ or $\hat{L}^{2,p}_\delta$ do not go to zero at the
punctures. Furthermore one checks easily that
$\hat{L}^{2,p}_\delta(\bE(\liegc))$ is a Lie algebra.

Finally the space of metrics in which we look for a solution of the
Hermite--Einstein equation is, for a small $\delta>0$ and a large $p$,
\begin{equation}
 \cH = \{h=h_0 e^s, \, s\in \hat{L}^{2,p}_\delta(\bE(\imag\lieh)) \} .\label{eq:27}
\end{equation}

From equation (\ref{eq:13}) it is clear that $F(h_0)\in L^p_\delta$ for any
$p$ and any $\delta>0$. We choose any fixed $\delta\in (0,1)$.

\subsection{Donaldson's functional}
\label{s:Donaldson-functional}
\newcommand{\Tr}{\operatorname {Tr}}

For a pair of metrics $h_0,h_0e^s\in\cH$ let
\begin{equation}\label{eq:variation-along-es}
M(h_0,h_0e^s) = \int_X\langle\imag \Lambda F(h_0),s\rangle_{h_0}  +\int_0^1 (1-t)\|\Ad(e^{ts/2})D''s\|_{L^2(X)}^2\,dt,
\end{equation}
where $D''s=\overline{\partial}^Es+[\varphi,s]$. Note that
$\overline{\partial}^Es$ belongs to $\Omega^{0,1}(E(\liehc))$
and $[\varphi,s]$ belongs to $\Omega^{1,0}(E(\liemc))$. Hence the two
summands are orthogonal and we can write
\begin{equation*}
\|\Ad(e^{ts/2})D''s\|_{L^2(X)}^2 = \|\Ad(e^{ts/2})\overline{\partial}^Es\|_{L^2(X)}^2   +
\|\Ad(e^{ts/2})[\varphi,s]\|_{L^2(X)}^2.
\end{equation*}
This allows to view $M$ as a particular case of the functionals
defined in \cite{mundet:2000,bradlow-garcia-prada-mundet:2003}
using a symplectic point of view (they are instances of integrals
of a moment map), see in particular \cite[(4.10)]{mundet:2000}. The arguments given in [op.cit.]
imply that $M$ satisfies the cocycle condition (\ref{eq:cocycle})
and property (\ref{eq:critical-points-solve-equation}).

The functional $M$ is also an analogue of Donaldson's functional
considered by Simpson in \cite{Sim88}. Indeed, we have the
following formula
$$\int_0^1 (1-t)\|\Ad(e^{ts/2})D''s\|_{L^2(X)}^2\,dt=\int_X\big(\psi(s)(D''s),D''s\big)_{h_0},$$ where we apply here the
notation introduced in Section \ref{ss:analytic-global-degree} to the
adjoint and the isotropy representations, and extend it to
sections of twists of $E(\liehc)$ and $E(\liemc)$, for the
function $\psi(t)=(e^{t}-t-1)/t^2$. This follows by decomposing $D''s$ on eigenspaces of the adjoint action of $s$, and from the calculation $\int_0^1(1-t)e^{t\lambda}dt = \psi(\lambda)$ for any eigenvalue $\lambda$ of $\ad s$. Also, $(.,.)$ is distinguished from $\langle.,.\rangle$: the latter is the Hermitian product on the bundle only, while the former also includes the Hermitian product on forms, so the result is a scalar (which is implicitly integrated against the volume form).

Hence Donaldson's functional can be rewritten in the following form
\begin{equation}
M(h_0,h_0e^s)=\int_X\langle\imag \Lambda F(h_0),s\rangle_{h_0}+\int_X\big(\psi(s)(D''s),D''s\big)_{h_0},\label{eq:42}
\end{equation}
which makes evident the relation to Simpson's definition.


\subsection{Reduction to the stable and semisimple case}
\label{sec:reduct-stable-semis}

Before solving the equation, we make two reductions.

 \emph{Jordan--H\"older reduction.}
Using similar arguments to the ones used in  the non-parabolic case \cite{garcia-prada-gothen-mundet:2009a},
one can prove  that if a $G$-Higgs bundle $(E,\varphi)$
is polystable there is a canonical reduction $(E',\varphi')$
of structure group to a $G'$-Higgs bundle, up to isomorphism, such that the reduced Higgs bundle
is stable. This is done by iterating the process, mentioned  in definition of polystability, of reduction
to a Levi given by a parabolic reduction for which the degree vanishes.
Moreover, this reduction also satisfies $\pardeg_\chi E=0$ for all characters of $\lieg'$. (Recall that we are in the case $c=0$).

Therefore, to prove existence of the Hermite-Einstein metric, we can suppose that $(E,\varphi)$ is actually stable.

\emph{Central part.}
The second step is to reduce to the semisimple part of $G$. Indeed, as is well-known, the Hermite-Einstein equation (\ref{eq:11}) decouples on the direct sum
\[ \lieh = (\liez(\lieg)\cap\lieh) \oplus \lieh_0, \] where $\lieh_0$ was defined in (\ref{eq:44}), into
\begin{align}
  \chi(R(h)) & = 0 \text{ for all characters $\chi$ of $\lieg$},\label{eq:37} \\
  R(h)_{\lieh_0}-[\varphi,\tau_h(\varphi)] & =0.\label{eq:38}
\end{align}
We have a trivial subbundle $E(\liez(\lieg))\subset E(\lieh)$, and for a section $s$ of $E(\imag\liez(\lieg))$ one has $\chi(R(h e^s))=\chi(R(h))+\dbar \partial \chi(s)$. Therefore, starting from the initial metric $h_0$, the equation (\ref{eq:37}) is achieved by $h=h_0 e^s$ if one can solve the equation, for all characters $\chi$ of $\lieg$,
\[ \Delta \chi(s) = 2i \Lambda \chi(R(h)). \]
This is just the Laplace equation on the trivial bundle $\imag\liez(\lieg)$. Since by hypothesis $\int \chi(R(h))=0$ for all characters $\chi$ of $\lieg$, this equation is solvable in the space $\hat L^{2,p}_\delta$ since $F(h_0)\in L^p_\delta$. (The non uniqueness of the solution of the Hermite-Einstein equation comes from the non uniqueness of such $s$, since one can add any constant section of $\imag\liez(\lieg)$).

Finally we can suppose that we are in the case of a stable bundle $(E,\varphi)$, and that our initial metric $h_0$ satisfies (\ref{eq:37}). We look for a solution of (\ref{eq:38}) of the form $h=h_0 e^s$ with $s$ a section of $E(\imag \lieh_0)$.

\subsection{Proof that polystability is sufficient}
The method consists in minimizing the Donaldson functional
$M(h_0,h)$ for $h\in \cH$ under the constraint
$\|F(h)\|_{L^{2,p}_\delta}\leq B$ for some large $B$. The
technical results in \cite{Biq97} reduce this problem to
obtaining a $C^0$-estimate on a minimizing sequence $h_i$,
which in turn is a consequence of stability, as we shall now see.

At the end, the
solution $h$ satisfies an elliptic equation, so has additional
regularity:
\begin{equation}
  \label{eq:25}
  h\in \cH^\infty = \{h_0 e^s,
  \, s\in \hat{C}^\infty_\delta(\bE(\imag\lieh_0)) \} .
\end{equation}
Actually it can be shown that there is a stronger decay of the
components of $s$ which are orthogonal to $\lier_i$.

Take some big $B$ and define
$$\cS^{\infty}(B)=\{s\in \hat{C}^\infty_\delta(\bE(\imag\lieh_0)),\,
 \,  \|F(h_0e^s)\|_{L^p_\delta}\leq B\}.$$
Assume that there do not exist any constants $C,C'$ such that
$$\sup |s|\leq C+C'M(h_0,h_0e^s)\qquad\qquad
\text{for any $s\in \cS^{\infty}(B)$}.$$ The bound on the
curvature implies  that
$$\sup|s|\leq \text{const.}(1+\|s\|_{L^1})$$
for some constant depending on $B$ (see Lemma 8.4 in
\cite{Biq97}). Hence our assumption implies that there neither exist
constants $C,C'$ such that
$$\|s\|_{L^1}\leq C+C'M(h_0,h_0e^s)\qquad\qquad
\text{for any $s \in \cS^{\infty}(B)$}.$$ It follows that we can take a
sequence $\{s_i\}\subset\cS^{\infty}(B)$ and positive real numbers $C_i\to\infty$ such
that
$$|s_i|_{L^1}\to\infty,\qquad\qquad
C_iM(h_0,h_0e^{s_i})\leq |s_i|_{L^1}=:\lambda_i.$$ Define $u_i=\lambda_i^{-1}s_i$, so
that $\|u_i\|_{L^1}=1$. Now the arguments in Section 5 of \cite{Sim88} imply
that, up to passing to a subsequence, the sections $u_i$ converge
(weakly and locally in $L^{1,2}$) to a nonzero and locally $L^{1,2}$
section $u_{\infty}$ of $\bE(\imag\lieh_0)|_{X\setminus D}$. This section satisfies
the following properties:
\begin{itemize}
\item the $L^2$ norm of $D''u_{\infty}$ is finite, and the projection of $D''u_\infty$ on eigenspaces of $\ad u_\infty$ corresponding to nonnegative eigenvalues of $\ad u_\infty$ is zero;
\item there exist locally $L^{1,2}$ orthogonal projections $(\pi_j)_{j=1...k}$ of $\End(E(\liehc))$, satisfying $[\pi_i,\pi_j]=0$,
  such that $\ad(u_{\infty})=\sum l_j\pi_j$, where $l_1<\dots<l_k$ are the eigenvalues of $\ad u_\infty$ on $E(\liehc)$;
\item for each $j$ let $\Pi^j=\pi_1+\dots+\pi_j$. Then $\Pi^j\circ\Pi^j=\Pi^j$ and $(1-\Pi^j)D''\Pi^j=0$.
\end{itemize}
The regularity result of Uhlenbeck and Yau \cite{uhlenbeck-yau} for locally $L^{1,2}$
subbundles implies that the image of $\Pi^j$ is a holomorphic subbundle
$E^j\subset E(\liehc)|_{X\setminus D}$ (see Proposition 5.8 in \cite{Sim88}). This
implies that $u_{\infty}$ is a smooth section of $\bE(\imag\lieh)|_{X\setminus D}\subset
E(\liehc)|_{X\setminus D}$, since $h_0$ is smooth on $X\setminus D$.

\newcommand{\qu}{/\kern-.7ex/}

\begin{lemma}
Take any $x,y\in X\setminus D$, choose trivializations of $\bE_x$
and $\bE_y$, and use them to identify $\bE(\liehc)_x$ and
$\bE(\liehc)_y$ with $\liehc$. Then the elements
$u_{\infty}(x),u_{\infty}(y)\in \liehc$ are conjugate.
\end{lemma}
\begin{proof}
Denote by $\liehc\qu H^{\mathbb C}$ the affine quotient of the
adjoint action. Since semisimple elements have closed adjoint
orbits, two semisimple elements $x,y\in \liehc$ are conjugate if
and only if their images by the projection map $\liehc\to
\liehc\qu H^{\mathbb C}$ coincide. The later is equivalent to
showing that for any invariant polynomial $p\in
\Hom_{\CC}(S^*\liehc,\CC)^{H^{\mathbb C}}$ we have $p(x)=p(y)$.
Let $p$ be any such invariant polynomial. If $u\in\imag\liehc$ and $n\in\liehc$ is concentrated on eigenspaces for negative eigenvalues of $\ad u$, then $p(u+n)=p(\Ad(tu)(u+n))=p(u+\Ad(tu)n) \rightarrow_{t\rightarrow+\infty} p(u)$, so $p(u+n)=p(u)$ and $d_up(n)=0$. Applying this to $u_\infty$ and $D''u_\infty$, we obtain
\begin{equation}
  \dbar p(u_\infty) = d_{u_\infty}p(D''u_\infty) = 0.
\end{equation}
Hence $p(u_{\infty})$ is a holomorphic function. On the other hand using the isomorphisms
\begin{align*}
\Hom_{\CC}(S^*\liehc,\CC)^{H^{\mathbb C}} &\simeq
\Hom_{\CC}(S^*\liehc,\CC)^{H} \qquad \text{($H\subset H^{\CC}$ is
Zariski dense)}\\
&\simeq \Hom_{\setR}(S^*(\imag\lieh),\setR)^{H}\otimes_{\setR}\CC
\end{align*}
we may write $p=a+\imag b$, where $a(\imag\lieh)\subset\setR$ and
$b(\imag\lieh)\subset\setR$ and both $a$ and $b$ are $\HC$ invariant. Then
$a(u_{\infty})$ and $b(u_{\infty})$ are holomorphic functions and since $u_{\infty}$
is a section of $\bE(\imag\lieh)$ the function $a(u_{\infty})$ (resp.
$b(u_{\infty})$) is real (resp. imaginary) valued. Hence both $a(u_{\infty})$
and $b(u_{\infty})$ vanish. We have thus proved that $p(u_{\infty})$ is constant for any
invariant polynomial $p$. Since $u_{\infty}$ takes semisimple values, the
lemma is proved.
\end{proof}

Let $x\in X\setminus D$ be any point, take a trivialization of $\bE_x$, use it
to identify $\bE(\imag\lieh)_x$ with $\imag\lieh$, and let
$s\in\imag\lieh_0$ be the element corresponding to $u_{\infty}(x)$. Then the
section $u_{\infty}$ defines a reduction $\sigma$ of the structure group of
$E|_{X\setminus D}$ to the parabolic subgroup $P=P_s$. Furthermore, this
section is holomorphic, because the bundles $E^j\subset E(\liehc)$ are
holomorphic.

\begin{lemma}\label{lemma:extension}
  The reduction $\sigma$ extends into a holomorphic reduction of $E$ to $P$
  on the whole $X$.
\end{lemma}
\begin{proof}
Denote by $\plie_s$ the Lie algebra of $P_s$.
Let $\Gr$ denote the Grassmannian variety parametrizing $\dim\plie_s$-dimensional
subspaces of $\hlie^{\CC}$. The adjoint representation $H^{\CC}\to\GL(\hlie^{\CC})$
induces an action of $H^{\CC}$ on $\Gr$. By \cite[Proposition 7.83]{knapp} the
stabilizer of $\plie_s\in\Gr$ is equal to $P_s$, so there is a unique
$H^{\CC}$-equivariant map
$\iota_0:F=H^{\CC}/P_s\to\Gr$ sending $P_s$ to $\plie_s$,
and the $H^{\CC}$-orbit through
$\plie_s$ can be identified with the image of $\iota_0$. Let
$\iota:E(F)\to E(\Gr)$ be the map induced by $\iota_0$.

We know that the $L^2$ norm of $D''(u_{\infty})=\ov{\partial}^Eu_{\infty}+[\phi,u_{\infty}]$
is finite. Since, as already pointed out, $\ov{\partial}^Eu_{\infty}$ and $[\phi,u_{\infty}]$
are orthogonal, it follows that the $L^2$ norm of $\ov{\partial}^Eu_{\infty}$ is finite.
Denote as before $\ad(u_{\infty})=\sum l_j\pi_j$. Since each projector $\pi_j$ can be
expressed as a polynomial of $u_{\infty}$, and $u_{\infty}$ is bounded (because the eigenvalues
$l_j$ are constant), the boundedness of $\|\ov{\partial}^Eu_{\infty}\|_{L^2}$
implies that the $L^2$ norm of $\ov{\partial}^E\pi_j$ is finite.
This $L^2$ norm is computed with respect to the Hermitian
metric $h_0'$ on $E(\hlie^{\CC})|_{X\setminus D}$
induced by $h_0$.

Let $E^j$ be the image of $\Pi^j=\pi_1+\dots+\pi_j$. Since $(1-\Pi^j)D''\Pi^j=0$
the same orthogonality argument as before implies that
$(1-\Pi^j)\ov{\partial}^E\Pi^j=0$, from which it follows
that $E^j$ is a holomorphic subbundle of
$E(\hlie^{\CC})|_{X\setminus D}$. We claim that we can apply \cite[Lemma 10.6]{Sim88}
to $E^j$ and thus get a holomorphic extension of $E^j$ as a holomorphic subbundle
$\ov{E}^j$ of $E(\hlie^{\CC})$ defined over the entire $X$.
Indeed, the requirements
for \cite[Lemma 10.6]{Sim88} to apply are that $h_0'$ has polynomial growth
when expressed in terms of a local trivialisation of $E(\hlie^{\CC})$ in a neighborhood
of $D$, and that the curvature of $h_0'$ has bounded $L^1$ norm. Now the first property
follows from (\ref{eq:7}), and the second one follows from the first formula in (\ref{eq:13}).

Since $u_{\infty}$ is a nowhere vanishing section of $E(\hlie^{\CC})$ and $\ad(u_{\infty})(u_{\infty})=0$ there exists some $m$ such that $l_m=0$.
Then the fibers of $E^m\subset E(\hlie^{\CC})|_{X\setminus D}$ can be identified
with Lie algebras isomorphic to $\plie_s$, so $E^m$ defines a holomorphic section
$\rho_0$ of $E(\Gr)$ along $X\setminus D$ which actually comes from a section $\rho_1$
of $E(F)$ along $X\setminus D$. Hence, $\rho_0=\iota\circ\rho_1$.
The fact that $E^m$ extends to a holomorphic subbundle $\ov{E}^m$ of $E(\hlie^{\CC})$
implies that $\rho_0$ can be extended to a holomorphic section $\ov{\rho}_0$ of
$E(\Gr)$. Now, $F$ is closed in $\Gr$ (because $H^{\CC}/P_s$ is compact, precisely since $P_s$
is a parabolic subgroup of $H^{\CC}$), and consequently there exists a holomorphic extension
$\ov{\rho}_1$ of $\rho_1$ satisfying $\ov{\rho}_0=\iota\circ\ov{\rho}_1$.

For any $y\in X$ let $Q_y:=\{a\in E(H^{\CC})_y\mid a\cdot \ov{E}_y^m\subseteq \ov{E}_y^m\}$.
Then $Q_y$ is conjugate to $P_s$ for any trivialization of $E_y$ (because the normalizer of
$P_s$ is $P_s$ itself) and $Q=\bigcup_{y\in X}Q_y$ is a holomorphic subbundle of $E(H^{\CC})$.
By construction, $Q$ is a holomorphic extension of $P$ to the whole $X$.
\end{proof}

The element $s\in\imag\lieh_0$ which was used to define $P=P_s$ also
provides a strictly antidominant character $\chi:\liep\to\CC$, see appendix
\ref{sec:parabolic-subgroups-1}. The same arguments as in Lemma 5.4 in
\cite{Sim88} allow to prove that $\varphi\in H^0(X\setminus D,E_\sigma(\mlie_s)\otimes K)$, and
that the section $u_{\infty}$ satisfies (recall that $D''u_{\infty}$ is concentrated on negative eigenspaces of $\ad u_\infty$)
\begin{equation}
\begin{split}
  \imag \int_{X\setminus D}\langle\Lambda F(h_0)& ,u_{\infty}\rangle
  - \langle\varpi(u_{\infty})(\overline{\partial}u_{\infty}),\overline{\partial}u_{\infty}\rangle \\
& \leq \int_{X\setminus D}\langle\imag \Lambda F(h_0),u_{\infty}\rangle
-\big(\varpi(u_{\infty})(D''u_{\infty}),D''u_{\infty}\big) \\ & \leq 0.
\end{split}\label{eq:41}
\end{equation}
The function $\varpi$ appears because the rescaling $u_i=\lambda_i^{-1}s_i$ implies to replace in the Donaldson functional the function $\psi(t)$ by the rescaled function $\lambda_i\psi(\lambda_it)$ which converge when $\lambda_i\rightarrow \infty$ to the function $-\varpi(t)=-t^{-1}$ for $t<0$ and $+\infty$ for $t\geq 0$.

This contradicts stability since by  Lemma
\ref{lemma:analytic-degree-parabolic} the first expression in (\ref{eq:41})
 is $$2\pi\pardeg(E)(\sigma,\chi).$$

\subsection{Uniqueness up to automorphism}
Suppose that $h$ and $he^s\in\hH$ are two critical points of the Donaldson
functional $M(h_0,\cdot)$. Then (\ref{eq:cocycle}) and (\ref{eq:variation-along-es})
imply that $M(h,he^{ts})$ is a constant function of $t$.
Here we need to use the fact that $M(h,he^{ts})$ is a convex
function of $t$, which follows from (\ref{eq:variation-along-es}).
Now formulas (\ref{eq:cocycle}), (\ref{eq:critical-points-solve-equation})
and (\ref{eq:variation-along-es}) imply that $D''s=0$, and therefore
$s$ is an element of $PE(\liehc)$ and $[\varphi,s]=0$. Finally $e^s\in PE(H^\C)$ and stabilizes $\varphi$.

\section{Parabolic local systems}
\label{parabolic-local-systems}

In this section we take $c=0$ and we refer to $c$-(poly)stability
of a parabolic $G$-Higgs bundle simply as (poly)stability.

\subsection{From Higgs bundles to parabolic local systems}

Let $(E,\varphi)$ be a polystable parabolic $G$-Higgs bundle
with $\pardeg_\chi E=0$ for all characters $\chi$ of $G$. By
Theorem \ref{th:correspondence-1}, we get an Hermite--Einstein
metric $h$ on $E$, which is quasi-isometric to some given
adapted metric $h_0$. Equation \ref{eq:11} simply means that
$$ D=A(h)+\varphi-\tau_h(\varphi) $$
is a flat $G$-connection on the $G$-bundle obtained by extending
the structure group of the $H$-bundle given by the
Hermite--Einstein
metric. We therefore obtain a $G$-local system on
$X':=X\setminus\{x_1,\cdots,x_r\}$. Recall that a $G$-local system on a manifold can be
equivalently seen as a representation of the fundamental group
of the manifold in $G$, a $G$-bundle over the manifold equipped
with a flat $G$-connection, or a $G$-bundle with locally constant
transition functions.

Let us now examine the behaviour of the local system near the puncture.
Let us see that just on the local model: from
(\ref{eq:14}) and (\ref{eq:15}), we deduce that, in the orthonormal
frame $\be$,
\begin{equation}
  \label{eq:16}
  \begin{split}
    D = d &+ \big(s_i-\tau(s_i)-\Ad(e^{\imag\theta\alpha_i})\frac{Y_i+X_i}{\ln r^2}\big) \frac{dr}{r}\\
    &+ \imag\big(-\alpha_i+s_i+\tau(s_i)-\Ad(e^{\imag\theta\alpha_i})\frac{Y_i-H_i-X_i}{\ln r^2}\big)d\theta.
  \end{split}
\end{equation}
Observe that $Y_i-H_i-X_i=\Ad(e^{-X_i})Y_i$ is nilpotent. The monodromy
around $x_i$ is
\begin{equation}
  \label{eq:17}
  e^{2\pi\imag \alpha_i}e^{2\pi(-s_i-\tau(s_i)+Y_i-H_i-X_i)}.
\end{equation}
In this formula the monodromy appears as the product of two commuting
elements of $G$ (the first is compact, the second is non compact). The
$\frac{dr}r$ term has also an interpretation: taking
\begin{equation}
  \label{eq:18}
  f=\be r^{-s_i+\tau(s_i)}(-\ln r^2)^{\frac12\Ad(\imag \theta\alpha_i)(X_i+Y_i)},
\end{equation}
we get a $G$-trivialization, which is parallel along rays from the origin.
The metric along these parallel rays has the form
\begin{equation}
  \label{eq:19}
  h = r^{2(-s_i+\tau(s_i))}(-\ln r^2)^{\Ad(\imag \theta\alpha_i)(X_i+Y_i)} ,
\end{equation}
and in this trivialization the connection $D$ takes the form
\begin{equation}
  \label{eq:23}
  D=d-\imag\big(\alpha_i-s_i-\tau(s_i)+\Ad(\imag \theta\alpha_i)\frac{Y_i-H_i-X_i}2\big) d\theta .
\end{equation}
The logarithmic part of $h$ in (\ref{eq:19}) gives no new information,
because the triple $(H_i,X_i,Y_i)$ is already encoded in the unipotent part
$\exp(\imag \pi(Y_i-H_i-X_i))$ of the monodromy. But the semisimple part
$s_i-\tau(s_i)$ is an additional information: it gives the polynomial order of
growth of the harmonic metric $h$ near the point $x_i$, along parallel
rays. In the case $G=\GL_n\setC$, this additional structure transforms the local
system into a \textbf{filtered local system} in the sense of Simpson, that
is the fibre over the ray has a filtration with weights induced by
$s_i-\tau(s_i)$.

In our case, the metric on the ray is a map  to the symmetric
space of non compact type $G/H$, and $\exp((s_i-\tau(s_i))u)$, where
$u=-\ln r$, describes a geodesic in $G/H$, with some fixed speed,
depending on $s_i$. The corresponding geometric data is a point on the
geodesic boundary of $G/H$, and a positive real number describing the
(constant) speed of the geodesic. This data is equivalent to that of a
parabolic subgroup $P$ of $G$ with an antidominant character $\chi$ of
the Lie algebra of $P$. This leads to the following definition.

\begin{definition}\label{local-system}
Let $\beta_i\in \liem$ be semisimple and
$\beta=(\beta_1,\cdots, \beta_r)$.
  A \textbf{parabolic $G$-local system on} $X\setminus\{x_1,\cdots,x_r\}$
{\bf of weight} $\beta$ is defined by
  the  following data:
  \begin{enumerate}
  \item a $G$-local system $F$ on $X\setminus\{x_1,\cdots,x_r\}$,
  \item on a ray $\rho_i$ going to $x_i$, a choice of a parabolic
    subgroup $P_i$  of $F(G)|_{\rho_i}$ isomorphic to $P_{\beta_i}$,
invariant under the monodromy
    transformation around $x_i$, with a strictly antidominant character
$\chi_i$
    of the Lie algebra of $P_i$, where the $F_x(G)$ for $x\in \rho_i$ are
    identified by parallel transport.
  \end{enumerate}
\end{definition}

  Recall from Appendix \ref{sec:parabolic-subgroups-1}, that pairs
  $(P,\chi)$ consisting of a parabolic subgroup $P$ of $G$ and a strictly
  antidominant character of the Lie algebra of $\liep$ are in one-to-one
  correspondence with elements in $\liem$. In the definition we
take $\chi_i$ to be the character corresponding to $\chi_{\beta_i}$ under
the isomorphism of $F(G)|_{\rho_i}$ with $P_{\beta_i}$.
  Note that, in contrast with a parabolic bundle, at $x_i$ the
weights of the parabolic $G$-local system are not constrained to lie in
a Weyl alcove.

\begin{remark}\label{monodromy-action}
Note that a choice of $P_i\subset F(G)|_{\rho_i}$ isomorphic to $P_{\beta_i}$
is equivalent to choosing an orbit of the action of $P_{\beta_i}$ on
$F_{\rho_i}$, i.e. a point in the flag manifold $F_{\rho_i}/P_{\beta_i}$.
The second condition in the definition is asking  that this point be
fixed by the monodromy group, that is,  the image of the representation
 $\pi_1(X,x_i)\to G$ corresponding to the local system.
This condition  makes the data independent of the choice of $\rho_i$.
\end{remark}

Let us come back to the $G$-local system coming from a $G$-Higgs bundle
with Hermite--Einstein metric $h$. The formula (\ref{eq:19}) gives the
behaviour of $h$ for the model, but the formula remains valid for the
adapted metric $h_0$, up to lower order terms, and also for the
Hermite--Einstein metric $h\in \cH^\infty$, up to replacing $h_0$ by $h_0e^s$, for
some constant $s\in \lier_i\cap \lier_i'$. Thus we see that we have a well
defined parabolic structure induced at the point $x_i$, with the parabolic
subgroup and the character of its Lie algebra defined by $s_i-\tau(s_i)$.


\begin{proposition}\label{prop:filt-local-syst}
  Let $(E,\varphi)$ be a polystable parabolic $G$-Higgs bundle with $\pardeg_\chi
  E=0$ for all characters $\chi$ of $G$. Then the $G$-local system induced by the
  Hermite--Einstein metric constructed by Theorem \ref{th:correspondence-1}
  carries naturally a structure of parabolic $G$-local system.

  If $\alpha_i\in \imag \bar \eA $ is the parabolic weight of $E$ at
  $x_i$, and $\Gr\Res_{x_i}\varphi=s_i+Y_i$, then the weight
of the parabolic $G$-local system  is
  given by the element $s_i-\tau(s_i)\in \liem$, and the projection of the
  monodromy around $x_i$ on the Levi group defined by $s_i-\tau(s_i)$ is
  $$\exp(2\pi\imag \alpha_i)\exp(2\pi\imag(-s_i-\tau(s_i)+Y_i-H_i-X_i)).$$
\end{proposition}

To be precise, the compatibility between the metric and the parabolic
structure is the following: on a ray going to $x_i$, we compare the metric
$h$, seen as an application into $G/H$, with a geodesic given by the
$(Q_i,\chi_i)$, parametrized by $(-\ln r)$, and the condition is that the
distance between them should grow at most like $|\ln r|^N$.  For
$G=\GL_n\C$, this is the property referred by Simpson in \cite{Sim90} as
\textbf{tameness}.

\begin{remark}
  If the $s_i-\tau(s_i)$ part of $\Gr\Res_{x_i}\varphi$ vanishes, then we
  simply get a $G$-local system on $X-\{x_i\}$ with monodromy around
  $x_i$ given by
$$
\exp(2\pi\imag \alpha_i-s_i-\tau(s_i))\exp(2\pi\imag(Y_i-H_i-X_i)).
$$
\end{remark}

\begin{proof}[Proof of Proposition \ref{prop:filt-local-syst}]
  We have already seen the model behaviour. In general we have a perturbed
  flat connection $D+a$, where $D$ is the model (\ref{eq:16}) and the
  perturbation $a\in C_\delta^\infty$, which implies $a=a_r \frac{dr}r + a_\theta d\theta$, with
  $|a_r|, |a_\theta|=O( |\ln r|^{-1-\delta} )$. The radial trivialization
  (\ref{eq:18}) is then modified by a bounded transformation (this does not
  change the parabolic weight $s_i-\tau(s_i)$ of the model), while the formula
  (\ref{eq:23}) for the connection in this radial trivialization comes with
  an additional term, depending only on the angle $\theta$,
  $$ b(\theta) d\theta = \Ad(f^{-1}) a_\theta d\theta . $$
  The term $r^{-s_i+\tau(s_i)}$ in $f$ and the initial bound on $a_\theta$ imply
  the vanishing of the components of $b$ on the eigenspaces of
  $\ad(s_i-\tau(s_i))$ corresponding to the nonnegative eigenvalues. On the
  contrary, there can be a nonzero contribution from the eigenspaces for
  the negative eigenvalues, which is an additional unipotent term in the
  monodromy, in the unipotent subgroup associated to the parabolic
  subgroup. Therefore, only the Levi part is fixed and is equal to that of
  the model.
\end{proof}

The  converse of Proposition \ref{prop:filt-local-syst} is given in
the next section.

\subsection{Harmonic metrics and polystability of $G$-parabolic local systems}

Given a flat $G$-bundle $(F,D)$ over $X'$ and a metric $h$ on $F$, that is,
a reduction of structure group to an $H$-bundle, we decompose
$D=D^+_h+\psi_h$, where $D^+_h$ is an $H$-connection and $\psi_h$ is a section of
$\Omega^1\otimes E(\liem)$. The metric $h$ is said to be \textbf{harmonic} if
\begin{equation}\label{harmonic-eq}
(D^+_h)^*\psi_h=0.
\end{equation}

If we regard the flat $G$-bundle as a representation $\rho:\pi_1(X')\lra G$,
then a harmonic metric is the same as a harmonic map from the universal
cover of $X'$ to the symmetric space $G/H$, which is equivariant with
respect to the action of the fundamental group on both sides.

A harmonic metric on a parabolic $G$-local system is a harmonic metric on
the local system, which is tamed in the sense defined in the previous
section. The existence of a harmonic metric on the a parabolic $G$-local
system is governed, like for the Hermite--Einstein equation, by a stability
condition.  To define this we first define the degree. This is simpler than
for $G$-Higgs bundles, since the global term of the degree vanishes here
due to the flatness of the connection.

Let $F$ be a parabolic $G$-local system with $(P_i,\chi_i)$ defining the
parabolic structure at the point $x_i$. Let $Q\subset G$ be a parabolic subgroup
of $G$ and $\chi$ be a strictly antidominant character of its Lie algebra $\lieq$. Let
$\sigma$ be a reduction of structure group of $F$ to a $Q$-bundle, which is
constant (invariant under the flat connection).  Using the relative degree
of two parabolic subgroups with antidominant characters, we define the
parabolic degree of $F$ with respect to the reduction to $(Q,\chi)$ by the
formula
\begin{equation}
\label{eq:locsys-def-par-deg-sigma}
  \pardeg(F)(Q,\chi,\sigma):=-\sum_i \deg\big( (P_i,\chi_i),(Q,\chi)\big).
\end{equation}
This makes sense since both $P_i$ and $Q$ can be identified to subgroups of
$F(G)$ near the marked point $x_i$.

We say that $F$ is {\bf semistable} if for any such reduction of the local
system one has
\begin{equation}\label{parabolic-stability}
\pardeg(F)(Q,\chi,\sigma) \geq 0.
\end{equation}
It is {\bf stable} if the inequality is strict for any non-trivial
reduction, and {\bf polystable} if it is semistable and equality
happens in (\ref{parabolic-stability}) if and only if there is a
reduction of the local system to a Levi subgroup $L\subset Q$ (as for Higgs
bundles, this means that there is a parabolic $L$-local system which
induces $F$ through the inclusion $L\hookrightarrow G$). In particular this definition implies a compatibility of the parabolic structure with the reduction. When there is no parabolic
structure, the condition only says that there is no reduction of the
local system to a parabolic subgroup, unless there is a reduction to a
Levi subgroup: the local system is reductive.

\begin{remark}\label{simplification-stability}
Like in the case of $G$-bundles when $G$ is complex, in the case of parabolic
$G$-local systems (for real or complex $G$) it is enough to check
(\ref{eq:locsys-def-par-deg-sigma}) for characters that lift to $Q$. In fact,
it  suffices to verify the condition for maximal parabolic subgroups and a
certain character $\chi_Q$.
\end{remark}

\begin{theorem}\label{th:filt-local-syst}
  Let $F$ be a parabolic $G$-local system, with vanishing parabolic degree with respect to all characters of $\lieg$. Then $F$ admits a harmonic metric $h$ (compatible with the parabolic structure near the marked points) if and only if $F$ is polystable. Moreover, any two such harmonic metrics are related by a parallel automorphism of $F$, preserving the parabolic structure at the marked points (i.e. the automorphism belongs to the parabolic group $P_i\subset F(G)|_{\rho_i}$ on the chosen ray $\rho_i$ going to the marked point $x_i$).

  The harmonic metric induces a polystable parabolic $G$-Higgs bundle, and the relation between the weights at the marked points is the same as in Proposition \ref{prop:filt-local-syst}.
\end{theorem}

The proof of this theorem can be made formally similar to that of Theorem \ref{th:correspondence-1}, see \cite[Theorem 6]{Sim90} and \cite[Section 11]{Biq97}, by replacing the symbols $(D'',D'=\partial^E-\tau(\varphi))$ by $(D,D^c=\imag((D^+)^{0,1}+\psi^{1,0}-(D^+)^{1,0}-\psi^{0,1})$ and the curvature $F=(D''+D')^2$ by the pseudocurvature $G=-\frac12(D-\imag D^c)^2$, so we will not give the details of the proof, but just sketch a few steps.

The first step is to construct an initial metric $h_0$ (a section of
$E(G/H)$) near a puncture $x_i$: start with a trivialization $f$ where the
flat connection $D$ is given by formula (\ref{eq:23}), up to terms in the
nilpotent part of the parabolic, and define the initial metric $h_0$ by
formula (\ref{eq:19}). Implicit in this construction is the choice of
appropriate Kostant--Sekiguchi $\liesl_2$-triples (see Appendix
Section~\ref{triples}). In the orthonormal trivialization $\be=fh_0^{-\frac12}$,
the flat connection has then the form (\ref{eq:16}), up to $C_\delta^\infty$
terms. Choose any extension of $h_0$ in the interior of $X$.

The second step is to define the functional space of metrics that we want
to use: the relevant choice here is
\begin{equation}
  \label{eq:24}
  \cH = \{ h=h_0e^s, \, s\in \hat{L}^{2,p}_\delta(\liem) \} .
\end{equation}
As in Section \ref{sec:correspondence}, this space preserves the fact that $h_0$
can change at the points $x_i$, since $s(x_i)\in \lier_i\cap \lier_i'$.  Of
course, this is only a technical space needed for the proof, since at the
end, we shall get by local elliptic regularity, as in (\ref{eq:25}),
\begin{equation}
  \label{eq:26}
  h\in \cH^\infty=\{h_0e^s, \, s\in \hat{C}^\infty_\delta(\liem)\}.
\end{equation}

To solve the problem in $\cH$, one first solves the abelian equation on the central part, therefore reducing to the semisimple part of $G$. Then one minimizes the relative energy
\begin{equation}
 N(h_0,h) = \int_X ( |\psi_h|_h^2-|\psi_{h_0}|_{h_0}^2 )\label{eq:40}
 \end{equation}
on $\cH$, under the constraint $\|(D^+_h)^*\psi_h\|_{L^p_\delta}\leq B$ (note the inaccuracy in \cite[p. 88]{Biq97}, where the second term was forgotten). It turns out that using the formalism $(D,D^c)$, the functional $N(h_0,h)$ coincides with the Donaldson functional $M(h_0,h)$ in formula \eqref{eq:variation-along-es}. Of course the reason to introduce this relative energy is that the usual harmonic map energy $\int_X |\psi_h|^2$ can be infinite in our case, while $N(h_0,h)$ is well defined. The fact that $N(h_0,h)$ might be unbounded below explains why a stability condition appears here, replacing Corlette's semisimplicity condition for harmonic maps.

The coincidence beetween \eqref{eq:40} and the Donaldson functional can be proved just by checking that they have the same gradient, but since this fact does not seem well known, we also give a direct proof:
\begin{lemma}
  One has $N(h_0,h_0e^s)=\int_X -((D^+_{h_0})^*\psi_{h_0},s)_{h_0}+\frac12(\psi(s)(Ds),Ds)_{h_0}$, where $\psi(t)=\frac{e^t-t-1}{t^2}$.
\end{lemma}
\begin{proof}
  We use the formulas from \cite[Section 11]{Biq97}. For $h=h_0e^s$, one has
  \begin{align*}
    \psi_h &= -\frac12 h^{-1}Dh \\
        &= \Ad(e^{-s})\psi_{h_0}-\frac12 e^{-s}D(e^s) \\
    &= \Ad(e^{-s})\psi_{h_0} -\frac12  \tfrac{1-e^{-t}}t(s)(Ds) .
  \end{align*}
  Since $\Ad(e^{\frac s2})\psi_h$ is a section of $E(\liem)$ and $\ad(s)$ exchanges $\lieh$ and $\liem$, we deduce
  \begin{align*}
    \Ad(e^{\frac s2})\psi_h &= \Ad(e^{-\frac s2})\psi_{h_0} - \tfrac{\sinh\frac t2}t(s) (Ds)\\
    &= \cosh(\tfrac t2)(s)(\psi_{h_0}) - \tfrac{\sinh\frac t2}t(s) (D^+_{h_0}s),
  \end{align*}
  and, since $|\psi_h|_h^2=|\Ad(e^{\frac s2})\psi_h|_{h_0}^2$,
  \begin{multline}
   |\psi_h|_h^2-|\psi_{h_0}|_{h_0}^2 = \big(\sinh^2(\tfrac t2)(s)(\psi_{h_0}),\psi_{h_0}\big)
    + \big(\tfrac{\sinh^2\frac t2}{t^2}(s) (D^+_{h_0}s),D^+_{h_0}s\big)\\
    - 2 \big( \cosh(\tfrac t2)(s)(\psi_{h_0}),\tfrac{\sinh\frac t2}t(s) (D^+_{h_0}s) \big).\label{eq:43}
  \end{multline}
  On the other hand, decomposing $Ds=D^+_{h_0}s+[\psi_{h_0},s]$, and taking the even and odd parts of $\psi(t)$, one has
  \begin{multline*}
    \big( \psi(s)(Ds),Ds \big) = \big( (\cosh t-1)(s)\psi_{h_0},\psi_{h_0}\big) + \big( \tfrac{\cosh t-1}{t^2}(s)(D^+_{h_0}s),D^+_{h_0}s\big) \\ - 2 \big( \tfrac{\sinh t-t}t(s)(D^+_{h_0}s),\psi_{h_0}\big).
  \end{multline*}
  Comparing with \eqref{eq:43}, we obtain
  \[ \frac12 \big( \psi(s)(Ds),Ds \big) = |\psi_h|_h^2-|\psi_{h_0}|_{h_0}^2 + \big( D^+_{h_0}s,\psi_{h_0}\big).\]
The lemma follows.
\end{proof}

The heart of the proof of Theorem \ref{th:filt-local-syst} consists in proving that non convergence of a minimizing sequence would imply the existence of a reduction of $E$ to a parabolic subgroup $P$, appearing with an antidominant character, which would contradict stability, and this is done exactly as in the Higgs bundle case, except that the $L^{1,2}$ holomorphic subbundles are replaced by $L^{1,2}$ parallel subbundles, so no subtle regularity theory is needed here.

The harmonic metric $h$ induces a $G$-Higgs bundle $(E,\varphi)$, where the
$\dbar$-operator of $E$ is $(D^+)^{0,1}$ and the Higgs field
$\varphi=\psi_h^{1,0}$. A priori, this defines $(E,\varphi)$ only in the interior of $X$,
and we have to extend it over the points $x_i$. Fortunately, because of the
good control (\ref{eq:26}) on the solution, this can be done relatively
easily. Because $h\in \cH^\infty$, we have, in an orthonormal trivialization
$\be$,
\begin{align*}
\dbar^E &= \dbar_0 + a, \quad \dbar_0= \dbar + \big(\alpha_i-\frac{\Ad(e^{\imag \theta \alpha_i})H_i}{\ln|z|^2}\big)
  \frac{d\zbar}{2\zbar} , \\
\varphi &= \varphi_0 + b, \quad \varphi_0=\big(s_i-\frac{\Ad(e^{\imag \theta\alpha_i})Y_i}{\ln|z|^2}\big)\frac{dz}{z} ,
\end{align*}
where the perturbations $a$ and $b$ belong to the space $C^\infty_\delta$.
\begin{lemma}
  There exists a gauge transformation $g\in\hat{C}^\infty_\delta(E(H^\setC))$,
  defined in a neighbourhood of $x_i$, such that
  $$ g(\dbar^E) = \dbar_0 . $$
\end{lemma}
\begin{proof}
  The equation to solve is
$$ g^{-1}\dbar_0 g = \Ad(g^{-1})a . $$
Writing $g=\exp(u)$, the equation becomes $e^{-u}\dbar_0e^u=e^{-\ad u}a$, with
linear part $\dbar_0u=a$.  One can deduce an explicit solution for the
linear problem from the Cauchy kernel, with suitable estimates
\cite[Lemma 9.1]{Biq97}. Shrinking to a smaller ball if necessary, a
fixed point argument gives the expected solution.
\end{proof}
We now have a $\hat{C}^\infty_\delta$-gauge $f=\be g$ of $\bE$ (seen as a $H^\setC$-bundle) in
which $\dbar^E$ is exactly the model $\dbar_0$, and $\varphi=\varphi_0+b'$ for
some $b'\in C^\infty_\delta$. We deduce an explicit holomorphic gauge,
$$ e = f (-\ln r^2)^{\Ad(e^{\imag \theta\alpha_i})H_i/2} r^{-\alpha_i}, $$
which we use to extend the holomorphic bundle $E$ over
$x_i$. Moreover, in the gauge $e$, the Higgs field becomes
$$ \varphi = \Ad\big(r^{\alpha_i}(-\ln r^2)^{-\Ad(e^{\imag \theta\alpha_i})H_i/2}\big) (\varphi_0+b'). $$
Here
$$\Ad(r^{\alpha_i}(-\ln r^2)^{-\Ad(e^{\imag \theta\alpha_i})H_i/2})\varphi_0 = (s_i + \Ad(z^{\alpha_i})Y_i) \frac{dz}z $$
is just our model for the Higgs field in the holomorphic
trivialization $e$, with $\Gr\Res_{x_i}=s_i+Y_i$, so we have to
analyse the behaviour of the remainder
$$ \varphi' = \Ad\big(r^{\alpha_i}(-\ln r^2)^{-\Ad(e^{\imag \theta\alpha_i})H_i/2}\big) b' , \quad
b'=O\big(\frac1{|\ln r|^\delta}\big) \frac{dz}{z\ln r}. $$ Here remind that
$b'$ is holomorphic outside the origin, and $0<\delta<1$.  This implies
that $\varphi'$ is meromorphic: decompose
$$\varphi'=\sum_\mu \varphi'_\mu\frac{dz}z$$
along the eigenvalues $\mu$ of $\ad(\alpha_i)$ on $\liemc$, and let us
analyse the pole of $\varphi'_\mu$. We certainly have
$$\varphi'_\mu=O(r^\mu|\ln r|^N)$$
which implies $v(\varphi'_\mu)\geq -\lfloor-\mu\rfloor$, as wanted for a parabolic $G$-Higgs
bundle. We can say something more on $\Gr\Res_{x_i}\varphi'$, that is on the
components $\varphi'_\mu$ for $\mu\in \setZ$: decompose further with respect to the
eigenvalues $\eta$ of $\ad(H_i)$ on $\liemc$, and we get
$$\varphi'_{\mu,\eta}=O(r^\mu|\ln r|^{-\frac\eta2-1-\delta}) . $$
This implies that the $(\mu,\eta)$-component for $\mu\in \setZ$
can be non zero only if $\eta<-2$, which implies that
$Y_i+\Gr\Res_{x_i}\varphi'$ is conjugate to $Y_i$, that is
$\Gr\Res_{x_i}\varphi$ remains conjugate to $s_i+Y_i$. This
finishes the proof of Theorem \ref{th:filt-local-syst}.\qed

Table \ref{table}  gives the relation between the singularities for the
Higgs bundle and for the corresponding local systems, in a similar way to
Simpson's table in \cite{Sim90}. The interesting feature here is that they correspond in
a way which extends the Kostant--Sekiguchi correspondence, see appendix
\ref{triples}. More specifically, in the nilpotent case, one gets exactly the
correspondence between $H^\setC$-nilpotent orbits in $\liem^\setC$ and
nilpotent $G$-orbits in $\lieg$ which is the Kostant--Sekiguchi correspondence
(this was turned by Vergne \cite{Ver95} into a diffeomorphism which can be
seen as a toy model of our correspondence theorem). The more general case with
semisimple residues corresponds to an extension of
Kostant--Sekiguchi--Vergne \cite{bielawski,biquard-note}.

\begin{table}
\begin{tabular}{|l|l|l|}
\hline
 & \textbf{Weight}
 & \textbf{Monodromy (projected to the Levi)} \\
\hline
$(E,\varphi)$ &  $\alpha$   & $s+Y= \Gr\Res_{x}\varphi$  \\
\hline
$(F,\nabla)$  &  $s-\tau(s)$   &  $\exp(2\pi\imag \alpha)\exp(2\pi\imag(-s-\tau(s)+Y-H-X))$     \\
\hline
\end{tabular}
\vspace{12pt}

\caption{Table of relations for weights and monodromies}
    \label{table}
\end{table}

\subsection{Deformations of the harmonicity equation and polystability}

As for parabolic $G$-Higgs bundles, there is also  a more general stability
condition for parabolic $G$-local systems depending on a parameter.
In this case the parameter is an element of the subspace of fixed points
of $\liem$ under the isotropy action of $H$, that is
$$
\liem^H=\{v\in \liem \;\;\;\mbox{such that}\;\;\; \Ad(h)(v)=v\;\;\;
\mbox{for every}\; h\in H\}.
$$

Let $F$ be a parabolic $G$-local system with $(P_i,\chi_i)$ defining the
parabolic structure. Let $Q\subset G$ be a parabolic subgroup
of $G$ and $\chi$ be an antidominant character of its Lie algebra $\lieq$. Let
$\sigma$ be a reduction of structure group of $F$ to a $Q$-bundle, which is
invariant under the flat connection.
Given an element $\zeta\in \liem^H$, we  define
$\zeta$-{\bf polystability} of $F$ by the condition
\begin{equation}\label{zeta-parabolic-stability}
\pardeg(F)(Q,\chi,\sigma) -\langle \zeta, s\rangle \geq 0,
\end{equation}
for every such $(Q,\chi)$ and reduction $\sigma$, where $s\in
\liem$ is the element corresponding to $(Q,\chi)$ (see Appendix
\ref{sec:parabolic-subgroups-1}).

This condition corresponds to a deformation of the harmonicity equation
(\ref{harmonic-eq}) given by
\begin{equation}\label{zeta-harmonic-eq}
(D^+_h)^*\psi_h=\zeta.
\end{equation}

Note that this equation is indeed $H$-gauge invariant since
$\zeta\in\liem^H$. Using the methods above one can prove the
following.

\begin{proposition}\label{zeta-correspondence}
A parabolic $G$-local system admits a reduction $h$  to $H$ satisfying
(\ref{zeta-harmonic-eq}) if and only if it is $\zeta$-polystable.
\end{proposition}

For $\zeta\neq 0$ a $\zeta$-polystable parabolic $G$-local system no longer
defines  a $G$-Higgs bundle since now $\varphi$ is not
a holomorphic section.
The equation $\dbar \varphi=0$ is replaced by
$$
\dbar \varphi=\eta\omega,
$$
with $\eta\in (\liem^\C)^{H^\C}$ and $\zeta=\eta-\tau(\eta)$,
where $\tau$ is the compact conjugation in $\lieg^\C$, and
$\omega$ is a K\"ahler form on $X$. However the Higgs field
$\varphi$  defines a holomorphic section if we replace the
bundle associated to $\liem^\C$ by    the bundle associated to
the $H^\C$-representation $\liem^\C/(\liem^\C)^{H^\C}$.
Strictly speaking this is no longer a $G$-Higgs bundle but it
is a Higgs pair in the more general sense mentioned in Remark
\ref{general-pairs}.

\begin{remark} The equation (\ref{zeta-harmonic-eq}) has been recently
studied  by Collins--Jacob--Yau \cite{CJY} for $G=\GL_n\C$, where they
prove in particular Proposition  \ref{zeta-correspondence} in this case.
\end{remark}

\section{Moduli spaces}\label{moduli}

\subsection{Moduli spaces of parabolic $G$-Higgs bundles}
\label{moduli-higgs-bundles}

We follow the notation of Section \ref{parabolic-higgs-bundles}.
Let $X$ be  a compact connected Riemann surface and let $S=\{x_1,\dots,x_r\}$
be  a finite set of different points of $X$. Let $D=x_1+\cdots +x_r$ be the
corresponding effective divisor. Let $(G,H,\theta,B)$ be a real reductive
Lie group  (see Appendix \ref{reductive}).
Consider parabolic weights $\alpha=(\alpha_1,\dots,\alpha_r)$ with
$\alpha_i\in \imag \bar\eA$, and let $c\in \liez$.


Let  $\cM_c(\alpha):=\cM_c(X,D,G,\alpha)$ be the  {\bf moduli
space of meromorphic equivalence classes of $c$-polystable
parabolic $G$-Higgs bundles $(E,\varphi)$  on $(X,D)$ with
parabolic weights $\alpha$}. (See Definition
\ref{def:meromorphic-equivalence} for the notion of meromorphic
equivalence.) Note that if none of the $\alpha_i$'s is
contained in a bad wall then we are simply considering
isomorphism classes in the definition of $\cM_c(\alpha)$. The
moduli space for $c=0$ will be simply denoted by $\cM(\alpha)$.

One can give an analytic  construction of the moduli space of
$c$-stable  parabolic $G$-Higgs bundles by means of the identification of
this moduli space with  the moduli space of solutions of the Hitchin
equations given  in section \ref{sec:correspondence}.
We do not give any details since this is by now fairly standard
(see \cite{biquard-thesis,konno,nakajima}), using the weighted Sobolev spaces
 for the connections, the Higgs fields and the gauge group.


We use the same notation as in Section \ref{sec:defin-parab-g}.
Let $L_i$ the Levi subgroup of $Q_i$ and $\widetilde{L_i}$ the subgroup
corresponding to (\ref{stabilizer}). Consider the spaces $\liem_i^0$
and $\widetilde{\liem_i^0}$ corresponding to (\ref{centralizer}).
Recall that if $\alpha_i\in \imag\eA'_{\lieg}$ (see \ref{very-good-guys}),
$\widetilde{L_i}=L_i$ and $\widetilde{\liem_i^0}=\liem_i^0$.
We denote by  $\widetilde{\liem_i^0}/\widetilde{L_i}$
the set of   $\widetilde{L_i}$-orbits.

There is  a map
$$
 \varrho: \cM_c(\alpha)\lra \prod_i (\widetilde{\liem_i^0}/\widetilde{L_i})
$$
defined by taking the $\widetilde{L_i}$-orbit  of
$\Gr\Res_{x_i}\varphi \in \widetilde{\liem_i^0}$.

We  fix now orbits $\LLL_i\in \widetilde{\liem_i^0}/\widetilde{L_i}$ and
denote $\LLL=(\LLL_1,\cdots, \LLL_r)$.
We consider the moduli space
$$
\cM_c(\alpha,\LLL):=\varrho^{-1} (\LLL).
$$

To ensure smoothness of $\cM_c(\alpha)$  at a point $(E,\varphi)$
one generally
needs that $(E,\varphi)$ be stable and simple.
A parabolic $G$-Higgs bundle
$(E,\varphi)$ is
said to be {\bf simple} if $\Aut(E,\varphi)=Z(H^\C)\cap \ker \iota$, where
$\iota$ is the isotropy representation (see Appendix \ref{reductive}), and
$\Aut(E,\varphi)$ is the group of parabolic automorphisms of $(E,\varphi)$.
If $G$ is complex we hence require $\Aut(E,\varphi)=Z(G)$.
Except for Lie groups
of type $A_n$ stability does not generally imply simplicity.
For a general reductive real Lie group $G$, to have smoothness
of $\cM_c(\alpha)$ at a point $(E,\varphi)$
 one also needs the
vanishing of  a certain  obstruction class in the
second hypercohomology group of the  deformation complex
associated to $(E,\varphi)$
(see \cite{garcia-prada-gothen-mundet:2009a}).
This condition is always satisfied if $G$ is  complex ---
it follows from stability and Serre duality. This  is also  satisfied
for $c=0$ when $G$ is a real form of a complex reductive group $G^\C$
 if the extension of $(E,\varphi)$ to a parabolic
$G^\C$-bundle is stable.

We have the following.

\begin{proposition}\label{moduli-general}
Let  $\cM_c^*(\alpha)\subset \cM_c(\alpha)$ and
$\cM_c^*(\alpha,\LLL)\subset \cM_c(\alpha, \LLL)$ be the subspaces of parabolic
$G$-Higgs bundles of weight $\alpha$ that are stable, simple and
have vanishing obstruction class. Then

(1) $\cM_c^*(\alpha)$ is a smooth Poisson manifold
foliated by symplectic leaves
$\cM_c^*(\alpha,\LLL)$, admitting a K\"ahler structure.

(2) If $G$ is complex $\cM_c^*(\alpha)$ is a holomorphic Poisson manifold
(in fact a hyperpoisson manifold)
foliated by holomorphic symplectic
leaves  $\cM_c^*(\alpha,\LLL)$, admitting a hyperk\"ahler structure.
\end{proposition}

For the Poisson structure and the K\"ahler structure of these
types of moduli spaces one can look at \cite{bottacin,markman}.

If $(G,H,\theta,B)$  satisfies the integral condition given  in
Remark \ref{integral-characters}, which happens in particular
if $G$ is complex, we can define a genericity condition for the
weights. We say that $\alpha=(\alpha_1,\cdots,\alpha_r)$ is
{\bf generic} if for any parabolic subgroup $P=P_s\subset H^\C$
and any antidominant character of $P$ in the finite collection
mentioned in Remark \ref{integral-characters},
 the relative degree satisfies $$\sum_i\deg((P_i,\alpha_i),(P,\chi))\notin\ZZ.$$
Since the choice of parabolic types of $P$ is also finite,
this defines walls dividing $\imag \bar\eA^r$ in chambers.
If we now take the stability parameter $c=0$ (which is
the relevant value in relation to local systems and representations) then,
since the first term in (\ref{stability-condition}) is an integer,
we have the following (recall that $\cM(\alpha)$ denotes the moduli
space for $c=0$).

\begin{proposition} \label{moduli-real}
If $(G,H,\theta,B)$  satisfies the integral condition in Remark
\ref{integral-characters}  (in particular if $G$ is complex) and
$\alpha$ is generic, then

(1) $\cM(\alpha)$ is a Poisson manifold (possibly with orbifold singularities)
foliated by  symplectic leaves $\cM(\alpha,\LLL)$, admitting a K\"ahler
structure. Moreover, if $\alpha$ and $\alpha'$ are in the same  chamber
$\cM(\alpha)= \cM(\alpha')$  and $\cM(\alpha,\LLL)=\cM(\alpha',\LLL)$
as real and complex orbifolds, respectively (the Poisson and symplectic
structures, respectively, depend on $\alpha$).

(2) If $G$ is complex $\cM(\alpha)$ is a holomorphic Poisson manifold (possibly with orbifold
singularities) foliated by holomorphic symplectic leaves $\cM(\alpha,\LLL)$,
admitting a hyperh\"ahler
structure. Moreover, if $\alpha$ and $\alpha'$ are in the same  chamber
$\cM(\alpha)= \cM(\alpha')$  as  holomorphic Poisson orbifolds
and $\cM(\alpha,\LLL)=\cM(\alpha',\LLL)$ as holomorphic symplectic orbifolds.
\end{proposition}

\begin{remark}
Result (2) in Proposition \ref{moduli-real} is a consequence of the fact
that in  the hyperk\"ahler structure defined as a hyperk\"ahler
quotient by the Hitchin equations, the symplectic form $\omega_1$ depends
on $\alpha$, while the
complex structure $I_1$, and the $I_1$-holomorphic symplectic form
$\Omega_1=\omega_2+i \omega_3$ depend on  $\LLL$.
\end{remark}

\begin{remark}
The orbifold singularities mentioned in Proposition \ref{moduli-real}
 take  place at the stable but
not simple points.   If $(E,\varphi)$ is stable but not simple,
there is a reduction of structure group of $(E,\varphi)$ to a reductive
subgroup ---the centralizer in $G$  of $\Aut(E,\varphi)$
(see \cite{garcia-prada-oliveira}).
It is not clear whether one could define an extra genericity condition for
$\alpha$ and $\LLL$ to avoid this phenomenon.
\end{remark}

There should be a  GIT construction of these moduli spaces. As far
as we are aware this has only been done for $G=\GL_n\C$
(see \cite{yokogawa}). In the generality considered here, this
will need most likely to involve parahoric torsors (see
\cite{balaji-seshadri,boalch,heinloth}).

\subsection{Moduli spaces of parabolic $G$-local systems and representations}
\label{moduli-local-systems}

Let $X$ be  a compact connected Riemann surface and let $S=\{x_1,\cdots,x_r\}$
be  a finite set of different points of $X$.
Let $(G,H,\theta,B)$ be a real reductive
Lie group.
We use the notation of Section \ref{parabolic-local-systems}.
Let  $\beta=(\beta_1,\cdots,\beta_r)$ be as in Definition
\ref{local-system}, and  $\zeta\in \liem^H$.

We define  $\cS_\zeta(\beta):=\cS_\zeta(X,S,G,\beta)$ to be
the  {\bf moduli space of isomorphism classes of $\zeta$-polystable
parabolic $G$-local systems $F$   on $(X,S)$ with parabolic
weights $\beta$}. The moduli space for $\zeta=0$ will be simply denoted
by $\cS(\beta)$. Let $P_\beta$  be the parabolic subgroup
 of $G$ defined by  $\beta$, and $L_\beta\subset P_\beta$, the Levi subgroup.
We fix loops $c_i$  enclosing  the marked points $x_1,\cdots, x_r$ simply.
The monodromy of the loop $c_i$ around $x_i$ takes values in $P_{\beta_i}$,
and we
consider  its projection to $L_{\beta_i}$. Its  conjugacy class $\CCC_i$
in   $L_{\beta_i}$ is independent of the simple loop that we have taken.
This defines a map

$$
 \mu: \cS_\zeta(\beta)\lra \prod_i \Conj(L_{\beta_i})
$$
where $\Conj(L_{\beta_i})$ is the set of conjugacy classes of $L_{\beta_i}$.

Let $\CCC=(\CCC_1,\cdots,\CCC_r)$ with $\CCC_i\in \Conj(L_{\beta_i})$, and
consider the moduli space
$$
\cS_\zeta(\beta,\CCC)=\mu^{-1}(\CCC).
$$

As in the case of parabolic Higgs bundles,  using the same techniques, one can
give an analytic
construction of $\cS_\zeta(\beta)$ and $\cS_\zeta(\beta,\CCC)$, via the
identification  provided by Theorem \ref{th:filt-local-syst} of these spaces
with the spaces of
solutions to the Hermitian--Einstein equations (see also Proposition \ref{zeta-correspondence}
for $\zeta\neq 0$). This is done by Nakajima \cite{nakajima} for the
case $G=\GL_2\C$.

We say that a parabolic $G$-local system is {\bf irreducible} if its group
of automorphisms coincides with the centre of $G$.  We have the following.

\begin{proposition}
Let $\cS^*_\zeta(\beta)\subset \cS_\zeta(\beta)$ and
$\cS^*_\zeta(\beta,\CCC)\subset \cS_\zeta(\beta,\CCC)$ be the
subsets of stable and irreducible elements. Then

(1) $\cS^*_\zeta(\beta)$ is a smooth Poisson manifold foliated by
symplectic leaves  $\cS^*_\zeta(\beta,\CCC)$, admitting a K\"ahler structure.

(2) If $G$ is complex $\cS^*_\zeta(\beta)$ is a holomorphic Poisson manifold
foliated by holomorphic,
symplectic leaves  $\cS^*_\zeta(\beta,\CCC)$, admitting a hyperk\"ahler
structure.
\end{proposition}

As in the case of Higgs bundles, by Remark
\ref{simplification-stability}, we can define a genericity condition for
$\beta$. We say that $\beta=(\beta_1,\cdots,\beta_r)$ is {\bf generic} if
for any maximal parabolic subgroup $Q\subset G$, and its corresponding
preferred character $\chi=\chi_Q$ (see Remark \ref{simplification-stability})
we have
$$
\sum_i \deg\big( (P_i,\chi_i),(Q,\chi_Q)\big)\neq 0.
$$

We observe that the function $\mu$ defining the relative degree in (\ref{eq:32}),
as a function on $Q\backslash G\subset \liem$, has a finite number of values,
and since the number of types of maximal parabolic subgroups $Q$ is finite,
the condition $\sum_i \deg\big( (P_i,\chi_i),(Q,\chi_Q)\big)=0$ defines
a finite number of walls  dividing $\liem^r$ in chambers.
By (\ref{eq:locsys-def-par-deg-sigma}) for generic $\beta$ a semistable
parabolic $G$-local system is actually stable, and we have the following.

\begin{proposition} \label{moduli-real-local-systems}
If $(G,H,\theta,B)$  satisfies the integral
condition in Remark \ref{integral-characters} (in particular if $G$ is complex) and
$\beta$ is generic, then

(1) $\cS(\beta)$ is a Poisson manifold (possibly with orbifold singularities)
foliated by  symplectic leaves $\cS(\beta,\CCC)$, admitting a K\"ahler
structure. Moreover, if $\beta$ and $\beta'$ are in the same  chamber
$\cS(\beta)= \cS(\beta')$  and $\cS(\beta,\CCC)=\cS(\beta',\CCC)$
as real and complex orbifolds, respectively (the Poisson and symplectic
structures, respectively, depend on $\beta$).

(2) If $G$ is complex $\cS(\beta)$ is a holomorphic Poisson manifold
(possibly with orbifold
singularities) foliated by holomorphic symplectic leaves $\cS(\beta,\CCC)$,
admitting a hyperh\"ahler  structure. Moreover, if $\beta$ and $\beta'$ are in the same  chamber
$\cS(\beta)= \cS(\beta')$  as  holomorphic Poisson orbifolds
and $\cS(\beta,\CCC)=\cS(\beta',\CCC)$ as holomorphic symplectic orbifolds.
\end{proposition}

We consider now the moduli space of representations of
$\pi_1(X\setminus S)$. By a \textbf{representation} of $\pi_1(X\setminus S)$ in
$G$ we mean a homomorphism $\rho\colon \pi_1(X\setminus S) \to G$.
A representation is \textbf{reductive}
if composed with the adjoint representation in the Lie algebra
of $G$ decomposes as a sum of irreducible representations.
If $G$ is algebraic this is equivalent to saying that the Zariski closure of the
image of $\pi_1(X\setminus S)$ in $G$ is a reductive group.
Define the
\textbf{moduli space of reductive representations of $\pi_1(X\setminus S)$ in $G$}
to be the orbit space
$$
\cR:=\cR(X,S,G)= \Hom^{+}(\pi_1(X\setminus S),G)/ G,
$$
where  $\Hom^{+}(\pi_1(X\setminus S),G)$ is the set of reductive representations
and $G$ acts by conjugation. This is a real analytic variety (algebraic if $G$
is algebraic). If $G$ is complex $\cR$ is the affine GIT quotient
$\Hom(\pi_1(X\setminus S),G)\sslash G$.
Let $c_i$ be a loop enclosing  $x_i$ simply.
Fix conjugacy classes $\CCC_i\in \Conj(G)$, $i=1,\cdots, r$,
 and let  $\CCC=(\CCC_1,\cdots,
\CCC_r)$. We define the moduli space of representations of
$\pi_1(X\setminus S)$ in $G$ with fixed conjugacy classes $\CCC$ as the
subvariety
$$
\cR(\CCC):=\{[\rho]\in \cR\;\;:\;\;  \rho([c_i])=\CCC_i,\;\;
i=1,\cdots, r \}.
$$

\begin{proposition}\label{correspondence-local-systems-representations}
Let $\CCC=(\CCC_1,\cdots,\CCC_r)$ with $\CCC_i\in \Conj(L_{\beta_i})$ like
in Section \ref{moduli-local-systems} and let
$\CCC'=(\CCC'_1,\cdots,\CCC'_r)$ with $\CCC_i'\in \Conj(P_{\beta_i})$.
Let $\pi_i:P_{\beta_i}\to L_{\beta_i}$ be  the
projection to the Levi subgroup.  Then there is a forgetful map
$$
\cS(\beta,\CCC)\to \bigcup_{\CCC'\;:\; \pi_i(\CCC_i')=\CCC_i}\cR(\CCC').
$$
In particular, if $\beta=0$, $L_{\beta_i}=G$, $\CCC_i\in \Conj(G)$ and
$$
\cS(0,\CCC)=\cR(\CCC).
$$
\end{proposition}

\subsection{Correspondences of moduli spaces}
\label{correspondences}

We formulate now in terms of moduli spaces the correspondences that we have
proved in this paper.  We follow the notation of Sections
\ref{moduli-higgs-bundles} and \ref{moduli-local-systems}.
From Proposition \ref{prop:filt-local-syst}
and  Theorem \ref{th:filt-local-syst} we have the following.

\begin{theorem} \label{correspondence-higgs-bundle-local-system}
Let $(G,H,\theta,B)$  be a real reductive group.
Let $(\alpha,\LLL)$ and $(\beta,\CCC)$ be  related
as in  Table \ref{table} (where $\beta=s-\tau(s)$). Then,
$\cM(\alpha,\LLL)$ and $\cS(\beta,\CCC)$ are homeomorphic.
In particular, if $\beta=0$, $\cM(\alpha,\LLL)$ and $\cR(\CCC)$
are homeomorphic. Moreover,
$\cM^*(\alpha,\LLL)$ and $\cS^*(\beta,\CCC)$ are diffeomorphic.
In particular, if $\beta=0$, $\cM^*(\alpha,\LLL)$ and $\cR^*(\CCC)$
are diffeomorphic, where $\cR^*(\CCC)$ is the subvariety of irreducible
representations.
\end{theorem}

\begin{theorem}\label{generic-alpha-beta}
Let $(G,H,\theta,B)$  satisfy the integral condition given in
Remark \ref{integral-characters}
(which is satisfied in particular if $G$ is complex). Let
 $(\alpha,\LLL)$ and $(\beta,\CCC)$ be related
by Table \ref{table}, with
$s=\tau(s)$. Then if $\alpha$ and $\beta$ are generic
$\cM(\alpha,\LLL)$ and $\cS(\beta,\CCC)$ are diffeomorphic.
\end{theorem}

Theorem \ref{generic-alpha-beta}, combined with Propositions
\ref{moduli-real} and \ref{moduli-real-local-systems}, imply
the following.

\begin{theorem}\label{integral-case}
Let $(G,H,\theta,B)$  satisfy the integral
condition given in Remark \ref{integral-characters}
(which is satisfied in particular if $G$ is complex). Let
$(\alpha,\LLL)$ and $(\beta,\CCC)$, and  $(\alpha',\LLL')$ and $(\beta',\CCC')$,
be related
by Table \ref{table}.
Then if $\alpha$, $\alpha'$, $\beta$ and $\beta'$ are generic

(1)  $\cM(\alpha,\LLL)$ and $\cM(\alpha',\LLL')$ are diffeomorphic.

(2) $\cS(\beta,\CCC)$ and $\cS(\beta',\CCC')$ are diffeomorphic.
\end{theorem}

This is simply because the combined walls defining genericity  for $\alpha$ and
$\beta$ have codimension bigger or equal than 2 in the space of parameters
defining a connected chamber for the generic parameters.

To show that the correspondences in Theorems
\ref{generic-alpha-beta} and \ref{integral-case} are homeomorphisms or diffeomorphisms (in fact real analytic isomorphisms) one can argue as in \cite{kobayashi,konno,Sim90}.

\begin{remark}
It would be interesting to explore the possible extension of this
correspondence to the moduli spaces $\cM_c(\alpha,\LLL)$ and
$\cS_\zeta(\beta,\CCC)$ for non-zero values of $c\in \liez(\lieh)$ and
$\zeta\in \liem^H$.
\end{remark}

\section{Examples}
\label{sec:examples}

We give some illustrations of the use of parabolic Higgs bundles for
defining certain classical components in the moduli space of representations
$\cR(G)$. We do not give details since this
follows more or less directly from the arguments in the compact case
and the parabolic machinery developed here.

\subsection{Teichm\"uller--Hitchin component of split groups}
\label{sec:teichm-comp-split}

We begin by the case $G=\PSL_2\setR$: this gives a parametrization of the
space of hyperbolic metrics with cusps at the marked points. We have
$H=\U_1/\setZ_2$ where the $\U_1$ is seen as a maximal compact subgroup of
$\SL_2\setR$. The construction on an unpunctured surface requires a square
root of $K$; in the punctured case we need a square root of
$K(D)$. The $\setC^*$-bundle $K(D)$ has not always a square root (this
requires the degree of $D$ to be even), but as a $\setC^*/\setZ_2$ bundle it
always has such a square root: let $E$ be such a choice (here $E$ is a
principal holomorphic $\setC^*/\setZ_2$ bundle). Equip $E$ with a trivial
parabolic structure at the marked points. Then
$$E(\liem^ \setC)=K(D)\oplus K(D)^{-1},$$
and we consider the meromorphic Higgs field $\varphi\in H^0(X,E(\liem^ \C)\otimes K(D))$
$$ \varphi = q_2 \oplus 1, \quad q_2\in H^0(X,K^2(D)),$$
where $1$ has a simple pole at the marked points, while $q_2$ appears
to be holomorphic at the marked points. In particular $\Res_{x_i}\varphi$ is
nilpotent and nonzero.

The Higgs bundles $(E,\varphi)$ for $q_2\in H^0(X,K^2(D))$ are stable, and, as
in \cite{hitchin}, the corresponding solutions of the Hitchin selfduality
equations (Hermite--Einstein) provide hyperbolic metrics on $X\setminus D$, whose monodromy is the
representation $\pi_1(X\setminus D)\to \PSL_2\setR$ corresponding to $(E,\varphi)$. In
particular, its monodromy around $x_i$ is unipotent so we obtain a
cusp at $x_i$. This gives a complete parametrization of the
Teichm\"uller space for the punctured surface $(X,D)$ by the space of
quadratic differentials $H^0(X,K^2(D))$.

If the degree of $D$ is even, then a choice of square root $L$ of
$K(D)$ gives a lifting of this $\PSL_2\setR$ component to $\SL_2\setR$, with
unipotent monodromies at the punctures. It is also well known that for
all degrees, it is possible to lift the component to $\SL_2\setR$ with
minus unipotent monodromies: in the Higgs bundle formalism, this
amounts to considering a parabolic structure with weight at the
boundary of the alcove, in the following way. One considers a square
root $L$ of $K$ with a parabolic weight $-1/2$ at each puncture (this
is morally a square root of $K(D)$, in particular $\pardeg
L=g-1+\frac12\deg D$). Then the same Higgs field as above,
$$ \varphi = \begin{pmatrix} 0 & q_2 \\ 1 & 0 \end{pmatrix} $$
is now a parabolic Higgs bundle in our sense, with nilpotent
residue $(\begin{smallmatrix} 0 & 0 \\ 1 & 0
\end{smallmatrix})$, this example is discussed as the end of
Section \ref{sec:defin-parab-g}. This bundle induces the
previous $\PSL_2\setR$ Higgs bundle after applying the Hecke
transform $z^{1/2}$ near each puncture (this exists in
$\setC^*/\setZ_2$). (See
\cite{nasatyr-steer,biswas-gastesi-govindarajan} for a related
description.)

Now pass to other split real groups. The generalisation of
\cite{hitchin-teichmueller} is straightforward. We work with $G$ a
split real group of adjoint type, and we consider an irreducible
representation $\rho:\SL_2\setR\to G$, sending $\U_1$ in the maximal compact $H$
of $G$. Choose as above a $\setC^*/\setZ_2$ principal holomorphic bundle $E$,
square root of $K(D)$. We have a decomposition $\lieg^ \setC=\oplus_1^l V_j$
into irreducible pieces under $\rho$, such that $V_1$ is the image of
$\rho$, and we choose a highest weight vector $e_j\in V_j$. If $(H,X,Y)$ is
a standard $\lies\liel_2$ basis, with $H\in \sqrt{-1}\lieu_1$, then $e_1=\rho(X)$,
and we also define $e_{-1}=\rho(Y)$. Moreover there exist a basis $(p_j)$
of invariant polynomials on $\lieg^ \setC$, of degrees $m_j+1$ (determined
by the fact that $\ad H$ acts with eigenvalue $m_j$ on $e_j$), such
that for any element $f=e_{-1}+f_1e_1+\cdots +f_le_l$ one has $f_j=p_j(f)$.
Now, adapting \cite{hitchin-teichmueller} to the punctured case, we
consider the Higgs bundle $\rho(E)$ with trivial parabolic structure at
the punctures, and Higgs field
$$ \varphi = e_{-1} + \sum_1^l q_j e_j, \quad q_j\in H^0(K^{m_j+1}(m_jD)). $$
It turns out that $\varphi$ is meromorphic with simple poles, and
$\Res_{x_j}\varphi=e_{-1}$ is regular nilpotent. This provides the expected
Teichm\"uller--Hitchin component.

One can remark that there is a natural deformation keeping the same Higgs bundles but changing the parabolic structure: one can consider at each marked point the parabolic group $P_{\rho(H)}$, with some strictly antidominant character $\alpha_i$. Then $e_{-1}\in \liep_{\rho(H)}$ so the space of Higgs fields with $\Res_{x_i}\varphi\in \liep_{\alpha_i}$ remain the same. This gives a
deformation of the Teichm\"uller--Hitchin component to a space of representations with fixed
compact monodromy around the punctures. One can obtain similarly any
regular semisimple monodromy at the punctures by allowing $q_j\in
H^0(K^{m_j+1}((m_j+1)D))$ and by fixing the highest order term of each
$q_j$ at each puncture: this modifies the residue of $\varphi$, hence the
noncompact part of the monodromy.

\subsection{Hermitian groups and the Milnor--Wood inequality}
\label{sec:herm-groups-miln}

Another case where there is a distinguished component in the space of
representations of $\pi_1(X\setminus D)$ into $G$ is when $G$ is Hermitian, that
is $G/H$ is a Hermitian symmetric space of noncompact type. The
general theory from the Higgs bundle viewpoint is done in \cite{BGR}
and can be generalised to the punctured case. In particular one
recovers the Milnor--Wood inequality of Burger--Iozzi--Wienhard
\cite{biw} in the punctured case. Again we do not give details of the
proofs, which can be adapted from \cite{BGR} to the parabolic case.

In the Hermitian case, $\lieh$ has a 1-dimensional centre, generated
by an element $J$ which induces the complex structure of $G/H$. In
particular it decomposes $\liem^ \setC$ into $\pm i$ eigenspaces: $\liem^
\setC=\liem^+\oplus\liem^-$. Now consider a parabolic $G$-Higgs bundle $(E,\varphi)$,
so $\varphi$ decomposes as $\varphi=\varphi^++\varphi^-$. We can define a Toledo invariant in
the following way: it is proved in \cite{BGR} that there exists a
character, called the Toledo character, $\chi_T:H^ \setC\to\setC^*$ and a
polynomial $\det:\liem^+\to\setC$ of degree $r=\rk(G/H)$, such that for any
$h\in H^ \setC$ one has $\det(h\cdot x)=\chi_T(h)\det(x)$. The Toledo invariant on
a compact surface is $\deg E(\chi_T)$, and on a punctured surface we
define the Toledo invariant by
$$ \tau(E) = \pardeg(E,\chi_T), $$
where the reduction of $E$ is $E$ itself. This is actually equal to
the parabolic degree of the line bundle $E(\chi_T)$ equipped with the
parabolic weight $\chi_T(\alpha_i)$ at each puncture $x_i$.

The proof of Theorem 4.5 in \cite{BGR} extends to give the following
Milnor--Wood inequality: if $(E,\varphi)$ is a semistable parabolic $G$-Higgs
bundle on the Riemann surface $(X,D)$ with $n$ punctures, then
\begin{equation}
 -\rk(\varphi^+) (2g-2+n) \leq \tau(E) \leq \rk(\varphi^-) (2g-2+n),\label{eq:34}
\end{equation}
where the rank of $\varphi^\pm$ is the generic rank (there is a well defined
notion of rank for elements of $\liem^\pm$, it is bounded by $r$). In
particular, one has $|\tau(E)|\leq r(2g-2+n)$, which gives another proof of
the Milnor--Wood inequality of \cite{biw} in the punctured case, when
the Higgs bundle comes from a representation.

The case of equality in the Milnor--Wood inequality is of interest (the
corresponding representations are called maximal representations), and
leads to a nice description of the moduli space. Restrict to the case
$G/H$ is of tube type. One way to state this condition is to say that
the Shilov boundary of $G/H$ is itself a symmetric space $H/H'$. It is
proved in \cite{BGR} that there is a Cayley transform, that is the
moduli space is isomorphic to a moduli space of $K^2$-twisted
$H^*$-Higgs bundles, where $H^*/H'$ is the noncompact dual of
$H/H'$. Here $K^2$-twisted means that the Higgs field takes values in
$E(\liem^ \setC)\otimes K^2$ rather than $E(\liem^ \setC)\otimes K$. This can be extended
to the punctured case in the following way. For simplicity, suppose
that the parabolic weights lie in $\sqrt{-1}\AAA'_\lieg$, which means
that all eigenvalues of $\ad \alpha_i$ on $\liem^ \setC$ have modulus smaller
than $1$. Then one can similarly prove that in the maximal case, one
has $\alpha_i\in \lieh'$ and the moduli space of polystable $G$-Higgs bundle
is isomorphic to a moduli space of $K(D)^2$-twisted polystable
$H^*$-Higgs bundles, with parabolic structure $\alpha_i$ at the
punctures. Such an isomorphism remains true if one drops the condition
that $\alpha_i\in \sqrt{-1}\AAA'_\lieg$, but then one has only $e^{2\pi
  \sqrt{-1}\alpha_i}\in H'$ and a Hecke transformation is needed to obtain
the parabolic weights of the Cayley transformed bundle.

This fact on $\alpha_i$ also implies that the monodromy around $x_i$ fixes
a point of the Shilov boundary, a fact also known from \cite{biw}.

\appendix

\section{Lie theory}
\label{s:Lie-theory}
\subsection{Weyl alcoves and conjugacy classes of a compact Lie group}
\label{alcoves}

For the following see e.g. \cite{BD}.

Let $H$ be a compact Lie group with Lie algebra $\lieh$. Let
$\langle\cdot,\cdot\rangle$ be a $H$-invariant inner product on
$\lieh$. Let $T\subset H$ be a Cartan subgroup, i.e. a maximal
torus, and $\liet\subset\lieh$ be its Lie algebra (a Cartan
subalgebra). Fix a system of real simple roots (see e.g.
\cite[Chap V, Def 1.3]{BD}) and denote by $R_+$ the set of
positive roots.
Consider the family of affine hyperplanes in $\liet$
$$
\HHH_{\lambda n}=\lambda^{-1} (n),\;\; \lambda\in R^+,\;\; n\in\Z
$$
together with the union $\liet_s=\cup_{\lambda,n} \HHH_{\alpha n}$.
As shown in \cite{BD}, this family is given by the critical points of the exponential map
\begin{equation}\label{exp}
\exp: \liet \lra T.
\end{equation}

The set $\liet - \liet_s$ decomposes
into convex connected components which are called
the {\bf alcoves} of $H$ (sometimes also referred as {\bf Weyl
  alcoves}). Note that, by definition, alcoves are open. So, a choice of an
alcove is basically a choice of a logarithm for (\ref{exp}).
A {\bf wall} of an alcove $\eA$  is one of the  subsets of
$\bar\eA \cap \HHH_{\lambda,n}$ of $\liet$ that have dimension
$k-1$, where $\bar\eA$ is the closure of $\eA$ and
$k=\rank(H)$.

Let $W:=N(T)/T$ be the Weyl group of $H$. The $H$-invariant inner product on
$\lieh$ induces
a $W$-invariant inner product in $\liet$.
The  {\bf co-character lattice} $\Lambda_{\cochar}\subset \liet$ is defined
as the kernel of the exponential map (\ref{exp}).
Recall  that the {\bf co-roots} are the elements of $\liet$
defined by
$$
\lambda^*=2b^{-1}(\lambda)/\langle \lambda,\lambda\rangle,
$$
where $b$ is the isomorphism $b:\liet \lra \liet^*$ defined by
the inner product $\langle \cdot,\cdot\rangle$.

The co-roots define a lattice $\Lambda_{\coroot}\subset \liet$.
We have that $\Lambda_{\coroot}\subset \Lambda_{\cochar}$ and
$\pi_1(H)=\Lambda_{\cochar}/\Lambda_{\coroot}$. In particular,
$\Lambda_{\coroot}= \Lambda_{\cochar}$ if $H$ is simply connected.

The {\bf affine Weyl group} is defined as
$$
W_{\aff}=\Lambda_{\cochar} \rtimes W
$$
where $\Lambda_{\cochar}$ acts on $\liet$ by translations.

The alcoves of $H$  are important for us because of their relation with
conjugacy classes of $H$. If $H$ is connected, every element of $H$ is
conjugate to an element of $T$, in particular every element of $H$ lies
in a Cartan subgroup. If  $\Conj(H)$ is the space of
{\bf conjugacy classes} of $H$ we have then homeomorphisms
$$
\Conj(H)\simeq T/W\simeq \liet/W_{\aff}.
$$
We  have the following.

\begin{proposition}
Let $H$ be a connected compact Lie group.
The closure $\bar \eA$ of any alcove $\eA$ contains a fundamental domain
$\eA^f$ for the action of $W_{\aff}$ on $\liet$ , i.e. every $W_{\aff}$-orbit
meets $ \eA^f$ in exactly one point. Hence the  space of conjugacy
classes of $H$ is in bijection with $\eA^f$.
\end{proposition}

\newcommand{\Spec}{\operatorname{Spec}}

Define
$$ \imag\eA' = \{ \alpha\in
    \imag\bar\eA\mid
\Spec(\ad(\alpha))\subset (-1,1)\},
$$
and let
$$\imag W\aA'=\bigcup_{w\in W}\imag w\eA'.$$
The following will play a crucial role in our definition of
parabolic Higgs bundle and in the analysis involved in the
Hitchin--Kobayashi correspondence.

\begin{proposition}\label{prop-alcove}
Let $\eA\subset\tlie$ be an alcove of $H$ such that $0\in \bar\eA$. Then:
\begin{enumerate}
\item If $\alpha\in \imag\bar \eA$ then
    $\Spec(\ad(\alpha))\subset [-1,1]$.
\item We have  $\imag\eA\subset \imag\eA'$.
\item We have $ \imag W\eA' = \{ \alpha\in \imag\tlie\mid
    \Spec(\ad(\alpha))\subset(-1,1)\}.$

\item For any $\alpha\in \imag\bar \eA$ there exist $k\in \ZZ$
    and $\lambda\in\imag(2\pi)^{-1}\Lambda_{\cochar}$ such
    that $$k\alpha+\lambda\in\imag W\aA'.$$
\end{enumerate}
\end{proposition}
\begin{proof}
Items (1), (2) and (3) are immediate consequences of the
definitions. We now prove (4). Let
$\Gamma=\imag(2\pi)^{-1}\Lambda_{\cochar}$, and note that
$\imag\tlie/\Gamma$ is compact. For any $\beta\in\imag\tlie$
let $[\beta]$ denote its class in $\imag\tlie/\Gamma$. Given
$\alpha\in\imag\tlie$ consider the sequence
$[\alpha],[2\alpha],[3\alpha],\dots$. By compactness this
sequence must accumulate somewhere. So one may take elements
of the form $[\mu\alpha]$ and
$[\nu\alpha]$ with $\mu\neq\nu$ in such a way that
$[\nu\alpha]-[\mu\alpha]=[(\nu-\mu)\alpha]$ is arbitrary close to $[0]$. In particular, since $\imag W\aA'$ is
a neighborhood of $0$ we can find $k$ and  $\lambda$ such that
$k\alpha+\lambda\in\imag W\aA'$.
\end{proof}

\begin{proposition}
Let $G$ be the complexification of a connected compact Lie group $H$, and
let $T^\CC$ be the complexification of a Cartan subgroup of $H$. Then every
semisimple element of $G$ is conjugate to an element of $T^\CC$, which in
particular
can be written as $\exp(\alpha)\exp(s)$ with $\alpha\in \bar\eA$,
$s\in \imag\liet$ and $[\alpha,s]=0$.
\end{proposition}

\subsection{Real reductive Lie groups}\label{reductive}

Following  Knapp
\cite[p.~384]{knapp},  a \textbf{real reductive Lie group}
is defined as a $4$-tuple
$(G,H,\theta,B)$, where $G$ is a real Lie group,  $H \subset G$ is a maximal
compact subgroup, $\theta\colon \lieg
\to \lieg$ is a Cartan involution, and $B$ is a
non-degenerate bilinear form on $\lieg$, which is $\Ad(G)$-invariant
and $\theta$-invariant.  The data $(G,H,\theta,B)$ has to satisfy in
addition that
\begin{itemize}
\item
the Lie algebra $\lieg$ of $G$ is reductive
\item
$\theta$ gives  a decomposition  (the Cartan decomposition)
\begin{displaymath}
 \lie{g} = \lieh \oplus \liem
\end{displaymath}
into its $\pm1$-eigenspaces, where  $\lieh$ is the Lie
algebras of $H$, so we have
$$[\hlie,\hlie]\subset\hlie,\qquad
[\hlie,\mlie]\subset\mlie,\qquad [\mlie,\mlie]\subset\hlie,$$

\item
$\lieh$ and $\liem$ are orthogonal under $B$, and $B$ is positive definite
on $\liem$ and negative definite on $\lieh$,

\item
multiplication as a map from $H\times \exp\liem$ into $G$ is an onto
diffeomorphism.

\item Every automorphism $\Ad(g)$ of $\lieg^\CC$ is inner for $g\in G$,
i.e., is given by some $x$ in $\Int\lieg$.

\end{itemize}


Of course a compact real Lie group $G$ whose Lie algebra is
equipped with a non-degenerate $\Ad(G)$-invariant bilinear form
belongs to a reductive tuple $(G,H,\theta,B)$ with $H=G$ and
$\theta=\Id$. Also the underlying real structure of a the
complexification $G$ of a compact Lie group $H$, whose Lie
algebra $\lieh$ is equipped with a non-degenerate
$\Ad(H)$-invariant bilinear form can be endowed with a natural
reductive structure.

 If $G$ is semisimple, the data $(G,H,\theta,B)$ can be
 recovered (to be precise, the quadratic form $B$ can only
   recovered up to a scalar but this will be sufficient for everything
   we do in this paper)  from the choice of a maximal compact
 subgroup $H \subset G$. There are other situations where less
 information is enough, e.g. for certain linear groups (see
 \cite[p.~385]{knapp}).

 The bilinear form $B$ does not play any role in the definition of
a parabolic $G$-Higgs bundle but
 is essential for defining  the  stability condition and the gauge equations
involved in the Hitchin--Kobayashi correspondence.

Note that  the compactness of $H$ together with the last
property above say that $G$ has only finitely many components.

Let $\liegc$ and $\liehc$ be the complexifications of
$\lieg$ and $\lieh$ respectively, and let
$H^\C$ be the complexification of $H$. Let

\begin{equation}\label{eqn:CCartan}
\liegc=\liehc \oplus \liemc
\end{equation}
be the complexification of the Cartan decomposition.
The group $H$ acts linearly on $\mlie$ through the adjoint
representation, and this action extends to a linear holomorphic action
of $\HC$ on $\liemc$, the \textbf{isotropy  representation}
that we will denote by
$$
\iota: H^\C\lra \Aut(\liemc),
$$
or sometimes by $\Ad$ since it is obtained by restriction
of the adjoint representation of $G$.

If $G$ is complex with maximal compact subgroup $H$, then $\lieg=\lieh\oplus
\imag\lieh$. We thus have that $\liem=\imag\lieh$, and the isotropy representation
coincides with the adjoint representation $\Ad: G\lra \Aut(\lieg)$

If $G$ is a reductive group, then the map $\Theta:G\lra G$ defined by
\begin{equation}\label{global-cartan}
\Theta(h\exp A)=h\exp(-A)\;\;\;\mbox{for}\;\;\; h\in H\;\;\mbox{and}\;\;
A\in \liem
\end{equation}
is an automorphism of $G$ and its differential is $\theta$. This is called
the {\bf global Cartan involution}.

\subsection{$\liesl_2$-triples and orbit theory}\label{triples}

We consider a reductive group $G$ as defined in Appendix \ref{reductive}.

An ordered triple of elements $(x,e,f)$ in $\lieg$ (or
$\lieg^\C$) is called a $\liesl_2$-{\bf triple} if the bracket
relations $[x,e]=2e$, $[x,f]=-2f$, and $[e,f]=x$ are satisfied.
One has that the elements $e$ and $f$ are nilpotent.  An
$\liesl_2$-triple $(x, e, f)$ in $\lieg^\CC$ is called {\bf
normal} if $e,f\in \liem^\CC$ and $x\in\lieh^\CC$. Some times
we refer to a normal triple as a Kostant--Rallis triple (see
\cite{KR}).

Let $\NNN(\lieg)$ and $\NNN(\liem^\CC)$ be the set of nilpotent
elements in $\lieg$ and $\liem^\CC$ respectively. One has the
following.

\begin{proposition}\label{nilpotent-orbits}

(1) Every element $0\neq e\in \NNN(\lieg)$ can be embedded in a
$\liesl_2$-triple $(x,e,f)$, establishing a 1--1 correspondence
between the set of all $G$-orbits in $\NNN(\lieg)$ and the set
of all $G$-conjugacy classes of $\liesl_2$-triples in $\lieg$.

(2) Every element $0\neq e\in \NNN(\liem^\CC)$ can be embedded
in a normal $\liesl_2$-triple $(x,e,f)$, establishing a 1--1
correspondence   between the set of all $H^\CC$-orbits in
$\NNN(\liem^\CC)$ and the set of all $H^\CC$-conjugacy classes
of normal $\liesl2$-triples in $\lieg^\CC$.
\end{proposition}

Statement (1) is a real version of a refinement of the
Jacobson--Morozov theorem, proved in \cite{Kostant}. For (2)
see \cite[Prop 4, 38]{KR}

We say that an  $\liesl_2$-triple in $\lieg$ is a
Kostant--Sekiguchi triple if $\theta(e)=-f$, and hence $\theta(x)=-x$.
A normal $\liesl_2$-triple in $\lieg^\C$ is called
Kostant--Sekiguchi triple if $f=\sigma(e)$
where  $\sigma$ is  the conjugation  of $\lieg^\C$
defining $\lieg$.

One has the following (see \cite{sekiguchi}).

\begin{proposition}\label{KR-KS}
(1) Every ${\liesl}_2$-triple in $\lieg$  is conjugate under
$G$ to a Kostant--Sekiguchi triple in $\lieg$. Two
Kostant--Sekiguchi triples are $G$-conjugate
if and only if they are conjugate under $H$.

(2) Every normal 
triple in $\lieg^\C$ is conjugate under $H^\C$ to a
Kostant--Sekiguchi triple. Two Kostant--Sekiguchi triples are
conjugate under $H^\C$  to the same  Kostant--Rallis triple  if
and only if they are conjugate under $H$.
\end{proposition}

Propositions \ref{nilpotent-orbits} and \ref{KR-KS} can be
combined, together with a linear transformation sometimes
called Cayley transform (see e.g. \cite[p. 579]{DJ})  to obtain
the Kostant--Sekiguchi correspondence:

\begin{proposition} There is a one-to-one correspondence
(see \cite{sekiguchi,Ver95}).
$$
\NNN(\lieg)/G \longleftrightarrow \NNN(\liem^\C)/H^\C.
$$
\end{proposition}

A similar correspondence can be established for orbits of arbitrary
elements (see \cite{bielawski,biquard-note}).

\subsection{Conjugacy classes of a real reductive Lie group}

Now let $(G,H,\theta,B)$ be a reductive group as defined in
Appendix \ref{reductive}. We can also give in  this case a
description of conjugacy classes of $G$.
A {\bf Cartan subgroup} of $G$ is defined as the centralizer in
$G$ of a Cartan subalgebra of $\lieg$. When $G$ is non-compact,
it is no longer true that every element of $G$ lies in a Cartan
subgroup. For example, for $G=\SL_2\RR$, the element
$\begin{pmatrix} 1 & 1 \\ 0 & 1
\end{pmatrix}$ does not lie in any Cartan subgroup (see
\cite{knapp}, p. 487).

We know (see \cite{knapp}) that any Cartan subalgebra is
conjugate via $\Int\lieg$ to a $\theta$-invariant Cartan
subalgebra, and that there are only finitely many conjugacy
classes of Cartan subalgebras. Consequently any Cartan subgroup
of $G$ is conjugate via $G$ to a $\Theta$-stable Cartan
subgroup, where $\Theta$ is given in (\ref{global-cartan}), and
there are only finitely many conjugacy classes of Cartan
subgroups. Moreover, a $\Theta$-invariant Cartan subgroup  is
reductive.

Cartan subgroups of a  non-compact real reductive Lie group can be
non-connected even if the group $G$ is connected. This already happens
for $\SL_2\RR$ (see \cite{knapp}). In the following we will see how alcoves
can be useful to deal with elements of a
Cartan subgroup  that are not in
the identity component.

\begin{proposition}
Let $T'$ be a $\Theta$-invariant Cartan subgroup. Every  element $t\in T'$
is conjugate to an element of the form  $\exp(\alpha)\exp(s)$ where
$\alpha\in \bar \eA$ and  $s\in\liem$, such that
$\exp(\alpha)$ and $\exp(s)$ commute.

\end{proposition}
\begin{proof}

Let $t\in T'$. Then by the reductivity and $\Theta$-invariance of $T'$ we have
that $t=h\exp(s')$ with $h\in T'\cap H$ and  $s'\in \liet'\cap\liem$.
We can extend $\liet'$ to a Cartan subalgebra
$\liet$ of $\lieh$ with corresponding Cartan subgroup $T$, and choose
an alcove $\eA$ containing $0$ such that $h$ is conjugate to (via an
element of the Weyl  group of $T$) to an element of the form $\exp(\alpha)$
with $\alpha\in \bar \eA$. So $h=k^{-1}\exp(\alpha)k$ with $k\in H$.
Then $ktk^{-1}=\exp(\alpha)k\exp(s')k^{-1}$. The element $k\exp(s')k^{-1}$ is
in $\exp(\liem)$, and can thus be written $k\exp(s')k^{-1}=\exp(s)$ with
$s\in\liem$.
Clearly, since $h$ and $\exp(s')$  (both are elements in $T'$) we have that
$\exp(\alpha)$ and $\exp(s)$ commute.

\end{proof}
For the correspondences proved in this paper
we will be considering conjugacy classes of $G$ in which one can find a representative of the form
$$
g=g_eg_hg_u
$$
so that all the factors commute two by two,  with
\begin{itemize}

\item
$g_e=\exp(2\imag \pi \alpha)$, with $\alpha\in \imag \bar\eA$ ({\bf elliptic}
  element)

\item
$g_h=\exp(s-\tau(s))$ with $s\in \liea$ ({\bf hyperbolic} element)

\item $g_u=\exp (n)$, where $n\in \liem$ is a nilpotent element
of the form $i(Y-X-H)$ where $(H,X,Y)$ is an appropriate $\liesl_2$-triple
 ({\bf unipotent} element).

\end{itemize}

The decomposition $g=g_eg_hg_u$ is the {\bf multiplicative
Jordan decomposition}, which is proved in Helgason
\cite{helgason} for $\GL_n\R$ and  in Eberlein \cite{eberlein}
for the connected component of the isometry group of a
symmetric space of non-compact type, hence for the connected
component of the adjoint group.

\section{Parabolic subgroups and relative degree}
\label{sec:parabolic-subgroups}

In this section we recall some basics on parabolic subgroups and define
precisely the relative degree of two parabolic subgroups.

\subsection{Parabolic subgroups}
\label{sec:parabolic-subgroups-1}

Let $\Sigma=H\backslash G$ a symmetric space of non compact
type. (The action is taken to be a right action, because this
fits better with the way symmetric spaces arise in K\"ahler
quotients.) Let $\glie=\hlie \oplus \mlie$ be the Cartan
decomposition, and $\alie \subset \mlie$ a maximal abelian
subalgebra (the dimension of $\alie$ is the rank of $\Sigma$).
Let $\Phi\subset \alie^*$ the roots of $\glie$, so that
$\glie=\glie_0\oplus\oplus_{\lambda\in \Phi}\glie_\lambda$,
where $\glie_0$ is the centralizer of $\alie$.

Choose a positive Weyl chamber $\alie^+ \subset \alie$, and let $\Phi^\pm \subset\Phi$
(resp. $\Delta\subset\Phi$) be the set of positive/negative roots (resp. simple roots).  For
$I\subset \Delta$, denote $\Phi^I\subset \Phi$ the set of roots which are linear combinations
of elements of $I$, then we define a standard parabolic subalgebra
$$ \plie_I = \glie_0 \oplus \sum_{\lambda\in \Phi^I\cup \Phi^-}\glie_\lambda, $$
and $P_I\subset G$ the corresponding subgroup. Any parabolic subalgebra of
$\glie$ is conjugate to one of the standard parabolic subalgebras.

Given an element $s\in \alie^+$, one picks a standard parabolic subgroup
$P_s$ in the following way. When $t\to+\infty$, the geodesic $t\mapsto* e^{ts}$ in $\Sigma$
(where $*$ is the base point, fixed by $H$), hits the visual boundary
$\partial_\infty\Sigma$ in a point, whose stabilizer in $G$ is the parabolic group
$P_s$. It follows that $P_s$ and its Lie algebra $\plie_s$ are defined by
\begin{align*}
  P_s &= \{ g\in G, e^{ts}ge^{-ts} \text{ is bounded as }t\to\infty \},\\
  \plie_s &= \{ x\in \glie, \Ad(e^{ts})x \text{ is bounded as }t\to\infty \}.
\end{align*}
It is immediate that
$$ \plie_s = \sum_{\lambda\in R, \; \lambda(s)\leq 0} \glie_\lambda = \plie_I $$
for $I=\{\lambda\in \Delta, \lambda(s)= 0\}$. The element $s$ also defines a Levi subgroup
$L_s\subset P_s$ and a Levi subalgebra $\llie_s\subset \plie_s$ by
$$
  L_s = \{ g\in G, \Ad(g)(s)=s \},\quad \llie_s = \{ x\in \glie, [s,x]=0 \}.
$$

A character $\chi$ of $\plie_s$ is given by an element in the dual of the Lie algebra of the centre of $L_s$. Via the invariant metric, this element defines an
element $s_\chi$ in the Lie algebra of the centre of $L_s$ and hence in $i\hlie$.When $\plie_s\subset  \plie_{s_\chi}$ we say that $\chi$ is {\bf antidominant}. When $\plie_s =  \plie_{s_\chi}$ we say that  $\chi$ is {\bf strictly antidominant}
The mapping $\chi_s:\plie_s\to\setR$ defined by
$$ \chi_s(x) = \langle s,x\rangle $$
is hence a strictly antidominant character of $\plie_s$.

\subsection{Relative degree}
\label{sec:relative-degree}

We now define a function which is important in the paper, since it
calculates the contribution to the parabolic degree at the
punctures. The setting is the following.

Let $\cO_H\subset \mlie$ be an $H$-orbit in $\mlie$. As is well
known, $\cO_H$ is also a $G$-homogeneous space. This can be
seen in the following way: given $s\in \cO_H$, one can consider
$\eta(s)=\lim_{t\to+\infty}*e^{ts}\in \partial_\infty\Sigma$.
It turns out that the image of $\cO_H$ under $\eta$ is a
$G$-orbit in $\partial_\infty\Sigma$. Of course the stabilizer
of $\eta(s)$ is the parabolic group $P_s$ defined above, so one
gets an identification
$$ \eta : \cO_H=H/(P_s\cap H) \subset \mlie \longrightarrow P_s\backslash G \subset \partial_\infty\Sigma. $$
The action of $g\in G$ on $\cO_H$ will be denoted by $s \cdot g$; if one
decomposes $g=ph$ with $h\in H$ and $p\in P$, then $s\cdot g=s\cdot h=\Ad(h^{-1})s$.

Now take another element $\sigma\in \mlie$. As we shall see below in the
proof of the proposition, the function
$$ t \mapsto \langle s\cdot e^{-t\sigma}, \sigma\rangle $$
is a nonincreasing function of $t$, so we can define a function
$$ \mu_s:\mlie \longrightarrow \setR, \quad \mu_s(\sigma) = \lim_{t\to+\infty} \langle s\cdot e^{-t\sigma},\sigma\rangle.$$
This function is actually (up to a normalization) a function defined
on $\partial_\infty\Sigma$, as follows from the following proposition.

\begin{proposition}
  Suppose $s$ and $\sigma$ normalized so that $|s| = |\sigma| = 1$. Then
  one has $$\mu_s(\sigma) = \cos \angle_{\mathrm{Tits}}(\eta(\sigma),\eta(s)),$$ where
  $\angle_{\mathrm{Tits}}$ is the Tits distance on $\partial_\infty\Sigma$. In particular,
  one has the reciprocity $\mu_\sigma(s)=\mu_s(\sigma)$.
\end{proposition}
\begin{proof}
  Decompose $e^{-t\sigma}=p_th_t$ with $h_t\in H$ and $p_t\in P_s$. Then
  $$ \langle s\cdot e^{-t\sigma}, \sigma\rangle=\langle\Ad(h_t^{-1})s,\sigma\rangle=\cos \angle(\Ad(h_t^{-1})s,\sigma). $$
  On the other hand, the distance
  $$ d(*e^{u\Ad(h_t^{-1})s}e^{t\sigma},*e^{us})=d(*e^{us}p_t^{-1},*e^{us})$$
  is bounded when $u\to+\infty$, so $u\to *e^{u\Ad(h_t)s}e^{t\sigma}$ is the geodesic
  emanating from $*e^{t\sigma}$ and going to $\eta(s)$. So the angle between
  the geodesics emanating from $*e^{t\sigma}$ and converging to $\eta(\sigma)$ and
  $\eta(s)$ is the angle between $\Ad(h_t^{-1})s$ and $\sigma$. It is
  well known that this angle is increasing and converges when $t\to\infty$ to
  the Tits distance between $\eta(s)$ and $\eta(\sigma)$, and the proposition
  follows.
\end{proof}

The function $-\mu_\sigma$ is the `asymptotic slope' of \cite{KLM}. In
\cite{M3}, the complex case is studied: if $G=H^\setC$ then the adjoint
orbit $\cO_H$ is a K\"ahler manifold, and the asymptotic slope can be
reinterpreted in terms of maximal weights of the action of $G$ on
$\cO_H$.

We will use this notion to define a relative degree. From the proposition,
$\mu_s(\sigma)$ depends only on giving two pairs $(P,s)$ and $(Q,\sigma)$ of a
parabolic subgroup of $G$ and an antidominant character on the parabolic
subgroup. So we can define the {\bf relative degree} of $(P,s)$ and
$(Q,\sigma)$ as
\begin{equation}
  \label{eq:32}
  \deg\big( (P,s), (Q,\sigma) \big) = \mu_s(\sigma) .
\end{equation}Observe, again from the proposition, that $\deg$ is a symmetric
function of its two arguments.


\begin{thebibliography}{Kro90b}

\bibitem{atiyah-bott}
M.~F. Atiyah and R.~Bott, \emph{The {Y}ang-{M}ills equations over {R}iemann
surfaces}, Philos. Trans. Roy. Soc. London Ser. A \textbf{308} (1982),
523--615.


\bibitem{balaji-biswas-pandey} V. Balaji, I. Biswas,  and Y. Pandey,
\emph{Connections on  parahoric torsors over curves}, Publ.RIMS, Kyoto Univ
{\bf 53} (2017)  551--585.


\bibitem{balaji-seshadri} V. Balaji and C.S. Seshadri,
\emph{Moduli of parahoric $\mathcal G$-torsors on a compact Riemann surface},
Journal of Algebraic Geometry, \textbf{24} (2015), 1--49.

\bibitem{bhosle-ramanathan}
U. Bhosle and A. Ramanathan, \emph{Moduli of Parabolic $G$-bundles on Curves},
 Mathematische Zeitschrift {\bf 202} (1989), 161--180.

\bibitem{bielawski}
R. Bielawski, \emph{Lie groups, Nahm's equations and
hyper-K\"ahler   manifolds}, in: Tschinkel, Yuri (ed.),
Algebraic groups. Proceedings of the
summer school, G\"ottingen, June 27-July 13, 2005.
G\"ottingen: Universit\"atsverlag G\"ottingen.
Universit\"atsdrucke G\"ottingen. Seminare Mathematisches
Institut, 1--17 (2007).

\bibitem{biquard-thesis} O. Biquard,
\emph{Fibr\'es paraboliques stables et connexions singuli\`eres plates}, Bull.
Soc. Math. France {\bf 119} (1991) 231--257.

\bibitem{Biq96b}
\bysame, \emph{Sur les \'equations de {N}ahm et la structure de
  {P}oisson des alg\`ebres de {L}ie semi-simples complexes}, Math. Ann.
  \textbf{304} (1996),  253--276.

\bibitem{Biq97}
\bysame, \emph{Fibr\'es de {H}iggs et connexions int\'egrables: le cas
  logarithmique (diviseur lisse)}, Ann. Sci. \'Ecole Norm. Sup. (4) \textbf{30}
  (1997),  41--96.

\bibitem{biquard-note}
\bysame,  \emph{Extended correspondence of Kostant--Sekiguchi--Vergne},
unpublished, available at \texttt{http://www.math.ens.fr/\textasciitilde{}biquard/eksv2.pdf}.


\bibitem{crm-notes}
O. Biquard, O. Garc\'{\i}a-Prada and I. Mundet i Riera, \emph{An Introduction
to Higgs Bundles}, Lecture Notes of the Second International School on
Geometry and Physics
``Geometric Langlands and Gauge Theory''
Centre de Recerca Matem\`atica,
Bellaterra (Spain),
17--26 March 2010.
Available on
http://citeseerx.ist.psu.edu/viewdoc/download?doi=10.1.1.179.513\&rep=rep1\&type=pdf

\bibitem{BGR}
O. Biquard, O. Garc\'{\i}a-Prada and R. Rubio,
\emph{Higgs bundles, the Toledo invariant and the Cayley correspondence},
Journal of Topology 10 (2017) 795--826.

\bibitem{biswas-gastesi-govindarajan}
I. Biswas, P.A. Gastesi, S. Govindarajan,
\emph{Parabolic Higgs bundles and Teichm\"uller spaces for punctured
surfaces}, Transactions of the AMS \textbf{349} (1997) 1551--1560.

\bibitem{boalch}
P. Boalch, \emph{Riemann--Hilbert for tame complex parahoric connections},
Transform. Groups  \textbf{16}  (2011),   27--50.

\bibitem{boden-hu}
H.U. Boden and Y. Hu, \emph{Variations of moduli of parabolic bundles},
Math. Ann. \textbf{301} (1995) 539--559.

\bibitem{boden-yokogawa} H.U. Boden and K. Yokogawa, \emph{Moduli spaces of
parabolic Higgs bundles and parabolic $K(D)$ pairs over smooth
curves. I}, Internat.\ J. Math. {\bf 7}  (1996) 573--598.

\bibitem{bottacin}
F.  Bottacin, \emph{Symplectic geometry on moduli spaces of stable pairs}, Ann.
 Sci. \'Ecole
Norm. Sup. {\bf 28} (1995) 391--433.


\bibitem{BGG03}
S.B. Bradlow, O. Garc{\'{\i}}a-Prada, and P.B. Gothen.
\emph{Surface group representations and {${\rm U}(p,q)$}-{H}iggs bundles},
 J. Differential Geom., {\bf 64} (2003), 111--170.

\bibitem{BGG06}
\bysame
\emph{Maximal surface group representations in isometry groups of classical
 {H}ermitian symmetric spaces} Geom. Dedicata, {\bf 122} (2006), 185--213.


\bibitem{bradlow-garcia-prada-mundet:2003}
S.~B. Bradlow, O.~Garc{\'{\i}}a-Prada, and I.~Mundet~i Riera,
\emph{Relative
  {H}itchin-{K}obayashi correspondences for principal pairs}, Quart. J. Math.
  \textbf{54} (2003), 171--208.

\bibitem{BD} T. Br\"ocker and T. tom Dieck,
\emph{Representations of Compact Lie Groups}, GTM 98, Springer 1985.

\bibitem{biw}
M. Burger, A. Iozzi and A. Wienhard,
\emph{Surface group representations with maximal Toledo invariant},
Ann. Math. (2) \textbf{172} (2010),  517--566.

\bibitem{CJY} T.C. Collins, A. Jacob and S.-T. Yau,
\emph{Poisson metrics on flat vector bundles over non-compact curves},
Comm. Anal. Geom., to appear.


\bibitem{corlette}
K.~Corlette, \emph{Flat ${G}$-bundles with canonical metrics}, J. Differential
  Geom., \textbf{28} (1988), 361--382.

\bibitem{DJ} D.\v{Z}. {\DJ}okovi\'c, \emph{Proof of a
    Conjecture
    of Kostant}, Trans. Amer. Math. Soc. {\bf 302} (1987),
    577--585.

\bibitem{donaldson-ns}
S.K. Donaldson, \emph{A new proof of a theorem of Narasimhan and Seshadri},
 J. Differential Geom. {\bf 18} (1983), 269--277.

\bibitem{donaldson-twisted}
S.K. Donaldson, \emph{Twisted harmonic maps and the self-duality equations},
Proc. London Math. Soc. (3) \textbf{55} (1987), 127--131.

\bibitem{eberlein}P.B. Eberlein, {\em Geometry of Nonpositively Curved
Manifolds}, Chicago Lectures in Mathematics, 1996.

\bibitem{garcia-prada}
O.~Garc{\'\i}a-Prada,
 {\em Higgs bundles and higher Teichm\"uller spaces}, Handbook on Teichm\"uller theory, to appear.

\bibitem{garcia-prada-gothen-mundet:2009a}
O.~Garc{\'\i}a-Prada, P.~B. Gothen, and I.~Mundet i~Riera,
 { \em The {H}itchin-{K}obayashi correspondence, {H}iggs pairs and surface
 group representations}, version 3: 2012, \texttt{arXiv:0909.4487}.

\bibitem{GGM} \bysame
\emph{Representations of surface groups in real symplectic
groups}, Journal of Topology, {\bf 6} (2013), 64--118.

\bibitem{GGMu} O. Garc\'{\i}a-Prada, P.B. Gothen, and V.  Mu\~noz,
{\em Betti numbers of the moduli space of rank 3  parabolic Higgs bundles},
Memoirs AMS, {\bf 187, No 879} (2007).

\bibitem{GLM}
O.~Garc{\'\i}a-Prada, M. Logares, V. Mu\~{n}oz,
{\em Moduli spaces of $\U(p,q)$-Higgs bundles}, Q. J. Math., {\bf 60} (2009)
183--233.

\bibitem{garcia-prada-oliveira}
O.~Garc{\'\i}a-Prada, and A. Oliveira,
\emph{Connectedness of Higgs bundle moduli for complex reductive Lie groups},
Asian Journal of Mathematics, {\bf 21} (2017) 791--810.


\bibitem{helgason}
S.~Helgason, \emph{Differential geometry, {L}ie groups, and symmetric spaces},
  Mathematics, vol.~80, Academic Press, San Diego, 1998.

\bibitem{heinloth}
J. Heinloth, {\em Uniformization of $G$-bundles}, Math. Ann., {\bf 347},
(2010), 499--528.

\bibitem{hitchin}
N.~J. Hitchin, \emph{The self-duality equations on a {R}iemann surface},
Proc. London Math. Soc., \textbf{55} (1987) 59--126.

\bibitem{hitchin-duke}
 \bysame  \emph{Stable bundles and integrable systems}, Duke Math. J. 54
(1987),  91--114.

\bibitem{hitchin-teichmueller}
\bysame \emph{Lie groups and Teichm\"uller space}, Topology {\bf 31} (1992),
449--473.

\bibitem{knapp}
A. W. Knapp, \emph{Lie Groups beyond an Introduction}, second ed.,
Progress in Mathematics, vol 140, Birkh\"auser Boston Inc., Boston,
MA, 1996.


\bibitem{KLM} M. Kapovich, B. Leeb and J. Millson;
\emph{Convex functions on symmetric spaces, side lengths of polygons and
the stability inequalities for weighted configurations at infinity},
J. Differential Geom. \textbf{81} (2009),  297--354.

\bibitem{kobayashi}
S. Kobayashi, {\em Differential Geometry of Complex Vector Bundles}, Princeton University Press, New Jersey, 1987.


\bibitem{konno} H. Konno, \emph{Construction of the moduli space of stable
parabolic Higgs bundles on a Riemann surface}, J. Math. Soc. Japan \textbf{45}
(1993) 253--276.

\bibitem{Kostant} B. Kostant, \emph{The principal three
    dimensional subgroup and the Betti numbers ofa complex
    simple Lie group}, Amer. J. Maht. \textbf{81} (1959),
    367--371.

\bibitem{KR}
B. Kostant and S. Rallis, \emph{Orbits and representations associated with
symmetric spaces}. Amer. J. Math.  \textbf{93}  (1971), 753--809.

\bibitem{Kro90b}
P.~B. Kronheimer, \emph{{A hyper-K\"ahlerian structure on coadjoint orbits of a
  semisimple complex group}}, J. Lond. Math. Soc. \textbf{42} (1990), 193--208.



\bibitem{Kro90a}
\bysame, \emph{Instantons and the geometry of the nilpotent variety}, J.
  Differ. Geom. \textbf{32} (1990), 473--490.


\bibitem{Log06}
M. Logares, {\em Betti Numbers of parabolic $\U(2,1)$-Higgs bundles moduli
spaces},
Geom. Dedicata {\bf 123} (2006) 187--200.

\bibitem{markman}
E. Markman, {\em Spectral curves and integrable systems}, Compositio Math.
{\bf 93} (1994) 255--290.

\bibitem{mehta-seshadri}
V. Mehta and C.S. Seshadri, {\em Moduli of vector bundles on curves with
parabolic structures}, Math. Ann, {\bf 248}, (1980), 205--239.

\bibitem{mundet:2000}
I.~Mundet i Riera, \emph{A Hitchin--Kobayashi correspondence for
K\"ahler fibrations}, J. Reine Angew. Math. {\bf 528} (2000),
41--80.

\bibitem{M2} \bysame, \emph{A Hilbert-Mumford criterion for
    polystability in K\"ahler geometry}, Trans. Am. Math. Soc. \textbf{362}
  (2010),  5169-5187.

\bibitem{M3} \bysame, \emph{Maximal weights in K\"ahler
  Geometry: Flag manifolds and Tits distance} (with an appendix by
  A. H. W. Schmitt), in: O. Garc\'\i a-Prada et. al. (eds.), Vector bundles
  and complex geometry. Conference on vector bundles in honor of S. Ramanan
  on the occasion of his 70th birthday, Madrid, Spain, June 16--20,
  2008. Providence, RI: American Mathematical Society (AMS). Contemporary
  Mathematics \textbf{522} (2010), 113--129.

\bibitem{ignasi-hitchin70} \bysame, \emph{Parabolic Higgs bundles for real reductive
Lie groups: A very basic introduction}, Geometry and
Physics: A Festschrift in honour of Nigel Hitchin, Oxford University Press,
2018.

\bibitem{MT} I. Mundet i Riera, G. Tian, \emph{A compactification of
the moduli space of twisted holomorphic maps}, Adv. Math. {\bf 222} (2009),  1117-1196.

\bibitem{nakajima} H. Nakajima, Hyper-K\"ahler structures on moduli spaces of
  parabolic Higgs bundles on Riemann surfaces. Moduli of vector bundles
(Sanda, 1994; Kyoto, 1994), 199--208, Lecture Notes in Pure and Appl. Math.
179, Dekker, 1996.

\bibitem{NS}
M.S. Narasimhan, C.S. Seshadri,
\emph{Stable and unitary vector bundles on a compact Riemann surface},
Ann. Math. (2) \textbf{82} (1965), 540-567.

\bibitem{nasatyr-steer} B. Nasatyr and B. Steer, \emph{Orbifold Riemann surfaces
and the Yang--Mills--Higgs equations},  Ann.\ Scuola Norm.\
Sup.\ Pisa Cl.\ Sci.\ {\bf 22} (1995) 595--643.

\bibitem{pappas-rapoport}
G. Pappas, M. Rapoport. Some questions about ${\mathcal G}$-bundles on curves. Algebraic and arithmetic
structures of moduli spaces (Sapporo 2007), 159--171, Adv. Stud. Pure Math., 58, Math.
Soc. Japan, Tokyo, 2010.




\bibitem{ramanathan}
A. Ramanathan, \emph{Stable principal bundles on a compact Riemann surface},
Math. Ann., {\bf 213}, (1975), 129--152.


\bibitem{schmitt}
A. H. W. Schmitt, \textsl{Geometric Invariant Theory and
Decorated Principal Bundles}, Zurich Lectures in Advanced Mathematics,
European Mathematical Society, 2008.

\bibitem{seshadri}
C.S. Seshadri,  \emph{Moduli of vector bundles on curves with parabolic
structures}, Bull. Amer. Math. Soc. {\bf 83} (1977), 124--126.

\bibitem{sekiguchi} J. Sekiguchi, \emph{Remarks on real nilpotent
    orbits of a symmetric pair}, J. Math. Soc. Japan {\bf 39} (1987),
   127--138.

\bibitem{Sim88}
C.T. Simpson, \emph{Constructing variations of {H}odge structure using
  {Y}ang-{M}ills theory and applications to uniformization}, J. Amer. Math.
  Soc. \textbf{1} (1988),  867--918.

\bibitem{Sim90}
\bysame, \emph{Harmonic bundles on noncompact curves}, J. Amer. Math. Soc.
  \textbf{3} (1990),  713--770.

\bibitem{Sim92}
\bysame, \emph{Higgs bundles and local systems}, Inst. Hautes \'Etudes Sci.
  Publ. Math. (1992),  5--95.

\bibitem{singer}
I. M. Singer, \emph{The geometric interpretation of a special connection},
Pacific J. Math. \textbf{9} (1959) 585--590.

\bibitem{teleman-woodward}
C. Teleman and C. Woodward,
\emph{Parabolic bundles, products of conjugacy classes, and
quantum cohomology}, Annales de L'Institut Fourier {\bf 3}
(2003), 713--748.

\bibitem{uhlenbeck-yau}
K. Uhlenbeck, S.T. Yau, \emph{On the existence of Hermitian--Yang--Mills connections in stable vector bundles.} Comm. Pure Appl.
Math. \textbf{39-S} (1986), 257--293.

\bibitem{Ver95}
M. Vergne, \emph{Instantons et correspondance de {K}ostant-{S}ekiguchi},
  C. R. Acad. Sci. Paris S\'er. I Math. \textbf{320} (1995),  901--906.

\bibitem{yokogawa} K. Yokogawa, \emph{Compactification of moduli of parabolic
  sheaves and moduli of parabolic Higgs sheaves}, J. Math. Kyoto Univ.
\textbf{33} (1993) 451--504.


\end{thebibliography}
\end{document}